\documentclass{article}
\usepackage[scale=0.8]{geometry}

\usepackage{amsmath}
\usepackage{amssymb}
\usepackage{amsthm}
\usepackage{dsfont}
\usepackage{amsfonts}
\usepackage{cases}
\usepackage{url}
\theoremstyle{plain}
\newtheorem{theor0}{Theorem}
\newenvironment{theor}
  {\pushQED{\qed}\begin{theor0}}
  {\popQED\end{theor0}}
\newtheorem{lem0}{Lemma}[section]
\newenvironment{lem}
  {\pushQED{\qed}\begin{lem0}}
  {\popQED\end{lem0}}
\newtheorem{prop0}[lem0]{Proposition}
\newenvironment{prop}
  {\pushQED{\qed}\begin{prop0}}
  {\popQED\end{prop0}}
\newtheorem{cor0}[lem0]{Corollary}
\newenvironment{cor}
  {\pushQED{\qed}\begin{cor0}}
  {\popQED\end{cor0}}
\newtheorem{propr0}[lem0]{Property}

\newtheorem{hyp0}[lem0]{Hypothesis}

\newtheorem{result0}[lem0]{Result}

\newtheorem{conj0}[lem0]{Conjecture}

\newtheorem{heur0}[lem0]{Heuristics}

\theoremstyle{definition}
\newtheorem{defin0}[lem0]{Definition}
\newenvironment{defin}
  {\pushQED{\qed}\begin{defin0}}
  {\popQED\end{defin0}}
\newtheorem{rems0}[lem0]{Remarks}
\newenvironment{rems}
  {\pushQED{\qed}\begin{rems0}}
  {\popQED\end{rems0}}
\newtheorem{ex0}[lem0]{Example}

\newtheorem{exs0}[lem0]{Examples}

\newtheorem{rem0}[lem0]{Remark}
\newenvironment{rem}
  {\pushQED{\qed}\begin{rem0}}
  {\popQED\end{rem0}}
\newtheorem{qu0}[lem0]{Question}

\newtheorem{qus0}[lem0]{Questions}

  \newtheorem{as0}[lem0]{Assumption}

\usepackage[applemac]{inputenc}
\usepackage[english,frenchb]{babel}
\usepackage[T1]{fontenc}% always use this.
\usepackage{csquotes}% recommended by biblatex when using french.
\usepackage[toc,page]{appendix}
\usepackage{textcomp}
\usepackage{amscd}
\usepackage{graphicx}
\usepackage{enumerate}
\usepackage{cancel}
\usepackage{extpfeil}
\usepackage[all]{xy}
\usepackage{empheq}
\usepackage{esint}
\usepackage{mathtools}
\usepackage{float}
\usepackage{wrapfig}

\newcommand{\e}{\epsilon}

\newcommand{\R}{\mathbb R}

\newcommand{\Cc}{\mathcal C}

\newcommand{\Pc}{\mathcal P}

\newcommand{\M}{\mathcal M}

\newcommand{\Id}{\operatorname{Id}}

\newcommand{\Mes}{\operatorname{Mes}}

\newcommand{\Div}{\operatorname{div}}

\newcommand{\loc}{{\operatorname{loc}}}
\newcommand{\uloc}{{\operatorname{uloc}}}

\newcommand{\curl}{{\,\operatorname{curl}\,}}

\newcommand{\Ld}{\operatorname{L}}

\usepackage[toc,page]{appendix}

\newcommand{\cvf}[1]{\mathrel{\mathop{\xrightharpoonup{#1}}}}

\newcommand{\step}[1]{\noindent \textit{Step} #1.}
\newcommand{\substep}[1]{\noindent \textit{Substep} #1.}
\numberwithin{equation}{section}

\usepackage{color}
\definecolor{dgreen}{rgb}{0, 0.6, 0}

\usepackage[colorlinks,citecolor=black,urlcolor=black]{hyperref}

\title{Well-posedness for mean-field evolutions arising in superconductivity}
\author{Mitia Duerinckx\\\\with an appendix jointly written with Julian Fischer}
\date{}

\begin{document}
\selectlanguage{english}
\maketitle

We establish the existence of a global solution for a new family of fluid-like equations, which are obtained in certain regimes in~\cite{DS-16} as the mean-field evolution of the supercurrent density in a (2D section of a) type-II superconductor with pinning and with imposed electric current. We also consider general vortex-sheet initial data, and investigate the uniqueness and regularity properties of the solution. For some choice of parameters, the equation under investigation coincides with the so-called lake equation from 2D shallow water fluid dynamics, and our analysis then leads to a new existence result for rough initial data.

\bigskip

\section{Introduction}\label{chap:intro}

\subsection{General overview}

We study the well-posedness of the following two fluid-like evolution models coming from the mean-field limit equations of Ginzburg-Landau vortices: first, for $\alpha\ge0$, $\beta\in\R$, we consider the ``incompressible'' flow
\begin{align}\label{eq:limeqn1}
\partial_tv=\nabla P-\alpha(\Psi+v) \curl v+\beta(\Psi+v)^\bot \curl v,\qquad\Div (av)=0,\qquad\text{in $\R^+\times\R^2$,}
\end{align}
and second, for $0\le \lambda<\infty$, $\alpha>0$, $\beta\in\R$, we consider the ``compressible'' flow
\begin{align}\label{eq:limeqn2}
\partial_tv=\lambda\nabla(a^{-1}\Div(av))-\alpha(\Psi+v) \curl v+\beta(\Psi+v)^\bot \curl v,\qquad\text{in $\R^+\times\R^2$,}
\end{align}
with $v:\R^+\times\R^2:=[0,\infty)\times\R^2\to\R^2$ and $\curl v\ge0$, where $\Psi:\R^2\to\R^2$ is a given forcing vector field, and where the weight $a:=e^h$ is determined by a given ``pinning potential'' $h:\R^2\to\R$.
Note that the incompressible model~\eqref{eq:limeqn1} can be seen as the limiting case $\lambda=\infty$ of the family of compressible models~\eqref{eq:limeqn2}. As established in a companion paper~\cite{DS-16} with Sylvia Serfaty, these equations are obtained in certain regimes as the mean-field evolution of the supercurrent density in a (2D section of a) type-II superconductor described by the 2D Ginzburg-Landau equation with pinning and with imposed electric current --- but without gauge and in whole space, for simplicity.
In this context, the cases $\lambda=\infty$, $0<\lambda<\infty$, and $\lambda=0$ correspond respectively to a low, an intermediate, and a high vortex density regime.
Note that in the parabolic case $\alpha>0$, $\beta=0$, the incompressible model~\eqref{eq:limeqn1} can be seen as a Wasserstein gradient flow for the vorticity $\curl v$, but a common gradient flow structure seems to be missing for the whole family of equations~\eqref{eq:limeqn2} with $\lambda\in[0,\infty]$.
In the conservative case $\alpha=0$ with $\Psi=0$, the incompressible model~\eqref{eq:limeqn1} takes the form of the so-called lake equation from 2D shallow water fluid dynamics~\cite[p.235]{Greenspan-80} (see also~\cite{Camassa-Holm-Levermore-1,Camassa-Holm-Levermore-2}), which reduces to the usual 2D Euler equation if the weight $a$ is constant.

In the nondegenerate case $\lambda>0$, we investigate existence, uniqueness, and regularity, both locally and globally in time, for the Cauchy problems associated with~\eqref{eq:limeqn1} and~\eqref{eq:limeqn2}, and we also consider vortex-sheet initial data. In Appendix~\ref{sec:degenerate} jointly written with Julian Fischer, a complete theory is further obtained for the degenerate parabolic case $\lambda=\beta=0$, $\alpha>0$.

\subsection{Brief discussion of the model}
Superconductors are materials that in certain circumstances lose their resistivity, which allows permanent supercurrents to circulate without loss of energy. In the case of type-II superconductors, if the external magnetic field is not too strong, it is expelled from the material (Meissner effect), while, if it is much too strong, the material returns to a normal state. Between these two critical values of the external field, these materials are in a mixed state, allowing a partial penetration of the external field through ``vortices'', which are accurately described by the (mesoscopic) Ginzburg-Landau theory.
Restricting ourselves to a 2D section of a superconducting material, it is standard to study for simplicity the 2D Ginzburg-Landau equation on the whole plane (to avoid boundary issues) and without gauge (although the gauge is expected to bring only minor difficulties). We refer e.g.\@ to~\cite{Tinkham,Tilley} for further reference on these models, and to~\cite{SS-book} for a mathematical introduction.
In this framework, in the asymptotic regime of a large Ginzburg-Landau parameter (which is indeed typically the case in real-life superconductors), vortices are known to become point-like, and to interact with one another according to a Coulomb pair potential.
In the mean-field limit of a large number of vortices, the evolution of the suitably normalized (macroscopic) mean-field density $\omega:\R^+\times\R^2\to\R$ of the vortex liquid was then naturally conjectured to satisfy the following Chapman-Rubinstein-Schatzman-E equation~\cite{WE-94,CRS-96}
\begin{align}\label{eq:LZh0}
\partial_t\omega=\Div(|\omega| \nabla (-\triangle)^{-1}\omega),\qquad\text{in $\R^+\times\R^2$},
\end{align}
where $(-\triangle)^{-1}\omega$ is indeed the Coulomb potential generated by the vortices. Although the vortex density $\omega$ is a priori a signed measure, we restrict here (and throughout this paper) to positive measures, $|\omega|=\omega\ge0$, so that the above is replaced by
\begin{align}\label{eq:LZh}
\partial_t\omega=\Div(\omega \nabla (-\triangle)^{-1}\omega),\qquad\text{in $\R^+\times\R^2$}.
\end{align}
More precisely, the mean-field supercurrent density $v:\R^+\times\R^2\to\R^2$ (linked to the vortex density through the relation $\omega=\curl v$) was conjectured to satisfy
\begin{align}\label{eq:LZh00-forv}
\partial_tv=\nabla P-v\curl v,\qquad\Div v=0,\qquad\text{in $\R^+\times\R^2$}.
\end{align}
Taking the curl of this equation indeed formally yields~\eqref{eq:LZh}, noting that the incompressibility constraint $\Div v=0$ allows to write $v=\nabla^\bot\triangle^{-1}\omega$.

In the context of superfluidity~\cite{Aftalion-06,Rougerie-these}, a conservative counterpart of the usual parabolic Ginzburg-Landau equation is used as a mesoscopic model. This counterpart is given by the Gross-Pitaevskii equation, which is a particular instance of a nonlinear Schrödinger equation. At the level of the mean-field evolution of the corresponding vortices, we then need to replace~\eqref{eq:LZh0}--\eqref{eq:LZh} by their conservative versions, thus replacing $\nabla(-\triangle)^{-1}\omega$ by $\nabla^\bot(-\triangle)^{-1}\omega$.
As argued in~\cite{Aranson-Kramer}, there is also physical interest in rather starting from the ``mixed-flow'' (or ``complex'') Ginzburg-Landau equation, which is a mix between the usual Ginzburg-Landau equation describing superconductivity ($\alpha=1$, $\beta=0$, below), and its conservative counterpart given by the Gross-Pitaevskii equation ($\alpha=0$, $\beta=1$, below). The above mean-field equation~\eqref{eq:LZh00-forv} for the supercurrent density $v$ is then replaced by the following, for $\alpha\ge0$, $\beta\in\R$,
\begin{align}\label{eq:limeqnS1}
\partial_tv=\nabla P-\alpha v\curl v+\beta v^\bot\curl v,\qquad\Div v=0,\qquad\text{in $\R^+\times\R^2$}.
\end{align}
Note that in the conservative case $\alpha=0$, this equation is equivalent to the 2D Euler equation, as is clear from the identity $v^\bot\curl v=(v\cdot \nabla)v-\frac12\nabla|v|^2$.

The first rigorous deductions of such (macroscopic) mean-field limit models from the (mesoscopic) 2D Ginzburg-Landau equation are due to~\cite{Kurzke-Spirn-14,Jerrard-Spirn-15,Serfaty-15}. As discovered by Serfaty~\cite{Serfaty-15}, in the dissipative case $\alpha>0$, the limiting equation~\eqref{eq:limeqnS1} is only correct in a regime of dilute vortices, while for a higher vortex density it must be replaced by the following compressible flow,
\begin{align}\label{eq:limeqnS2}
\partial_tv=\lambda\nabla (\Div v)-\alpha v\curl v+\beta v^\bot\curl v,\qquad\text{in $\R^+\times\R^2$},
\end{align}
for some $0<\lambda<\infty$.
In~\cite[Theorem~8.1.3]{D-thesis} we have further shown that for an even higher vortex density the relevant limiting equation is~\eqref{eq:limeqnS2} with $\lambda=0$.
In contrast, in the conservative case $\alpha=0$, the equation~\eqref{eq:limeqnS1} is always expected to hold in the corresponding mean-field limit.
To the best of our knowledge, this compressible model~\eqref{eq:limeqnS2} with $0\le\lambda<\infty$
is completely new in the literature.

When an electric current is applied to a type-II superconductor, it flows through the material, inducing a Lorentz-like force that makes the vortices move, dissipates energy, and disrupts the permanent supercurrents. As most technological applications of superconducting materials occur in the mixed state, it is crucial to design ways to reduce this energy dissipation, by preventing vortices from moving. For that purpose a common attempt consists in introducing in the material inhomogeneities (e.g.\@ impurities, or dislocations), which are indeed meant to destroy superconductivity locally and therefore ``pin down'' the vortices.
This is usually modeled by correcting the 2D Ginzburg-Landau equations with
a non-uniform equilibrium density
$a:\R^2\to[0,1]$ (or ``pinning weight''), which locally lowers the energy penalty associated with the vortices (see e.g.~\cite{Chapman-Richardson-97,BFGLV} for further details).
As formally predicted by Chapman and Richardson~\cite{Chapman-Richardson-97}, and first completely proven by~\cite{Jian-Song-01,S-Tice-11} (see also~\cite{Jerrard-Smets-15,Kurzke-Marzuola-Spirn-15} for the conservative case), in the asymptotic regime of a large Ginzburg-Landau parameter, this non-uniform density $a$ translates at the level of isolated vortices into an effective ``pinning potential'' $h=\log a$, indeed attracting the vortices to the minima of $a$. As shown in our companion paper~\cite{DS-16}, the mean-field equations~\eqref{eq:limeqnS1}--\eqref{eq:limeqnS2} are then replaced by~\eqref{eq:limeqn1}--\eqref{eq:limeqn2}, where the forcing $\Psi$ can be decomposed as $\Psi:=F^\bot-\nabla^\bot h$ in terms of the pinning force $-\nabla h$ and of some vector field $F:\R^2\to\R^2$ related to the imposed electric current (see also~\cite{Tice-10,S-Tice-11}).
In the conservative regime $\alpha=0$, $\beta=1$, the incompressible model~\eqref{eq:limeqn1} takes the form of the following inhomogeneous version of the 2D Euler equation: using the identity $ v^\bot\curl v=( v\cdot \nabla) v-\frac12\nabla| v|^2$, and setting $\tilde P:= P-\frac12| v|^2$,
\begin{align}\label{eq:2D-euler-lake}
\partial_t v=\nabla\tilde P+\Psi^\bot\curl v+( v\cdot \nabla) v,\qquad\Div(a v)=0,\qquad\text{in $\R^+\times\R^2$}.
\end{align}
In the context of 2D fluid dynamics, this conservative equation is known as the lake equation~\cite[p.235]{Greenspan-80} (see also~\cite{Camassa-Holm-Levermore-1,Camassa-Holm-Levermore-2}): the pinning weight $a$ corresponds to the effect of a varying depth in shallow water~\cite{Oliver-97a}, while the forcing $\Psi$ is similar to a background flow.

\subsection{Relation to previous works}
The simplified model~\eqref{eq:LZh} describes the mean-field limit of the gradient-flow evolution of any particle system with Coulomb interactions~\cite{D-15}. As such, it is related to nonlocal aggregation and swarming models, which have attracted a lot of mathematical interest in recent years (see e.g.~\cite{Bertozzi-Laurent-Leger-12,Carrillo-Choi-Hauray-14} and the references therein); they consist in replacing the Coulomb potential $(-\triangle)^{-1}$ by a convolution with a more general kernel corresponding to an attractive (rather than repulsive) nonlocal interaction.
Equation~\eqref{eq:LZh} was first studied by Lin and Zhang~\cite{Lin-Zhang-00}, who established global existence for vortex-sheet initial data $\omega|_{t=0}\in\Pc(\R^2)$, and uniqueness in some Zygmund space.
To prove global existence for such rough initial data, they proceed by regularization of the data, then passing to the limit in the equation using the compactness given by some very strong a priori estimates obtained by ODE arguments.
As our main source of inspiration, their approach is described in more detail in the sequel.
When viewing~\eqref{eq:LZh} as a mean-field model for the motion of the Ginzburg-Landau vortices in a superconductor, there is also interest in changing sign solutions and the correct model is then rather~\eqref{eq:LZh0}, for which global existence and uniqueness have been investigated in~\cite{Du-Zhang-03,Masmoudi-Zhang-05}, but for which an $\Ld^p$ well-posedness theory is still missing.
In~\cite{Ambrosio-S-08,Ambrosio-Mainini-S-11}, using
an energy approach where the equation is seen as a formal gradient flow in the Wasserstein space of probability measures à la Otto~\cite{Otto-01}, made rigorous by the minimizing movement approach of Jordan, Kinderlehrer, and Otto~\cite{JKO-98} (see also~\cite{Ambrosio-Gigli-Savare}),
analogues of equations~\eqref{eq:LZh0}--\eqref{eq:LZh} were studied in a 2D bounded domain, taking into account the possibility of mass entering or exiting the domain. In the case of nonnegative vorticity $\omega\ge0$, essentially the same existence and uniqueness results as those by Lin and Zhang were established in that setting in~\cite{Ambrosio-S-08}.
In the case $\omega\ge0$ on the whole plane, still a different approach was developed by Serfaty and V\'azquez~\cite{Serfaty-Vazquez-14}, where equation~\eqref{eq:LZh} is obtained as a limit of nonlocal diffusions, and where uniqueness is further established for bounded solutions using transport arguments à la Loeper~\cite{Loeper-06}.
Note that no uniqueness is expected to hold for general measure solutions of~\eqref{eq:LZh} (see~\cite[Section~8]{Ambrosio-S-08}). In the present paper, we focus on the case $\omega\ge0$ on the whole plane $\R^2$.

The model~\eqref{eq:limeqnS1} is a linear interpolation between the gradient-flow equation~\eqref{eq:LZh} (obtained for $\alpha=1$, $\beta=0$) and its conservative counterpart that is nothing but the 2D Euler equation (obtained for $\alpha=0$, $\beta=1$). The theory for the 2D Euler equation has been well-developed for a long time: global existence for vortex-sheet data is due to Delort~\cite{Delort-91}, while the only known uniqueness result, due to Yudovich~\cite{Yudovich-63}, holds in the class of bounded vorticity (see also~\cite{Bardos-Titi} and the references therein).
Regarding the general model~\eqref{eq:limeqnS1}, global existence and uniqueness results for smooth solutions are easily obtained by standard methods (see e.g.~\cite{Chemin-98}). Although not surprising, global existence for this model is further established here for vortex-sheet initial data, as well as uniqueness in the class of bounded vorticity.

In contrast, the compressible model~\eqref{eq:limeqnS2}, first introduced by Serfaty~\cite{Serfaty-15}, is completely new in the literature. In \cite[Appendix~B]{Serfaty-15}, only local-in-time existence and uniqueness of smooth solutions are proven in the non-degenerate case $\lambda>0$, using a standard iterative method. In the present paper,
in the parabolic regime $\alpha>0$, $\beta=0$, global existence for vortex-sheet data is further established in the non-degenerate case $\lambda>0$, while in the degenerate case $\lambda=0$ global existence with bounded data is obtained by exploiting the particular scalar structure of the corresponding equation.

The general equations~\eqref{eq:limeqn1}--\eqref{eq:limeqn2}, introduced in our companion paper~\cite{DS-16},
are seen as inhomogeneous versions of~\eqref{eq:limeqnS1}--\eqref{eq:limeqnS2} with forcing.
Since these models are new in the literature (except in the case~\eqref{eq:2D-euler-lake}), we wish to provide in the present paper a detailed discussion of local and global existence, uniqueness, and regularity issues.
In the conservative regime $\alpha=0$, $\beta=1$, the incompressible model~\eqref{eq:limeqn1} takes the form of the so-called lake equation~\eqref{eq:2D-euler-lake}, which has been studied in a bounded domain by Levermore, Oliver, and Titi~\cite{Levermore-Oliver-Titi-1,Levermore-Oliver-Titi-2,Oliver-97a} (see also~\cite{Bresch-Metivier-06}): global existence was established for $\Ld^2$ initial vorticity, as well as uniqueness in the class of bounded vorticity.
In the present paper, we improve on these previous results by establishing for equation~\eqref{eq:2D-euler-lake} on the whole plane $\R^2$ a global existence result for initial data in $\Ld^q(\R^2)$ with $q>1$.
It should be clear from the Delort type identity~\eqref{eq:delort1} below that inhomogeneities give rise to important difficulties: indeed, for $h$ non-constant, the first term $-\frac12| v|^2\nabla^\bot h$ in~\eqref{eq:delort1} does not vanish and is clearly not weakly continuous as a function of $v$ (although the second term is, as in Delort's classical theory for the 2D Euler equation~\cite{Delort-91}), so that the usual Delort's argument is no longer available to pass to the limit in the nonlinearity $v\curl v$.
Because of this difficulty and of the lack of strong enough a priori estimates for the conservative equation~\eqref{eq:2D-euler-lake}, we do not manage to reach vortex-sheet initial data in that case, as opposed to the simpler situation of the 2D Euler equation.

\subsection{Notions of weak solutions for~\eqref{eq:limeqn1} and~\eqref{eq:limeqn2}}

We first introduce the vorticity formulation of equations~\eqref{eq:limeqn1} and~\eqref{eq:limeqn2}, which will be more convenient to work with. Setting $\omega:=\curl v$ and $\zeta:=\Div(a v)$, each of these equations may be rewritten as a nonlinear nonlocal transport equation for the vorticity $\omega$,
\begin{align}\label{eq:limeqn1VF}
\partial_t\omega=\Div\!\Big(\omega\big(\alpha(\Psi+ v)^\bot+\beta(\Psi+ v)\big)\Big),\quad\curl v=\omega,\quad\Div(a v)=\zeta,
\end{align}
where in the incompressible case~\eqref{eq:limeqn1} we have $\zeta:=0$, while in the compressible case~\eqref{eq:limeqn2} $\zeta$ is the solution of the following transport-diffusion equation (which is highly degenerate as $\lambda=0$),
\begin{align}\label{eq:limeqn2VF}
\partial_t\zeta-\lambda\triangle\zeta+\lambda\Div(\zeta\nabla h)&=\Div\!\Big(a\omega\big(\!-\alpha(\Psi+ v)+\beta(\Psi+ v)^\bot\big)\Big).
\end{align}
Let us now precisely define our notions of weak solutions for~\eqref{eq:limeqn1} and~\eqref{eq:limeqn2}. (We denote by $\M_\loc^+(\R^2)$ the convex cone of locally finite non-negative Borel measures on $\R^2$, and by $\Pc(\R^2)$ the convex subset of probability measures, endowed with the usual weak-* topology.)
\begin{defin}\label{defin:sol}
Let $h,\Psi\in\Ld^\infty(\R^2)$, $T>0$, and set $a:=e^h$.
\begin{enumerate}[(a)]
\item Given $ v^\circ\in\Ld^2_\loc(\R^2)^2$ with $\omega^\circ=\curl v^\circ\in\M_\loc^+(\R^2)$ and $\zeta^\circ:=\Div(a v^\circ)\in\Ld^2_\loc(\R^2)$,
we say that $ v$ is a {\it weak solution of~\eqref{eq:limeqn2}} on $[0,T)\times\R^2$ with initial data $ v^\circ$, if $ v\in\Ld^2_\loc([0,T)\times\R^2)^2$ satisfies $\omega:=\curl v\in\Ld^1_\loc([0,T);\M_\loc^+(\R^2))$, $\zeta:=\Div(a v)\in\Ld^2_\loc([0,T);\Ld^2(\R^2))$, $| v|^2\omega\in\Ld^1_\loc([0,T);\Ld^1(\R^2))$ (hence also $\omega v\in\Ld^1_\loc([0,T)\times\R^2)^2$), and satisfies~\eqref{eq:limeqn2} in the distributional sense, that is, for all $\psi\in C^\infty_c([0,T)\times\R^2)^2$,
\begin{align*}
\int_{\R^d}\psi(0,\cdot) \cdot  v^\circ+\iint_{\R^+\times\R^d} v\cdot\partial_t\psi =\lambda\iint_{\R^+\times\R^d} a^{-1}\zeta\Div\psi+\iint_{\R^+\times\R^d}\psi\cdot(\alpha(\Psi+ v)-\beta(\Psi+ v)^\bot)\omega.
\end{align*}
\item Given $ v^\circ\in\Ld^2_\loc(\R^2)^2$ with $\omega^\circ:=\curl v^\circ\in\M_\loc^+(\R^2)$ and $\Div(a v^\circ)=0$, we say that $ v$ is a {\it weak solution of~\eqref{eq:limeqn1}} on $[0,T)\times\R^2$ with initial data $ v^\circ$, if $ v\in\Ld^2_\loc([0,T)\times\R^2)^2$ satisfies $\omega:=\curl v\in\Ld^1_\loc([0,T);\M_\loc^+(\R^2))$,  $| v|^2\omega\in\Ld^1_\loc([0,T);\Ld^1(\R^2)^2)$ (hence also $\omega v\in\Ld^1_\loc([0,T)\times\R^2)^2$), $\Div(a v)=0$ in the distributional sense, and satisfies the vorticity formulation~\eqref{eq:limeqn1VF} in the distributional sense, that is, for all $\psi\in C^\infty_c([0,T)\times\R^2)$,
\begin{align*}
\int_{\R^d} \psi(0,\cdot)\omega^\circ+\iint_{\R^+\times\R^d}\omega \partial_t\psi =\iint_{\R^+\times\R^d}\nabla\psi\cdot (\alpha(\Psi+ v)^\bot+\beta(\Psi+ v))\omega.
\end{align*}
\item Given $ v^\circ\in\Ld^2_\loc(\R^2)^2$ with $\omega^\circ:=\curl v^\circ\in\M_\loc^+(\R^2)$ and $\Div(a v^\circ)=0$, we say that $ v$ is a {\it very weak solution of~\eqref{eq:limeqn1}} on $[0,T)\times\R^2$ with initial data $ v^\circ$, if $ v\in\Ld^2_\loc([0,T)\times\R^2)^2$ satisfies $\omega:=\curl v\in\Ld^1_\loc([0,T);\M_\loc^+(\R^2))$, $\Div(a v)=0$ in the distributional sense, and satisfies, for all $\psi\in C^\infty_c([0,T)\times\R^2)$,
\begin{multline*}
\int_{\R^d} \psi(0,\cdot)\omega^\circ+\iint_{\R^+\times\R^d}\omega \partial_t\psi =\iint_{\R^+\times\R^d}(\alpha\nabla\psi+\beta\nabla^\bot\psi)\cdot \Big(\Psi^\bot\omega+\frac12| v|^2\nabla h\Big)\\
-\iint_{\R^+\times\R^d}aS_{ v}:\nabla\big(a^{-1}(\alpha\nabla\psi+\beta\nabla^\bot\psi)\big),
\end{multline*}
in terms of the stress-energy tensor $S_{ v}:= v\otimes v-\frac12\Id| v|^2$.\qedhere
\end{enumerate}
\end{defin}

\begin{rems}\label{rem:sol}$  $
\begin{enumerate}[(i)]
\item Weak solutions of~\eqref{eq:limeqn2} are defined directly from~\eqref{eq:limeqn2}, and satisfy in particular the vorticity formulation~\eqref{eq:limeqn1VF}--\eqref{eq:limeqn2VF} in the distributional sense. Regarding weak solutions of~\eqref{eq:limeqn1}, they are rather defined in terms of the vorticity formulation~\eqref{eq:limeqn1VF}, in order to avoid compactness and regularity issues related to the pressure $ P$. Nevertheless, if $ v$ is a weak solution of~\eqref{eq:limeqn1} in the above sense, then under mild regularity assumptions we may use the formula $ v=a^{-1}\nabla^\bot(\Div a^{-1}\nabla)^{-1}\omega$ to deduce that $ v$ actually satisfies~\eqref{eq:limeqn1} in the distributional sense on $[0,T)\times\R^2$ for some distribution $ P$ (cf.\@ Lemma~\ref{lem:pressure} below for detail).
\item The definition~(c) of a very weak solution of~\eqref{eq:limeqn1} is motivated as follows (see also the notion of ``general weak solutions'' of~\eqref{eq:LZh} in~\cite{Lin-Zhang-00}). In the purely conservative case $\alpha=0$, there are too few a priori estimates to make sense of the product $\omega v$. As is now common in 2D fluid dynamics (see e.g.\@ \cite{Chemin-98}), the idea is to reinterpret this product in terms of the stress-energy tensor $S_{ v}$, using the following identity: given $\Div(a v)=0$, we have for smooth enough fields
\begin{align}\label{eq:delort1}
\omega v=-\frac12| v|^2\nabla^\bot h-a^{-1}(\Div(aS_{ v}))^\bot,
\end{align}
where the right-hand side now makes sense in $\Ld^1_\loc([0,T);W^{-1,1}_\loc(\R^2)^2)$ whenever $ v\in\Ld^2_\loc([0,T)\times\R^2)^2$.
In particular, if $\omega\in\Ld^p_\loc ([0,T)\times\R^2)$ and $ v\in\Ld^{p'}_\loc([0,T)\times\R^2)$ for some $1\le p\le\infty$, $\frac1p+\frac1{p'}=1$, then the product $\omega v$ makes perfect sense and the above identity~\eqref{eq:delort1} holds in the distributional sense, hence in that case $ v$ is a weak solution of~\eqref{eq:limeqn1} whenever it is a very weak solution. In reference to Delort's work~\cite{Delort-91}, identity~\eqref{eq:delort1} is henceforth called an ``(inhomogeneous) Delort type identity''.\qedhere
\end{enumerate}
\end{rems}

\subsection{Statement of the main results}

Global existence and regularity results are summarized in the following theorem.
Our approach relies on proving a priori estimates for the vorticity $\omega$ in $\Ld^q(\R^2)$ for some $q>1$. For the compressible model~\eqref{eq:limeqn2}, such estimates are only obtained in the parabolic regime, hence our limitation to that setting. In parabolic cases, particularly strong estimates are available, and existence is then established even for vortex-sheet initial data, thus completely extending the known theory for~\eqref{eq:LZh} (see~\cite{Lin-Zhang-00,Serfaty-Vazquez-14}). Note that the additional exponential growth in the boundedness effect~\eqref{eq:flat1} below is only due to the forcing $\Psi$.
In the conservative incompressible case, the situation is the most delicate because of a lack of strong enough a priori estimates, and only existence of very weak solutions is expected and proven. As is standard in 2D fluid dynamics (see e.g.~\cite{Chemin-98}), the natural space for the solution $ v$ is $\Ld^\infty_\loc(\R^+;\bar  v^\circ+\Ld^2(\R^2)^2)$ for a given smooth reference field $\bar  v^\circ:\R^2\to\R^2$.

\begin{theor}[Global existence]\label{th:main}
Let $\lambda>0$, $\alpha\ge0$, $\beta\in\R$, $h,\Psi\in W^{1,\infty}(\R^2)^2$, and set $a:=e^h$. Let $\bar  v^\circ\in W^{1,\infty}(\R^2)^2$ be some reference map with $\bar \omega^\circ:=\curl\bar  v^\circ\in \Pc\cap H^{s_0}(\R^2)$ for some $s_0>1$, and with either $\Div(a\bar  v^\circ)=0$ in the case~\eqref{eq:limeqn1}, or $\bar\zeta^\circ:=\Div(a\bar  v^\circ)\in H^{s_0}(\R^2)$ in the case~\eqref{eq:limeqn2}. Let $ v^\circ\in \bar  v^\circ+\Ld^2(\R^2)^2$, with $\omega^\circ:=\curl v^\circ\in\Pc(\R^2)$, and with either $\Div(a v^\circ)=0$ in the case~\eqref{eq:limeqn1}, or $\zeta^\circ:=\Div(a v^\circ)\in\Ld^2(\R^2)$ in the case~\eqref{eq:limeqn2}.
The following hold:
\begin{enumerate}[(i)]
\item \emph{Parabolic compressible case (that is, \eqref{eq:limeqn2} with $\alpha>0$, $\beta=0$):}\\
There exists a weak solution $ v\in\Ld^\infty_\loc(\R^+;\bar  v^\circ+\Ld^2(\R^2)^2)$
on $\R^+\times\R^2$ with initial data $ v^\circ$, with $\omega:=\curl v\in \Ld^\infty(\R^+;\Pc(\R^2))$ and $\zeta:=\Div(a v)\in \Ld^2_\loc(\R^+;\Ld^2(\R^2))$, and with
\begin{align}\label{eq:flat1}
\|\omega^t\|_{\Ld^\infty}\le(\alpha t)^{-1}+C\alpha^{-1}e^{Ct},\qquad\text{for all $t>0$,}
\end{align}
where the constant $C>0$ depends only on an upper bound on $\alpha$, $|\beta|$, and $\|(h,\Psi)\|_{W^{1,\infty}}$.
Moreover, if $\omega^\circ\in\Ld^q(\R^2)$ for some $q>1$, then $\omega\in\Ld^\infty_\loc(\R^+;\Ld^q(\R^2))$.
\item \emph{Parabolic incompressible case (that is, \eqref{eq:limeqn1} with $\alpha>0$, $\beta=0$, or with $\alpha>0$, $\beta\in\R$, $h$ constant):}\\
There exists a weak solution $ v\in\Ld^\infty_\loc(\R^+;\bar  v^\circ+\Ld^{2}(\R^2)^2)$ on $\R^+\times\R^2$ with initial data $ v^\circ$, with $\omega:=\curl v\in \Ld^\infty(\R^+;\Pc(\R^2))$, and with the boundedness effect~\eqref{eq:flat1}.
Moreover, if $\omega^\circ\in\Ld^q(\R^2)$ for some $q>1$, then $\omega\in\Ld^\infty_\loc(\R^+;\Ld^q(\R^2))\cap\Ld^{q+1}_\loc(\R^+;\Ld^{q+1}(\R^2))$.
\item \emph{Mixed-flow incompressible case (that is, \eqref{eq:limeqn1} with $\alpha>0$, $\beta\in\R$):}\\
If $\omega^\circ\in\Ld^q(\R^2)$ for some $q>1$, there exists a weak solution $ v\in\Ld^\infty_\loc(\R^+;\bar  v^\circ+\Ld^{2}(\R^2)^2)$ on $\R^+\times\R^2$ with initial data $ v^\circ$, and with $\omega:=\curl v\in \Ld^\infty_\loc(\R^+;\Pc\cap\Ld^q(\R^2))\cap\Ld^{q+1}_\loc(\R^+;\Ld^{q+1}(\R^2))$.
\item \emph{Conservative incompressible case (that is, \eqref{eq:limeqn1} with $\alpha=0$, $\beta\in\R$):}\\
If $\omega^\circ\in\Ld^q(\R^2)$ for some $q>1$, there exists a very weak solution $ v\in\Ld^\infty_\loc(\R^+;\bar  v^\circ+\Ld^2(\R^2)^2)$ on $\R^+\times\R^2$ with initial data $ v^\circ$, and with $\omega:=\curl v\in\Ld^\infty_\loc(\R^+;\Pc\cap\Ld^q(\R^2))$. This is a weak solution whenever $q\ge4/3$.
\end{enumerate}
We set $\zeta^\circ,\bar\zeta^\circ,\zeta:=0$ in the incompressible case~\eqref{eq:limeqn1}.
If in addition $\omega^\circ,\zeta^\circ\in\Ld^\infty(\R^2)$, then we further have $ v\in\Ld^\infty_\loc(\R^+;\Ld^\infty(\R^2)^2)$, $\omega\in\Ld^\infty_\loc(\R^+;\Ld^1\cap\Ld^\infty(\R^2))$, and $\zeta\in\Ld^\infty_\loc(\R^+;\Ld^2\cap\Ld^\infty(\R^2))$.
If $h$, $\Psi$, $\bar  v^\circ\in W^{s+1,\infty}(\R^2)^2$ and $\omega^\circ$, $\bar\omega^\circ$, $\zeta^\circ$, $\bar\zeta^\circ\in H^s(\R^2)$ for some $s>1$, then $ v\in \Ld^\infty_\loc(\R^+;\bar  v^\circ+H^{s+1}(\R^2)^2)$ and $\omega,\zeta\in \Ld^\infty_\loc(\R^+;H^{s}(\R^2))$. If $h$, $\Psi$, $ v^\circ\in C^{s+1}(\R^2)^2$ for some non-integer $s>0$, then $ v\in\Ld^\infty_\loc(\R^+;C^{s+1}(\R^2)^2)$.
\end{theor}

Regarding the regimes that are not described in the above (that is, the mixed-flow compressible case as well as the a priori unphysical case $\alpha<0$), only local-in-time existence is proven for smooth enough initial data (stated here in Sobolev spaces).
Note that for the mixed-flow degenerate case $\lambda=0$, $\alpha>0$, $\beta\ne0$, even local-in-time existence remains an open problem.

\begin{theor}[Local existence]\label{th:mainloc}
Given some $s>1$, let $h,\Psi,\bar  v^\circ\in W^{s+1,\infty}(\R^2)^2$, set $a:=e^h$, and let $ v^\circ\in \bar  v^\circ+H^{s+1}(\R^2)^2$ with $\omega^\circ:=\curl v^\circ$, $\bar\omega^\circ:=\curl\bar  v^\circ\in H^s(\R^2)$, and with either $\Div(a v^\circ)=\Div(a\bar  v^\circ)=0$ in the case~\eqref{eq:limeqn1}, or $\zeta^\circ:=\Div(a v^\circ)$, $\bar\zeta^\circ:=\Div(a\bar  v^\circ)\in H^s(\R^2)$ in the case~\eqref{eq:limeqn2}. The following hold:
\begin{enumerate}[(i)]
\item \emph{Incompressible case (that is, \eqref{eq:limeqn1} with $\alpha,\beta\in\R$):}\\
There exists $T>0$ and a weak solution $ v\in\Ld^\infty_\loc([0,T); \bar  v^\circ+H^{s+1}(\R^2)^2)$ on $[0,T)\times\R^2$ with initial data $ v^\circ$.
\item \emph{Non-degenerate compressible case (that is, \eqref{eq:limeqn2} with $\alpha,\beta\in\R$, $\lambda>0$):}\\
There exists $T>0$ and a weak solution $ v\in\Ld^\infty_\loc([0,T); \bar  v^\circ+H^{s+1}(\R^2)^2)$ on $[0,T)\times\R^2$ with initial data~$ v^\circ$.
\item \emph{Degenerate parabolic compressible case (that is, \eqref{eq:limeqn2} with $\alpha\in\R$, $\beta=\lambda=0$):}\\
If $\Psi$, $\bar  v^\circ\in W^{s+2,\infty}(\R^2)^2$ and $\omega^\circ$, $\bar\omega^\circ\in H^{s+1}(\R^2)$, there exists $T>0$ and a weak solution $ v\in \Ld^\infty_\loc([0,T); \bar  v^\circ+ H^{s+1}(\R^2)^2)$ on $[0,T)\times\R^2$ with initial data $ v^\circ$, and with $\omega:=\curl v\in\Ld^\infty_\loc([0,T);H^{s+1}(\R^2))$.\qedhere
\end{enumerate}
\end{theor}

We now turn to uniqueness issues. No uniqueness is expected to hold for general weak measure solutions of~\eqref{eq:limeqn1}, as it is already known to fail for the 2D Euler equation (see e.g.~\cite{Bardos-Titi} and the references therein), and as it is also expected to fail for equation~\eqref{eq:LZh} (see~\cite[Section~8]{Ambrosio-S-08}). In both cases, as already explained, the only known uniqueness results are in the class of bounded vorticity. For the general incompressible model~\eqref{eq:limeqn1}, similar arguments are still available and the same uniqueness result holds. For the compressible model~\eqref{eq:limeqn2}, we only obtain uniqueness in a class with stronger regularity, as a consequence of some weak-strong principle stated in Proposition~\ref{prop:limeqn-unique}.

\begin{theor}[Uniqueness]\label{th:mainunique}
Let $\lambda\ge0$, $\alpha,\beta\in\R$, $T>0$, $h,\Psi\in W^{1,\infty}(\R^2)$, and set $a:=e^h$. Let $ v^\circ:\R^2\to\R^2$ with $\curl v^\circ\in\Pc(\R^2)$, and with either $\Div(a v^\circ)=0$ in the case~\eqref{eq:limeqn1}, or $\Div(a v^\circ)\in\Ld^2(\R^2)$ in the case~\eqref{eq:limeqn2}.
\begin{enumerate}[(i)]
\item \emph{Incompressible case (that is,~\eqref{eq:limeqn1} with $\alpha,\beta\in\R$):}\\
There exists at most a unique weak solution $ v$ on $[0,T)\times\R^2$ with initial data $ v^\circ$, in the class of all $w$'s with $\curl w\in\Ld^\infty_\loc([0,T);\Ld^\infty(\R^2))$.
\item \emph{Non-degenerate compressible case (that is,~\eqref{eq:limeqn2} with $\alpha,\beta\in\R$, $\lambda>0$):}\\
There exists at most a unique weak solution $ v$ on $[0,T)\times\R^2$ with initial data $ v^\circ$, in the class $\Ld^2_\loc([0,T); v^\circ+\Ld^2(\R^2)^2)\cap\Ld^\infty_\loc([0,T);W^{1,\infty}(\R^2)^2)$.
\qedhere
\end{enumerate}
\end{theor}

Finally, in Appendix~\ref{sec:degenerate} jointly written with Julian Fischer, we establish the following global well-posedness result for the degenerate parabolic case $\lambda=\beta=0$, $\alpha>0$.
The proof is of a very different nature from the other cases, exploiting the explicit scalar structure of the solution $v$.

\begin{theor}[Degenerate parabolic compressible case]\label{th:deg-case+JF}
Let $\lambda=\beta=0$, $\alpha=1$, let $v^\circ,\Psi\in W^{1,\infty}(\R^2)^2$ with $\curl v^\circ\in\Pc(\R^2)$.
Then there exists a global strong solution $ v\in \Ld^\infty_\loc(\R^+;\Ld^\infty(\R^2))\cap\Ld^\infty_\loc(\R^+; v^\circ+\Ld^1(\R^2))$ of~\eqref{eq:limeqn2} on $\R^+\times\R^2$ with initial data~$ v^\circ$ and with $\curl v\in\Ld^\infty_\loc(\R^+;\Pc\cap\Ld^\infty(\R^2))$. This solution $v$ is unique in the class of all $w$'s in $\Ld^\infty_\loc(\R^+\times\R^2)$ with $\curl w\in\Ld_\loc^\infty(\R^+;\Pc\cap\Ld^\infty_\loc(\R^2))$.\\
If in addition for some $s\ge0$ we have $ v^\circ,\Psi\in W^{1\vee s,\infty}(\R^2)^2$ and $\curl v^\circ,\curl\Psi\in W^{s,\infty}(\R^2)$, then $ v\in W^{1,\infty}_\loc(\R^+;W^{s,\infty}(\R^2)^2)$. If for some $s\ge1$ we further have $ v^\circ,\Psi\in W^{s,\infty}(\R^2)^2$, $\curl v^\circ\in H^s\cap W^{s,\infty}(\R^2)$, and $\curl\Psi\in W^{s,\infty}(\R^2)$, then $ v\in\Ld^\infty_\loc(\R^+; v^\circ+H^s\cap W^{s,\infty}(\R^2)^2)$.
\end{theor}

\subsection{Roadmap to the proof of the main results}

To ease the presentation, various independent PDE results needed in the proofs are isolated in Section~\ref{chap:prelim}, including general a priori estimates for transport and transport-diffusion equations, some global elliptic regularity results, as well as critical potential theory estimates. The interest of such estimates for our purposes should be already clear from a quick look at the vorticity formulation~\eqref{eq:limeqn1VF}--\eqref{eq:limeqn2VF}.
To the best of our knowledge, most of these PDE results
cannot be found in this form in the literature, and proofs are included in Appendix~\ref{app:proofs-prel}.

We start in Section~\ref{chap:local} with the local existence of smooth solutions, summarized in Theorem~\ref{th:mainloc} above. In the non-degenerate case $\lambda>0$, the proof follows from a standard iterative scheme as in~\cite[Appendix~B]{Serfaty-15}. It is performed here in Sobolev spaces, but could be done in Hölder spaces as well. In the degenerate parabolic case $\lambda=\beta=0$, $\alpha>0$, a similar argument holds, but requires a more careful analysis of the iterative scheme.

In Section~\ref{chap:global} we then turn to global existence. In order to pass from local to global existence, we prove estimates for the Sobolev and Hölder norms of solutions through the norm of their initial data. As shown in Section~\ref{chap:propagation}, these estimates essentially follow from an a priori control of the vorticity in $\Ld^\infty(\R^2)$.
In the work by Lin and Zhang~\cite{Lin-Zhang-00} on the simpler model~\eqref{eq:LZh}, such an a priori estimate for the vorticity is achieved by a direct ODE argument, using that for~\eqref{eq:LZh} the evolution of the vorticity along characteristics can be integrated explicitly. This explicit structure is lost in the more sophisticated models~\eqref{eq:limeqn1} and~\eqref{eq:limeqn2}, but in the parabolic case we still manage to design suitable ODE type arguments (cf.\@ Lemma~\ref{lem:Lpest}(iii)).
This leads to the nice boundedness effect~\eqref{eq:flat1} for the vorticity (depending on the initial mass $\int\omega^\circ=1$ only!), which of course differs from~\cite{Lin-Zhang-00} by the additional exponential growth due to the forcing $\Psi$, and which is at the core of our existence result for vortex-sheet initial data.
In the mixed-flow case for the incompressible model~\eqref{eq:limeqn1}, such ODE arguments are no longer available, and only a weaker estimate is obtained, controlling for all $1\le q\le\infty$ the $\Ld^q$-norm of the solution (as well as its space-time $\Ld^{q+1}$-norm if $\alpha>0$) by the $\Ld^q$-norm of the initial data (cf.\@ Lemma~\ref{lem:Lpvort}). This is proven by a careful examination of the evolution of $\Ld^q$-norms of the vorticity.

In order to handle rough initial data, we regularize the data and then pass to the limit in the equation, using the compactness given by the available a priori estimates. As already noticed, for $h$ non-constant, the usual Delort's argument~\cite{Delort-91} fails (due to the first right-hand side term in~\eqref{eq:delort1}), so that stronger compactness is needed to pass to the limit in the nonlinearity $\omega v$ than in the simpler case of the 2D Euler equation.
While energy estimates only give bounds for $ v$ in $\bar  v^\circ+\Ld^2(\R^2)^2$ and for $\zeta$ in $\Ld^2(\R^2)$ (cf.\@ Lemma~\ref{lem:aprioriest}),
the additional estimates for the vorticity in $\Ld^q(\R^2)$, $q>1$, turn out to be crucial. %To get to vortex-sheet initial data in parabolic cases, a
As in~\cite{Lin-Zhang-00}, we need to make use of some compactness result due to Lions~\cite{Lions-98} in the context of the compressible Navier-Stokes equations. The model~\eqref{eq:limeqn1} in the conservative case $\alpha=0$ is however more subtle because of a lack of strong enough a priori estimates: only very weak solutions are then expected and obtained (for initial vorticity in $\Ld^q(\R^2)$ with $q>1$), and compactness is in that case carefully proven by hand, which is one of the main achievements in this paper (cf.\@ Proposition~\ref{prop:globexist3}(iv)).

Uniqueness issues are addressed in Section~\ref{chap:unique}.
Similarly as in~\cite[Appendix~B]{Serfaty-15}, weak-strong uniqueness principles for both~\eqref{eq:limeqn1} and~\eqref{eq:limeqn2} are established by energy methods in the non-degenerate case $\lambda>0$.
In the degenerate parabolic case $\lambda=\beta=0$, $\alpha>0$, these energy methods fail: an additional term needs to be added to the usual energy, and on this basis a different weak-strong uniqueness principle is obtained.
Following the modulated energy strategy developed by Serfaty~\cite{Serfaty-15}, these weak-strong principles are the key to the mean-field limit results for Ginzburg-Landau vortices in the companion paper~\cite{DS-16}.
For the incompressible model~\eqref{eq:limeqn1}, uniqueness in the class of bounded vorticity is further obtained using the approach by Serfaty and V\'azquez~\cite{Serfaty-Vazquez-14} for the simpler model~\eqref{eq:LZh}, which consists in adapting the corresponding uniqueness result for the 2D Euler equation due to Yudovich~\cite{Yudovich-63} together with a transport argument à la Loeper~\cite{Loeper-06}.

Finally, in Appendix~\ref{sec:degenerate} jointly written with Julian Fischer, a global well-posedness result is established for the degenerate parabolic case $\lambda=\beta=0$, $\alpha>0$. The proof consists in exploiting the scalar structure of the solution $ v$ to reduce the equation to a Burgers type equation with additional quadratic damping and forcing terms, and with unit initial data. Suitable ODE type arguments then allow to explicitly integrate this equation, and the desired properties of the solution easily follow.

\subsubsection*{Notation}
We use the notation $C$ for (unless explicitly stated) universal constants that may vary from line to line. We write $\lesssim$ and $\gtrsim$ for $\le$ and $\ge$ up to such multiplicative constants $C$, and we use the notation $\simeq$ if both relations $\lesssim$ and $\gtrsim$ hold. We add a subscript in order to indicate the dependence on other parameters. 
However, as we need to keep track of the dependence on various %norms of
controlled quantities, and as subscripts would quickly become unreadable, we usually do not use any subscript and simply indicate in the beginning of each statement or proof on what quantities constants are allowed to depend.

For any vector field $F=(F_1,F_2)$ on $\R^2$, we denote $F^\bot=(-F_2,F_1)$, $\curl F=\partial_1F_2-\partial_2F_1$, and also as usual $\Div F=\partial_1F_1+\partial_2F_2$. Given two linear operators $A,B$ on some function space, we denote by $[A,B]:=AB-BA$ their commutator. For any exponent $1\le p\le\infty$, we denote its Hölder conjugate by $p':=p/(p-1)$. We denote by $B(x,r)$ the ball of radius $r$ centered at $x$ in $\R^d$, and we set $B_r:=B(0,r)$ and $B(x):=B(x,1)$. We use the notation $a\wedge b=\min\{a,b\}$ and $a\vee b=\max\{a,b\}$ for all $a,b\in\R$. Given a function $f:\R^d\to\R$, we denote its positive and negative parts by $f^+(x):=0\vee f(x)$ and $f^-(x):=0\vee(-f)(x)$, respectively.
The space of Lebesgue-measurable functions on $\R^d$ is denoted by $\Mes(\R^d)$, the set of Borel probability measures on $\R^d$ is denoted by $\Pc(\R^d)$, and for all $\sigma>0$, $C_b^\sigma(\R^d)$ stands as usual for the Banach space $C_b^{\lfloor\sigma\rfloor,\sigma-\lfloor\sigma\rfloor}(\R^d)$ of bounded Hölder functions. For $\sigma\in(0,1)$, we denote by $|\cdot|_{C^\sigma}$ the usual Hölder seminorm, and by $\|\cdot\|_{C^\sigma}:=|\cdot|_{C^\sigma}+\|\cdot\|_{\Ld^\infty}$ the corresponding norm. We denote by $\Ld^p_{\uloc}(\R^d)$ the Banach space of functions that are uniformly locally $\Ld^p$-integrable, with norm $\|f\|_{\Ld^p_\uloc}:=\sup_x\|f\|_{\Ld^p(B(x))}$. Given a Banach space $X\subset \Mes(\R^d)$ and $t>0$, we use the notation $\|\cdot\|_{\Ld^p_tX}$ for the usual norm in $\Ld^p([0,t];X)$.

\section{Preliminary results}\label{chap:prelim}
In this section, we establish various PDE results that are needed in the sequel and are of independent interest. As most of them do not depend on the choice of space dimension $2$, they are stated here in general dimension $d\ge1$.
We first recall the following useful proxy for a fractional Leibniz rule, which is essentially due to Kato and Ponce~\cite{Kato-Ponce-88} based on ideas by Coifman and Meyer~\cite{Coifman-Meyer-78,Coifman-Meyer-86} (see e.g.\@ \cite[Theorem~1.4]{Gulisashvili-Kon-96}).

\begin{lem}[Kato-Ponce inequality]\label{lem:katoponce-1}
Let $d\ge1$, $s\ge0$, $p\in(1,\infty)$, and let $\frac1{p_i}+\frac1{q_i}=\frac1p$ with $p_i,q_i\in(1,\infty]$ for~$i=1,2$. Then for $f,g\in C^\infty_c(\R^d)$ we have
\[\|fg\|_{W^{s,p}}\lesssim \|f\|_{\Ld^{p_1}}\|g\|_{W^{s,q_1}}+\|g\|_{\Ld^{p_2}}\|f\|_{W^{s,q_2}}.\qedhere\]
\end{lem}

The following gives a general estimate for the evolution of the Sobolev norms of the solutions of transport equations (see also~\cite[equation~(7)]{Lin-Zhang-00} for a simpler version), which will be useful in the sequel since the vorticity~$\omega$ indeed satisfies an equation of this form~\eqref{eq:limeqn1VF}. The proof is postponed to Appendix~\ref{app:proofs-prel}.

\begin{lem}[A priori estimate for transport equations]\label{lem:katoponce}
Let $d\ge1$, $s\ge0$, $T>0$. Given a vector field $w\in \Ld^\infty_\loc([0,T);W^{1,\infty}(\R^d)^d)$ with $w-W\in \Ld^\infty_\loc([0,T);H^{s+1}(\R^d)^d)$ for some reference map $W\in W^{s+1,\infty}(\R^d)^d$, let $\rho\in \Ld^\infty_\loc([0,T);H^s(\R^d))$ satisfy the transport equation $\partial_t\rho=\Div(\rho w)$ in the distributional sense on $[0,T)\times\R^d$.
Then for all $t\in[0,T)$ we have
\begin{multline}\label{eq:katoponcecom}
\partial_t\|\rho^t\|_{H^s}\lesssim_s \|(\nabla w^t,\nabla W)\|_{\Ld^{\infty}}\|\rho^t\|_{H^s}+\|W\|_{W^{s+1,\infty}}\|\rho^t\|_{\Ld^2}\\
+\min\big\{\|\rho^t\|_{\Ld^\infty}\|\Div(w^t-W)\|_{H^s}+\|\rho^t\|_{W^{1,\infty}}\|w^t-W\|_{H^{s}}\,;\,\|\rho^t\|_{\Ld^{\infty}}\|w^t-W\|_{H^{s+1}}\big\},
\end{multline}
where we use the notation $\|(\nabla w^t,\nabla W)\|_{\Ld^{\infty}}:=\|\nabla w^t\|_{W^{1,\infty}}\vee\|\nabla W\|_{W^{1,\infty}}$. Moreover, for all $t\in[0,T)$,
\begin{align}\label{eq:tsph-1}
\|\rho^t-\rho^\circ\|_{\dot H^{-1}}\le\|\rho\|_{\Ld^\infty_t\Ld^2}\|w\|_{\Ld^1_t\Ld^\infty}.
\end{align}
\end{lem}

As the evolution of the divergence $\zeta$ in the compressible model~\eqref{eq:limeqn2} is given by the transport-diffusion equation~\eqref{eq:limeqn2VF}, the following parabolic regularity results will be needed.
While items~(i) and~(ii) are classical, item~(iii) is less standard (see however~\cite[Section~3.4]{BCD-11} for a variant of this estimate), and a complete proof is included in Appendix~\ref{app:proofs-prel}.

\begin{lem}[A priori estimates for transport-diffusion equations]\label{lem:parreg+tsp}
Let $d\ge1$, $T>0$. Let $g\in\Ld^1_\loc([0,T)\times\R^d)^d$, and let $w$ satisfy $\partial_tw-\triangle w+\Div(w\nabla h)=\Div g$ in the distributional sense on $[0,T)\times\R^d$ with initial data $w^\circ$. The following hold:
\begin{enumerate}[(i)]
\item for all $s\ge0$, if $\nabla h\in W^{s,\infty}(\R^d)^d$, $w\in \Ld^\infty_\loc([0,T);H^s(\R^d))$, and $g\in \Ld^2_\loc([0,T);H^s(\R^d)^d)$, then we have for all $t\in[0,T)$,
\[\|w^t\|_{H^s}\le Ce^{Ct}(\|w^\circ\|_{H^s}+\|g\|_{\Ld^2_tH^s}),\]
where the constant $C$ depends only on an upper bound on $s$ and $\|\nabla h\|_{W^{s,\infty}}$;
\item if $\nabla h\in\Ld^\infty(\R^d)$, $w^\circ\in \Ld^2(\R^d)$, $w\in \Ld^\infty_\loc([0,T); \Ld^2(\R^d))$, and $g\in\Ld^2_\loc([0,T);\Ld^2(\R^d))$, then we have for all $t\in[0,T)$,
\[\|w^t-w^\circ\|_{\dot H^{-1}\cap \Ld^2}\le Ce^{Ct}(\|w^\circ\|_{\Ld^2}+\|g\|_{\Ld^2_t\Ld^2}),\]
where the constant $C$ depends only on an upper bound on $\|\nabla h\|_{\Ld^\infty}$;
\item for all $1\le p,q\le\infty$, and all $\frac{dq}{d+q}< s\le q$, $s\ge1$, if $\nabla h\in\Ld^\infty(\R^d)$, $w\in\Ld^p_\loc([0,T);\Ld^q(\R^d))$, and $g\in\Ld^p_\loc([0,T);\Ld^s(\R^d))$, then we have for all $t\in[0,T)$,
\begin{align*}
\|w\|_{\Ld^p_t\Ld^q}&\le C(\|w^\circ\|_{\Ld^q}+\kappa^{-1}t^\kappa\|g\|_{\Ld^p_t\Ld^s})\exp\Big(\inf_{2<r<\infty}r^{-1}\big(1+(r-2)^{-r/2}\big)(Ct)^{r/2}\Big).
\end{align*}
where $\kappa:=\frac d2(\frac1d+\frac1q-\frac1s)>0$, and where the constant $C$ depends only on an upper bound on $\|\nabla h\|_{\Ld^\infty}$.\qedhere
\end{enumerate}
\end{lem}

Another ingredient that we need is the following string of critical potential theory estimates. The Sobolev embedding for $W^{1,d}(\R^d)$ gives that $\|\nabla\triangle^{-1}w\|_{\Ld^\infty}$ is {\it almost} bounded by the $\Ld^d(\R^d)$-norm of $w$, while the Calderón-Zygmund theory gives that $\|\nabla^2\triangle^{-1}w\|_{\Ld^\infty}$ is {\it almost} bounded by the $\Ld^\infty(\R^d)$-norm of $w$. The following result makes these assertions precise in a quantitative way in the spirit of Brézis and Gallouët~\cite{Brezis-Gallouet-80}. Item~(iii) can be found e.g.\@ in~\cite[Appendix]{Lin-Zhang-00} in a slightly different form, but we were unable to find items~(i) and~(ii) in the literature. The proof is postponed to Appendix~\ref{app:proofs-prel}. (By $-\triangle^{-1}$ we henceforth mean the convolution with the Green's kernel.)

\begin{lem}[Potential estimates in $\Ld^\infty$]\label{lem:singint-1}
Let $d\ge2$. For all $w\in C^\infty_c(\R^d)$ the following hold:\footnote{A direct adaptation of the proof further shows that in parts~(i) and~(ii) the $\Ld^\infty$-norms in the left-hand sides could be replaced by Hölder $C^\e$-norms with $\e\in[0,1)$: the exponents $d$ in the right-hand sides would then need to be all replaced by $(1-\e)^{-1}d>d$, and an additional multiplicative prefactor $(1-\e)^{-1}$ is further needed.}
\begin{enumerate}[(i)]
\item for all $1\le p<d<q\le\infty$, choosing $\theta\in(0,1)$ such that $\frac1d=\frac\theta p+\frac{1-\theta}q$, we have
\begin{align*}
\|\nabla\triangle^{-1}w\|_{\Ld^\infty}&\lesssim \big((1-\tfrac dq)\wedge(1-\tfrac pd)\big)^{-1+\frac1d}\,\|w\|_{\Ld^d}\bigg(1+\log\frac{\|w\|_{\Ld^p}^{\theta}\|w\|_{\Ld^q}^{1-\theta}}{\|w\|_{\Ld^d}}\bigg)^{1-\frac1d};
\end{align*}
\item if $ w=\Div\xi$ for $\xi\in C^\infty_c(\R^d)^d$, then, for all $d<q\le\infty$ and $1\le p<\infty$, we have
\begin{align*}
\|\nabla\triangle^{-1} w\|_{\Ld^\infty}&\lesssim (1-\tfrac dq)^{-1+\frac1d}\| w\|_{\Ld^{d}}\bigg(1+\log^+\frac{\| w\|_{\Ld^q}}{\| w\|_{\Ld^{d}}}\bigg)^{1-\frac1d}+p\|\xi\|_{\Ld^p};
\end{align*}
\item for all $0<s\le1$ and $1\le p<\infty$, we have
\begin{align*}
\|\nabla^2\triangle^{-1} w\|_{\Ld^\infty}&\lesssim s^{-1}\| w\|_{\Ld^\infty}\bigg(1+\log\frac{\|w\|_{C^s}}{\| w\|_{\Ld^\infty}}\bigg)+p\| w\|_{\Ld^p}.\qedhere
\end{align*}
\end{enumerate}
\end{lem}

In addition to the Sobolev regularity of solutions of~\eqref{eq:limeqn1}--\eqref{eq:limeqn2}, we study in the sequel their Hölder regularity as well, in the framework of the Besov spaces $C^s_*(\R^d):=B^s_{\infty,\infty}(\R^d)$ (see e.g.~\cite{BCD-11}). These spaces actually coincide with the usual Hölder spaces $C_b^s(\R^d)$ only for non-integer $s\ge0$ (for integer $s\ge0$ they are strictly larger than $W^{s,\infty}(\R^d)\supset C^s_b(\R^d)$ and coincide with the corresponding Zygmund spaces). The following potential theory estimates are then needed both in Sobolev and in Hölder-Zygmund spaces. As we were unable to find item~(ii) stated in the literature, a short proof is included in Appendix~\ref{app:proofs-prel}.

\begin{lem}[Potential estimates in Sobolev and Hölder-Zygmund spaces]\label{lem:pottheoryCsHs}
Let $d\ge2$. For all $w\in C^\infty_c(\R^d)$, the following hold:
\begin{enumerate}[(i)]
\item for all $s\ge0$,
\[\|\nabla\triangle^{-1}w\|_{H^s}\lesssim \|w\|_{\dot H^{-1}\cap H^{s-1}},\qquad\|\nabla^2\triangle^{-1}w\|_{H^s}\lesssim\|w\|_{H^s};\]
\item for all $s\in\R$,
\[\|\nabla\triangle^{-1}w\|_{C^s_*}\lesssim_s\|w\|_{\dot H^{-1}\cap C_*^{s-1}},\qquad\|\nabla^2\triangle^{-1}w\|_{C^s_*}\lesssim_{s}\|w\|_{\dot H^{-1}\cap C^{s}_*},\]
and for all $1\le p<d$ and $1\le q<\infty$,
\[\|\nabla\triangle^{-1}w\|_{C^s_*}\lesssim_{p,s}\|w\|_{\Ld^p\cap\Ld^\infty\cap C^{s-1}_*},\qquad \|\nabla^2\triangle^{-1}w\|_{C^s_*}\lesssim_{q,s}\|w\|_{\Ld^q\cap C^{s}_*},\]
\end{enumerate}
where the subscripts $s,p,q$ indicate the additional dependence of the multiplicative constants on an upper bound on $s$, $(d-p)^{-1}$, and $q$, respectively.
\end{lem}

We now state some global elliptic regularity results for the operator $-\Div(b\nabla)$ on the whole plane $\R^2$. Considering both the case of a right-hand side $f$ and that of a right-hand side in divergence form $\Div g$, we compare the properties of the corresponding solutions in terms of assumptions on $(f,g)$.
As no reference was found in the literature for this 2D setting, a detailed proof is included in Appendix~\ref{app:proofs-prel}.

\begin{lem}[2D global elliptic regularity]\label{lem:globellreg}
Let $b\in W^{1,\infty}(\R^2)^{2\times 2}$ be uniformly elliptic, that is, $\Id\le b\le \Lambda\Id$ for some $\Lambda<\infty$.
Given $f\in C^\infty_c(\R^2)$ and $g\in C^\infty_c(\R^2)^2$, we consider the decaying solutions $u$ and $v$ of the following equations in $\R^2$,
\[-\Div (b\nabla u)=f,\qquad\text{and}\qquad -\Div (b\nabla v)=\Div g.\]
The following properties hold.
\begin{enumerate}[(i)]
\item \emph{Meyers type estimates:}
There exists $2< p_0,q_0,r_0<\infty$ (depending only on an upper bound on $\Lambda$) such that for all $2< p\le p_0$, all $q_0\le q<\infty$, and all $r_0'\le r\le r_0$ with $\frac1{r_0}+\frac{1}{r_0'}=1$,
\[\|\nabla u\|_{\Ld^p}\le C_p\|f\|_{\Ld^{2p/(p+2)}},\qquad \|v\|_{\Ld^q}\le C_q\|g\|_{\Ld^{2q/(q+2)}},\qquad \|\nabla v\|_{\Ld^r}\le C\|g\|_{\Ld^r},\]
for some constant $C$ depending only on an upper bound on $\Lambda$, and for constants $C_p$ and $C_q$ further depending on an upper bound on $(p-2)^{-1}$ and $q$, respectively.
\item \emph{Sobolev regularity:}
For all $s\ge0$ we have
\[\|\nabla u\|_{H^s}\le C_s\|f\|_{\dot H^{-1}\cap H^{s-1}},\qquad \|\nabla v\|_{H^s}\le C_s\|g\|_{H^s},\]
where the constant $C_s$ depends only on an upper bound on $s$ and on $\|b\|_{W^{s,\infty}}$.
\item \emph{Schauder type estimate:} For all $s\in(0,1)$ we have
\[|\nabla u|_{C^s}\le C_s\|f\|_{\Ld^{2/(1-s)}},\qquad |v|_{C^s}\le C_s'\|g\|_{\Ld^{2/(1-s)}},\]
where the constant $C_s$ (resp.\@ $C_s'$) depends only on $s$ and on an upper bound on $\|b\|_{W^{s,\infty}}$ (resp.\@ on $s$ and on the modulus of continuity of $b$).
\end{enumerate}
In particular, we have
\[\|\nabla u\|_{\Ld^\infty}\le C\|f\|_{\Ld^1\cap\Ld^\infty},\qquad\|v\|_{\Ld^\infty}\le C'\|g\|_{\Ld^1\cap\Ld^\infty},\] where the constant $C$ (resp.~$C'$) depends only on an upper bound on $\|b\|_{W^{1,\infty}}$ (resp.~$\Lambda$).
\end{lem}

The interaction force $ v$ in equation~\eqref{eq:limeqn1VF} is defined by the values of $\curl v$ and $\Div (a v)$. The following result shows how $ v$ is controlled by such specifications. The proof is postponed to Appendix~\ref{app:proofs-prel}.

\begin{lem}\label{lem:reconstr}
Let $a,a^{-1}\in \Ld^{\infty}(\R^2)$. For all $\delta\omega,\delta\zeta\in\dot H^{-1}(\R^2)$, there exists a unique $\delta v\in\Ld^2(\R^2)^2$ such that $\curl\delta v=\delta\omega$ and $\Div(a\delta v)=\delta\zeta$. Moreover, for all $s\ge0$, if $a,a^{-1}\in W^{s,\infty}(\R^2)$ and $\delta\omega,\delta\zeta\in\dot H^{-1}\cap H^{s-1}(\R^2)$, we have
\[\|\delta v\|_{H^s}\le C\|\delta\omega\|_{\dot H^{-1}\cap H^{s-1}}+C\|\delta\zeta\|_{\dot H^{-1}\cap H^{s-1}},\]
where the constant $C$ depends only on an upper bound on $s$ and $\|(a,a^{-1})\|_{W^{s,\infty}}$.
\end{lem}

As emphasized in Remark~\ref{rem:sol}(i), weak solutions of the incompressible model~\eqref{eq:limeqn1} are rather defined via the vorticity formulation~\eqref{eq:limeqn1VF} in order to avoid compactness issues related to the pressure $ P$. Although this will not be used in the sequel, we quickly check that under mild regularity assumptions a weak solution $ v$ of~\eqref{eq:limeqn1} automatically also satisfies equation~\eqref{eq:limeqn1} in the distributional sense on $[0,T)\times\R^2$ for some pressure~$P$. The proof is postponed to Appendix~\ref{app:proofs-prel}.

\begin{lem}[Control on the pressure]\label{lem:pressure}
Let $\alpha,\beta\in\R$, $T>0$, $h\in W^{1,\infty}(\R^2)$, and $\Psi,\bar  v^\circ\in \Ld^\infty(\R^2)^2$. There exists some $2<q_0\lesssim1$ large enough (depending only on an upper bound on $\|h\|_{\Ld^\infty}$) such that the following holds: If~$ v\in\Ld^\infty_\loc([0,T);\bar  v^\circ+\Ld^2(\R^2)^2)$ is a weak solution of~\eqref{eq:limeqn1} on $[0,T)\times\R^2$ with $\omega:=\curl v\in\Ld^\infty_\loc([0,T);\Pc\cap\Ld^{q_0}(\R^2))$, then~$ v$ satisfies~\eqref{eq:limeqn1} in the distributional sense on $[0,T)\times\R^2$ for some pressure $ P\in\Ld^\infty_\loc([0,T);\Ld^{q_0}(\R^2))$.
\end{lem}

\section{Local-in-time existence of smooth solutions}\label{chap:local}

In this section, we prove the local-in-time existence of smooth solutions of~\eqref{eq:limeqn1}--\eqref{eq:limeqn2} as summarized in Theorem~\ref{th:mainloc}. Note that we choose to work here in the framework of Sobolev spaces, but the results could easily be adapted to Hölder spaces (compare indeed with Lemma~\ref{lem:conservHold}). We start with the non-degenerate case $\lambda>0$, using a standard iterative scheme as e.g.\@ in~\cite[Appendix~B]{Serfaty-15}.

\begin{prop}[Local existence, non-degenerate case]\label{prop:locexist}
Let $\alpha,\beta\in\R$, $\lambda>0$. Let $s>1$, and let $h,\Psi,\bar  v^\circ\in W^{s+1,\infty}(\R^2)^2$.
Let $ v^\circ\in\bar  v^\circ+H^{s+1}(\R^2)^2$ with $\omega^\circ:=\curl v^\circ$, $\bar\omega^\circ:=\curl \bar  v^\circ\in H^s(\R^2)$, and with either $\Div(a v^\circ)=\Div(a\bar  v^\circ)=0$ in the case~\eqref{eq:limeqn1}, or $\zeta^\circ:=\Div(a v^\circ)$, $\bar\zeta^\circ:=\Div(a\bar  v^\circ)\in H^s(\R^2)$ in the case~\eqref{eq:limeqn2}. Then there exists $T>0$ and a weak solution $ v\in\Ld^\infty([0,T);\bar  v^\circ+H^{s+1}(\R^2)^2)$ of~\eqref{eq:limeqn1} or of~\eqref{eq:limeqn2} on $[0,T)\times\R^2$ with initial data $ v^\circ$. Moreover, $T$ depends only on an upper bound on $|\alpha|$, $|\beta|$, $\lambda$, $\lambda^{-1}$, $s$, $(s-1)^{-1}$, $\|(h,\Psi,\bar  v^\circ)\|_{W^{s+1,\infty}}$, $\| v^\circ-\bar  v^\circ\|_{H^{s+1}}$, $\|(\omega^\circ,\bar\omega^\circ,\zeta^\circ,\bar\zeta^\circ)\|_{H^s}$.
\end{prop}

\begin{proof}
We focus on the compressible case~\eqref{eq:limeqn2}, the situation being similar and simpler in the incompressible case~\eqref{eq:limeqn1}. Let $s>1$. We set up the following iterative scheme: let $ v_0:= v^\circ$, $\omega_0:=\omega^\circ=\curl v^\circ$ and $\zeta_0:=\zeta^\circ=\Div(a v^\circ)$, and for all $n\ge0$ given $ v_n$, $\omega_n:=\curl v_n$, and $\zeta_n:=\Div(a v_n)$ we let $\omega_{n+1}$ and $\zeta_{n+1}$ solve on $\R^+\times\R^2$ the linear equations
\begin{gather}\label{eq:iterativescheme-1}
\partial_t\omega_{n+1}=\Div(\omega_{n+1}(\alpha(\Psi+ v_n)^\bot+\beta(\Psi+ v_n))),\quad \omega_{n+1}|_{t=0}=\omega^\circ,\\
\partial_t\zeta_{n+1}=\lambda\triangle\zeta_{n+1}-\lambda\Div(\zeta_{n+1}\nabla h)+\Div(a\omega_n(-\alpha(\Psi+ v_n)+\beta(\Psi+ v_n)^\bot)),\quad \zeta_{n+1}|_{t=0}=\zeta^\circ,\label{eq:iterativescheme-2}
\end{gather}
and we let $ v_{n+1}$ satisfy $\curl v_{n+1}=\omega_{n+1}$ and $\Div(a v_{n+1})=\zeta_{n+1}$. For all $n\ge0$, let also
\[t_n:=\sup\Big\{t\ge0:\|(\omega_n^t,\zeta_n^t)\|_{H^s}+\| v_n^t-\bar  v^\circ\|_{H^{s+1}}\le C_0\Big\},\]
for some $C_0\ge1$ to be suitably chosen (depending on the initial data), and let $T_0:=\inf_n t_n$. We show that this iterative scheme is well-defined with $T_0>0$, and that it converges to a solution of equation~\eqref{eq:limeqn2} on $[0,T_0)\times\R^2$.

We split the proof into two steps. In this proof, we use the notation $\lesssim$ for $\le$ up to a constant $C>0$ that depends only on an upper bound on $|\alpha|$, $|\beta|$, $\lambda$, $\lambda^{-1}$, $s$, $(s-1)^{-1}$, $\|(h,\Psi,\bar  v^\circ)\|_{W^{s+1,\infty}}$, $\| v^\circ-\bar  v^\circ\|_{H^{s+1}}$, $\|(\zeta^\circ,\bar\zeta^\circ)\|_{H^s}$, and $\|(\omega^\circ,\bar\omega^\circ)\|_{H^s}$.

\medskip
\noindent\step1 The iterative scheme is well-defined.

In this step, we show that for all $n\ge0$ the system~\eqref{eq:iterativescheme-1}--\eqref{eq:iterativescheme-2} admits a unique solution $(\omega_{n+1},\zeta_{n+1}, v_{n+1})$ with $\omega_{n+1}\in\Ld^\infty_\loc(\R^+; H^s(\R^2))$, $\zeta_{n+1}\in\Ld^\infty_\loc(\R^+;H^s(\R^2))$, and $ v_{n+1}\in \Ld^\infty_\loc(\R^+;\bar  v^\circ+H^{s+1}(\R^2)^2)$, and that moreover for a suitable choice of $1\le C_0\lesssim1$ we have $T_0\ge C_0^{-4}>0$.
We argue by induction. Let $n\ge0$ be fixed, and assume that $(\omega_{n},\zeta_{n}, v_n)$ is well-defined with $\omega_n\in\Ld^\infty_\loc(\R^+; H^s(\R^2))$, $\zeta_n\in\Ld^\infty_\loc(\R^+;H^s(\R^2))$, and $ v_n\in \Ld^\infty_\loc(\R^+;\bar  v^\circ+H^{s+1}(\R^2)^2)$. (For $n=0$, this is indeed trivial by assumption.)

We first study the equation for $\omega_{n+1}$. By the Sobolev embedding with $s>1$, $ v_n$ is Lipschitz-continuous, and by assumption $\Psi$ is also Lipschitz-continuous, hence the transport equation~\eqref{eq:iterativescheme-1} admits a unique continuous solution $\omega_{n+1}$, which automatically belongs to $\Ld^\infty_\loc(\R^+;\omega^\circ+\dot H^{-1}\cap H^s(\R^2))$ by Lemma~\ref{lem:katoponce}. More precisely, for all $t\ge0$, Lemma~\ref{lem:katoponce} together with the Sobolev embedding for $s>1$ yields
\begin{align*}
\partial_t\|\omega_{n+1}^t\|_{H^s}&\le C(1+\| v_n^t\|_{W^{1,\infty}})\|\omega_{n+1}^t\|_{H^s}+C \|\omega_{n+1}^t\|_{\Ld^\infty}\| v_n^t-\bar  v^\circ\|_{H^{s+1}}\\
&\le C(1+\| v_n^t-\bar  v^\circ\|_{H^{s+1}})\|\omega_{n+1}^t\|_{H^s}.
\end{align*}
Hence, for all $t\in[0,t_n]$, we obtain $\partial_t\|\omega_{n+1}^t\|_{H^s}\le CC_0\|\omega_{n+1}^t\|_{H^s}$, which proves
\[\|\omega_{n+1}^t\|_{H^s}\le e^{CC_0t}\|\omega^\circ\|_{H^s}\le Ce^{CC_0t}.\]
Noting that
\[\|\omega^\circ-\bar\omega^\circ\|_{\dot H^{-1}}\le \| v^\circ-\bar  v^\circ\|_{\Ld^2}\le C,\]
Lemma~\ref{lem:katoponce} together with the Sobolev embedding for $s>1$ also gives for all $t\ge0$,
\begin{align*}
\|\omega_{n+1}^t-\bar \omega^\circ\|_{\dot H^{-1}}\le C+\|\omega_{n+1}^t-\omega^\circ\|_{\dot H^{-1}}&\le C+Ct\|\omega_{n+1}\|_{\Ld^\infty_t\Ld^2}(1+\| v_n\|_{\Ld^\infty_t\Ld^\infty})\\
&\le C+Ct\|\omega_{n+1}\|_{\Ld^\infty_tH^s}(1+\| v_n-\bar  v^\circ\|_{\Ld^\infty_tH^s}),
\end{align*}
and hence, for all $t\in[0,t_n]$,
\[\|\omega_{n+1}^t-\bar\omega^\circ\|_{\dot H^{-1}}\le C(1+tC_0)e^{CC_0t}.\]

We now turn to $\zeta_{n+1}$. Equation~\eqref{eq:iterativescheme-2} (with $\lambda>0$) is a transport-diffusion equation and admits a unique solution $\zeta_{n+1}$, which belongs to $\Ld^\infty_\loc(\R^+;\zeta^\circ+\dot H^{-1}\cap H^s(\R^2))$ by Lemma~\ref{lem:parreg+tsp}(i)--(ii). More precisely, for all $t\ge0$, Lemma~\ref{lem:parreg+tsp}(i) yields for $s>1$
\begin{align}\label{eq:boundzetan1alm}
\|\zeta_{n+1}^t\|_{H^s}&\le Ce^{Ct}\big(\|\zeta^\circ\|_{H^s}+\|a\omega_n(\alpha(\Psi+ v_n)^\bot+\beta(\Psi+ v_n))\|_{\Ld^2_tH^s}\big)\nonumber\\
&\le Ce^{Ct}\big(1+t^{1/2}\|\omega_n\|_{\Ld^\infty_tH^s}(1+\| v_n-\bar  v^\circ\|_{\Ld^\infty_tH^s})\big),
\end{align}
where we have used Lemma~\ref{lem:katoponce-1} together with the Sobolev embedding to estimate the terms. Noting that
\[\|\zeta^\circ-\bar\zeta^\circ\|_{\dot H^{-1}}\le\|a v^\circ-a\bar  v^\circ\|_{\Ld^2}\le C,\]
Lemma~\ref{lem:parreg+tsp}(ii) together with the Sobolev embedding for $s>1$ also gives for all $t\ge0$,
\begin{align*}
\|\zeta^t_{n+1}-\bar\zeta^\circ\|_{\dot H^{-1}}\le C+\|\zeta^t_{n+1}-\zeta^\circ\|_{\dot H^{-1}}&\le C+Ce^{Ct}(\|\zeta^\circ\|_{\Ld^2}+\|a\omega_n(\alpha(\Psi+ v_n)^\bot+\beta(\Psi+ v_n))\|_{\Ld^2_t\Ld^2})\\
&\le Ce^{Ct}(1+t^{1/2}\|\omega_n\|_{\Ld^\infty_tH^s}(1+\| v_n-\bar  v^\circ\|_{\Ld^\infty_tH^s}).
\end{align*}
Combining this with~\eqref{eq:boundzetan1alm} yields for all $t\in[0,t_n]$,
\[\|\zeta_{n+1}^t\|_{H^s}+\|\zeta_{n+1}^t-\bar\zeta^\circ\|_{\dot H^{-1}}\le Ce^{Ct}\big(1+t^{1/2}C_0(1+C_0)\big)\le C(1+t^{1/2}C_0^2)e^{Ct}.\]

We finally turn to $ v_{n+1}$. By the above properties of $\omega_{n+1}$ and $\zeta_{n+1}$, Lemma~\ref{lem:reconstr} ensures that $ v_{n+1}$ is uniquely defined in $\Ld^\infty_\loc(\R^+;\bar  v^\circ+H^{s+1}(\R^2)^2)$ with $\curl( v_{n+1}^t-\bar  v^\circ)=\omega_{n+1}^t-\bar\omega^\circ$ and $\Div(a( v_{n+1}^t-\bar  v^\circ))=\zeta_{n+1}^t-\bar\zeta^\circ$ for all $t\ge0$. More precisely, Lemma~\ref{lem:reconstr} gives for all $t\in[0,t_n]$,
\begin{align*}
\| v_{n+1}^t-\bar  v^\circ\|_{H^{s+1}}&\le C\|\omega_{n+1}^t-\bar \omega^\circ\|_{\dot H^{-1}\cap H^s}+C\|\zeta_{n+1}^t-\bar \zeta^\circ\|_{\dot H^{-1}\cap H^s}\\
&\le C+C\|\omega_{n+1}^t-\bar \omega^\circ\|_{\dot H^{-1}}+C\|\omega_{n+1}^t\|_{H^s}+C\|\zeta_{n+1}^t-\bar\zeta^\circ\|_{\dot H^{-1}}+C\|\zeta_{n+1}^t\|_{H^s}\\
&\le C(1+tC_0+t^{1/2}C_0^2)e^{CC_0t}.
\end{align*}

Hence, we have proven that $(\omega_{n+1},\zeta_{n+1}, v_{n+1})$ is well-defined in the correct space, and moreover, combining all the previous estimates, we find for all $t\in[0,t_n]$,
\[\|(\omega_{n+1}^t,\zeta_{n+1}^t)\|_{H^s}+\| v_{n+1}^t-\bar  v^\circ\|_{H^{s+1}}\le C(1+tC_0+t^{1/2}C_0^2)e^{CC_0t}.\]
Therefore, choosing $C_0=1+3Ce^C\lesssim1$, we obtain for all $t\le t_{n}\wedge C_0^{-4}$,
\[\|(\omega_{n+1}^t,\zeta_{n+1}^t)\|_{H^s}+\| v_{n+1}^t-\bar  v^\circ\|_{H^{s+1}}\le C_0,\]
and thus $t_{n+1}\ge t_n\wedge C_0^{-4}$. The result follows by induction.

\medskip
\noindent\step2 Passing to the limit in the scheme.

In this step, we show that up to an extraction the iterative scheme $(\omega_n,\zeta_n, v_n)$ converges to a weak solution of equation~\eqref{eq:limeqn2} on $[0,T_0)\times\R^2$.

By Step~1, the sequences $(\omega_n)_n$ and $(\zeta_n)_n$ are bounded in $\Ld^\infty([0,T_0];H^s(\R^2)^2)$, and the sequence $( v_n)_n$ is bounded in $\Ld^\infty([0,T_0];\bar  v^\circ+H^{s+1}(\R^2)^2)$. Up to an extraction, we thus have $\omega_n\cvf{*}\omega$, $\zeta_n\cvf{*}\zeta$ in $\Ld^\infty([0,T_0];H^s(\R^2))$, and $ v_n\cvf{*} v$ in $\Ld^\infty([0,T_0];\bar  v^\circ+H^{s+1}(\R^2)^2)$. Comparing with equation~\eqref{eq:iterativescheme-1}, we deduce that $(\partial_t\omega_n)_n$ is bounded in $\Ld^{\infty}([0,T_0];H^{s-1}(\R^2))$.
Since by the Rellich theorem the space $H^{s}(U)$ is compactly embedded in $H^{s-1}(U)$ for any bounded domain $U\subset\R^2$, the Aubin-Simon lemma ensures that we have $\omega_n\to\omega$ strongly in $C^0([0,T_0];H^{s-1}_\loc(\R^2))$. This implies in particular $\omega_n v_n\to \omega v$ in the distributional sense, and hence we may pass to the limit in the weak formulation of equations~\eqref{eq:iterativescheme-1}--\eqref{eq:iterativescheme-2}, which yields $\curl v=\omega$, $\Div(a v)=\zeta$, with $\omega$ and $\zeta$ satisfying in the distributional sense on $[0,T_0)\times\R^2$,
\begin{gather*}
\partial_t\omega=\Div(\omega(\alpha(\Psi+ v)^\bot+\beta(\Psi+ v))),\quad \omega|_{t=0}=\omega^\circ,\\
\partial_t\zeta=\lambda\triangle\zeta-\lambda\Div(\zeta\nabla h)+\Div(a\omega(-\alpha(\Psi+ v)+\beta(\Psi+ v)^\bot)),\quad \zeta|_{t=0}=\zeta^\circ,
\end{gather*}
that is, the vorticity formulation~\eqref{eq:limeqn1VF}--\eqref{eq:limeqn2VF}. Let us quickly deduce that $ v$ is a weak solution of~\eqref{eq:limeqn2}.
From the above equations, we deduce $\partial_t\omega\in \Ld^\infty([0,T_0];\dot H^{-1}\cap H^{s-1}(\R^2))$ and $\partial_t\zeta\in \Ld^\infty([0,T_0];\dot H^{-1}\cap H^{s-2}(\R^2))$. Lemma~\ref{lem:reconstr} then implies $\partial_t v\in\Ld^\infty([0,T_0];H^{s-1}(\R^2)^2)$. We may then deduce that the quantity
\[V:=\partial_t v-\lambda\nabla(a^{-1}\zeta)+\alpha(\Psi+ v)\omega-\beta(\Psi+ v)^\bot\omega\]
belongs to $\Ld^\infty([0,T_0];\Ld^2(\R^2)^2)$ and satisfies $\curl V=\Div(aV)=0$ in the distributional sense. Using the Hodge decomposition in $\Ld^2(\R^2)^2$, we easily conclude $V=0$, hence $ v\in\Ld^\infty([0,T_0];\bar  v^\circ+H^{s+1}(\R^2)^2)$ is indeed a weak solution of~\eqref{eq:limeqn2} on $[0,T_0)\times\R^2$.
\end{proof}

We turn to the local-in-time existence of smooth solutions of~\eqref{eq:limeqn2} in the degenerate case $\lambda=0$.
The analysis of the iterative scheme needs to be carefully adapted in this case: in particular, $\omega$ and $ v$ are now on an equal footing with regard to regularity. Note that the proof only holds in the parabolic regime $\beta=0$.

\begin{prop}[Local existence, degenerate case]\label{prop:locexistdeg}
Let $\alpha\in\R$, $\beta=\lambda=0$. Let $s>2$, and let $h\in W^{s,\infty}(\R^2)$, $\Psi,\bar  v^\circ\in W^{s+1,\infty}(\R^2)^2$. Let $ v^\circ\in \bar  v^\circ+H^{s}(\R^2)^2$ with $\omega^\circ:=\curl v^\circ$, $\bar \omega^\circ:=\curl \bar  v^\circ\in H^s(\R^2)$ and $\zeta^\circ:=\Div(a v^\circ)$, $\bar \zeta^\circ:=\Div(a\bar  v^\circ)\in H^{s-1}(\R^2)$. Then, there exists $T>0$ and a weak solution $ v\in \Ld^\infty([0,T);\bar  v^\circ+ H^s(\R^2)^2)$ of~\eqref{eq:limeqn2} on $[0,T)\times\R^2$, with initial data $ v^\circ$. Moreover, $T$ depends only on an upper bound on $|\alpha|$, $s$, $(s-2)^{-1}$, $\|h\|_{W^{s,\infty}}$, $\|(\Psi,\bar  v^\circ)\|_{W^{s+1,\infty}}$, $\| v^\circ-\bar  v^\circ\|_{H^s}$, $\|(\omega^\circ,\bar \omega^\circ)\|_{H^s}$, and $\|(\zeta^\circ,\bar \zeta^\circ)\|_{H^{s-1}}$.
\end{prop}

\begin{proof}
We consider the same iterative scheme $(\omega_n,\zeta_n, v_n)$ as in the proof of Proposition~\ref{prop:locexist}, but with $\lambda=\beta=0$. Let $s>2$. For all $n\ge0$, let
\[t_n:=\sup\Big\{t\ge0:\|\omega_n^t\|_{H^s}+\|\zeta_n^t\|_{H^{s-1}}+\| v_n^t-\bar  v^\circ\|_{H^{s}}\le C_0\Big\},\]
for some $C_0\ge1$ to be suitably chosen (depending on initial data), and let $T_0:=\inf_n t_n$. In this proof, we use the notation $\lesssim$ for $\le$ up to a constant $C>0$ that depends only on an upper bound on $|\alpha|$, $s$, $(s-2)^{-1}$, $\|h\|_{W^{s,\infty}}$, $\|(\Psi,\bar  v^\circ)\|_{W^{s+1,\infty}}$, $\| v^\circ-\bar  v^\circ\|_{H^s}$, $\|(\zeta^\circ,\bar\zeta^\circ)\|_{H^{s-1}}$, and $\|(\omega^\circ,\bar\omega^\circ)\|_{H^s}$.

Just as in the proof of Proposition~\ref{prop:locexist}, we first need to show that this iterative scheme is well-defined and that $T_0>0$. We proceed by induction: let $n\ge0$ be fixed, and assume that $(\omega_n,\zeta_n, v_n)$ is well-defined with $\omega_n\in\Ld^\infty_\loc(\R^+; H^s(\R^2))$, $\zeta_n\in \Ld^\infty_\loc(\R^+;H^{s-1}(\R^2))$, and $ v_n\in\Ld^\infty_\loc(\R^+;\bar  v^\circ+H^{s}(\R^2)^2)$. (For $n=0$ this is indeed trivial by assumption.)

We first study $\zeta_{n+1}$. As $\lambda=0$, equation~\eqref{eq:iterativescheme-2} takes the form $\partial_t\zeta_{n+1}=-\alpha\Div(a\omega_n(\Psi+ v_n))$. Integrating this equation in time then yields
\begin{align*}
\|\zeta^t_{n+1}\|_{H^{s-1}}&\le\|\zeta^\circ\|_{H^{s-1}}+|\alpha|\int_0^t\|\omega^u_n(\Psi+ v^u_n)\|_{H^{s}}du\lesssim 1+t(1+\| v_n-\bar  v^\circ\|_{\Ld^\infty_tH^{s}})\|\omega_n\|_{\Ld^\infty_{t}H^{s}}.
\end{align*}
where we have used Lemma~\ref{lem:katoponce-1} together with the Sobolev embedding to estimate the last term. Similarly, noting that $\|\zeta^\circ-\bar\zeta^\circ\|_{\dot H^{-1}}\le \|a v^\circ-a\bar  v^\circ\|_{\Ld^2}\le C$, we find for $s>1$,
\begin{align*}
\|\zeta^t_{n+1}-\bar\zeta^\circ\|_{\dot H^{-1}}\le C+\|\zeta^t_{n+1}-\zeta^\circ\|_{\dot H^{-1}}&\le 
\|\zeta^\circ\|_{H^{s-1}}+|\alpha|\int_0^t\|\omega^u_n(\Psi+ v^u_n)\|_{\Ld^2}du\\
&\lesssim 1+t(1+\| v_n-\bar  v^\circ\|_{\Ld^\infty_tH^{s}})\|\omega_n\|_{\Ld^\infty_{t}H^{s}}.
\end{align*}
Hence we obtain for all $t\in[0,t_n]$,
\begin{align*}
\|\zeta^t_{n+1}\|_{H^{s-1}}+\|\zeta^t_{n+1}-\bar\zeta^\circ\|_{\dot H^{-1}}&\le C+ Ct(1+C_0)C_0\le C(1+tC_0^2).
\end{align*}

We now turn to the study of $\omega_{n+1}$. As $\beta=0$, equation~\eqref{eq:iterativescheme-1} takes the form $\partial_t\omega_{n+1}=\alpha\Div(\omega_{n+1}(\Psi+ v_n)^\bot)$. For all $t\ge0$, Lemma~\ref{lem:katoponce} together with the Sobolev embedding for $s>2$ then yields (here the choice $\beta=0$ is crucial, since otherwise the higher norm $\| v_n^t-\bar  v^\circ\|_{H^{s+1}}$ would appear in the right-hand side!)
\begin{align*}
\partial_t\|\omega_{n+1}^t\|_{H^s}&\lesssim (1+\| v_n^t\|_{W^{1,\infty}})\|\omega_{n+1}^t\|_{H^s}+\|\omega_{n+1}^t\|_{\Ld^\infty}\|\curl( v_n^t-\bar  v^\circ)\|_{H^s}+\|\omega_{n+1}^t\|_{W^{1,\infty}}\| v_n^t-\bar  v^\circ\|_{H^s}\nonumber\\
&\lesssim (1+\|\omega_n^t\|_{H^s}+\| v_n^t-\bar  v^\circ\|_{H^s})\|\omega_{n+1}^t\|_{H^s}.
\end{align*}
For all $t\in[0,t_n]$, this implies $\partial_t\|\omega_{n+1}^t\|_{H^s}\le C(1+2C_0)\|\omega_{n+1}^t\|_{H^s}$, and thus
\[\|\omega_{n+1}^t\|_{H^s}\le \|\omega^\circ\|_{H^s}e^{C(1+2C_0)t}\le Ce^{CC_0t}.\]
Moreover, noting that $\|\omega^\circ-\bar\omega^\circ\|_{\dot H^{-1}}\le\| v^\circ-\bar  v^\circ\|_{\Ld^2}\le C$, and applying Lemma~\ref{lem:katoponce} together with the Sobolev embedding, we obtain
\begin{align*}
\|\omega_{n+1}^t-\bar\omega^\circ\|_{\dot H^{-1}}&\le C+\|\omega_{n+1}^t-\omega^\circ\|_{\dot H^{-1}}\\
&\le C+Ct(1+\| v_n\|_{\Ld^\infty_t\Ld^\infty})\|\omega_{n+1}\|_{\Ld^\infty_t\Ld^2}\\
&\le C+Ct(1+\| v_n-\bar  v^\circ\|_{\Ld^\infty_tH^s})\|\omega_{n+1}\|_{\Ld^\infty_t\Ld^2},
\end{align*}
hence for all $t\in[0,t_n]$
\begin{align*}
\|\omega_{n+1}^t-\bar\omega^\circ\|_{\dot H^{-1}}&\le C+Ct(1+C_0)\|\omega_{n+1}\|_{\Ld^\infty_t\Ld^2}\le C+CC_0te^{CC_0t}.
\end{align*}

We finally turn to $ v_{n+1}$. By the above properties of $\omega_{n+1}$ and $\zeta_{n+1}$, Lemma~\ref{lem:reconstr} ensures that $ v_{n+1}$ is uniquely defined in $\Ld^\infty_\loc(\R^+;\bar  v^\circ+H^s(\R^2)^2)$, and for all $t\in[0,t_n]$ we have
\begin{align*}
\| v_{n+1}^t-\bar  v^\circ\|_{H^s}&\le C\|\omega_{n+1}^t-\bar\omega^\circ\|_{\dot H^{-1}\cap H^{s-1}}+C\|\zeta_{n+1}^t-\bar \zeta^\circ\|_{\dot H^{-1}\cap H^{s-1}}\\
&\le C+C\|\omega_{n+1}^t-\bar \omega^\circ\|_{\dot H^{-1}}+C\|\omega_{n+1}^t\|_{H^s}+C\|\zeta_{n+1}^t-\bar\zeta^\circ\|_{\dot H^{-1}}+C\|\zeta_{n+1}^t\|_{H^{s-1}}\\
&\le C(1+tC_0^2)e^{CC_0t}.
\end{align*}

Hence, we have proven that $(\omega_{n+1},\zeta_{n+1}, v_{n+1})$ is well-defined in the correct space, and moreover, combining all the previous estimates, we find for all $t\in[0,t_n]$
\[\|\omega_{n+1}^t\|_{H^s}+\|\zeta_{n+1}^t\|_{H^{s-1}}+\| v_{n+1}^t-\bar  v^\circ\|_{H^s}\le C(1+tC_0^2)e^{CC_0t}.\]
Therefore, choosing $C_0=1+2Ce^C\lesssim1$, we obtain for all $t\le t_n\wedge C_0^{-2}$
\[\|\omega_{n+1}^t\|_{H^s}+\|\zeta_{n+1}^t\|_{H^{s-1}}+\| v_{n+1}^t-\bar  v^\circ\|_{H^s}\le C_0,\]
and thus $t_{n+1}\ge t_n\wedge C_0^{-2}$. The conclusion now follows just as in the proof of Proposition~\ref{prop:locexist}.
\end{proof}

\section{Global existence}\label{chap:global}

As local existence is proven above in the framework of Sobolev spaces, the strategy for global existence consists in looking for a priori estimates on Sobolev norms. Since we are also interested in Hölder regularity of solutions, we establish a priori estimates on Hölder-Zygmund norms as well. As will be seen, the key ingredient is given by some a priori estimates for the vorticity $\omega$ in $\Ld^p(\R^2)$ with $p>1$.

\subsection{A priori estimates}\label{chap:apriori}

We start with the following elementary energy estimates. Note that in the degenerate case $\lambda=0$, the a priori estimate for $\zeta$ in $\Ld^2_\loc(\R^+;\Ld^2(\R^2))$ disappears, which is the main difficulty to establish a global result in that case.
Although we stick in the sequel to the framework of item~(iii), a priori estimates in slightly more general spaces are obtained in item~(ii) for the compressible model~\eqref{eq:limeqn2}.

\begin{lem}[Energy estimates]\label{lem:aprioriest}
Let $\lambda\ge0$, $\alpha\ge0$, $\beta\in\R$, $T>0$ and $\Psi\in W^{1,\infty}(\R^2)$. Let $ v^\circ\in\Ld^2_\loc(\R^2)^2$ be such that $\omega^\circ:=\curl v^\circ\in \Pc\cap\Ld^2_\loc(\R^2)$, and such that either $\Div(a v^\circ)=0$ in the case~\eqref{eq:limeqn1}, or $\zeta^\circ:=\Div(a v^\circ)\in\Ld^2_\loc(\R^2)$ in the case~\eqref{eq:limeqn2}. Let $ v\in\Ld^2_\loc([0,T)\times\R^2)^2$ be a weak solution of~\eqref{eq:limeqn1} or of~\eqref{eq:limeqn2} on $[0,T)\times\R^2$ with initial data $ v^\circ$. Set $\zeta:=0$ in the case~\eqref{eq:limeqn1}. Then the following properties hold.
\begin{enumerate}[(i)]
\item For all $t\in[0,T)$, we have $\omega^t\in\Pc(\R^2)$.
\item \emph{Localized energy estimate for~\eqref{eq:limeqn2}:} If $ v\in\Ld^2_\loc([0,T);\Ld^2_\uloc(\R^2)^2)$ is such that $\omega\in \Ld^\infty_\loc([0,T);\Ld^\infty(\R^2))$ and $\zeta\in\Ld^2_\loc([0,T);\Ld^2_\uloc(\R^2))$, then we have for all $t\in[0,T)$,
\begin{multline*}
\| v^t\|_{\Ld^2_\uloc}^2+\alpha\|| v|^2\omega\|_{\Ld^1_t\Ld^1_\uloc}+\lambda\|\zeta\|_{\Ld^2_t\Ld^2_\uloc}^2\le \begin{cases}Ce^{C(1+\lambda^{-1})t}\| v^\circ\|_{\Ld^2_\uloc}^2,&\text{if $\alpha=0$, $\lambda>0$};\\
C\alpha^{-1}\lambda^{-1}(e^{\lambda t}-1)+Ce^{\lambda t}\| v^\circ\|_{\Ld^2_\uloc}^2,&\text{if $\alpha>0$, $\lambda>0$;}\\
C\alpha^{-1}t+C\| v^\circ\|_{\Ld^2_\uloc}^2,&\text{if $\alpha>0$, $\lambda=0$;}
\end{cases}
\end{multline*}
where the constant $C$ depends only on an upper bound on $\alpha$, $|\beta|$, $\lambda$, $\|h\|_{W^{1,\infty}}$, $\|\Psi\|_{\Ld^\infty}$, and additionally on $\|\nabla\Psi\|_{\Ld^{\infty}}$ in the case $\alpha=0$.
\item \emph{Relative energy estimate for~\eqref{eq:limeqn1} and~\eqref{eq:limeqn2}:} If there is some $\bar  v^\circ\in W^{1,\infty}(\R^2)^2$ such that $ v^\circ\in\bar  v^\circ+\Ld^2(\R^2)^2$, $\bar\omega^\circ:=\curl \bar  v^\circ\in \Ld^2(\R^2)$, and such that either $\Div(a\bar  v^\circ)=0$ in the case~\eqref{eq:limeqn1}, or $\bar\zeta^\circ:=\Div(a\bar  v^\circ)\in \Ld^2(\R^2)$ in the case~\eqref{eq:limeqn2}, and if $ v\in\Ld^\infty_\loc([0,T);\bar  v^\circ+\Ld^2(\R^2))$, $\omega\in\Ld^\infty_\loc([0,T);\Ld^\infty(\R^2))$, $\zeta\in\Ld^2_\loc([0,T);\Ld^2(\R^2))$, then we have for all $t\in[0,T)$,
\begin{multline*}
\int_{\R^2}a| v^t-\bar  v^\circ|^2+\alpha\int_0^tdu\int_{\R^2} a| v^u-\bar  v^\circ|^2\omega^u+\lambda\int_0^tdu\int_{\R^2} a^{-1}|\zeta^u|^2\\
\le \begin{cases}Ct(1+\alpha^{-1})+\int_{\R^2} a| v^\circ-\bar  v^\circ|^2,&\text{in both cases~\eqref{eq:limeqn1} and~\eqref{eq:limeqn2}, with $\alpha>0$;}\\
e^{Ct}\big(1+\int_{\R^2} a| v^\circ-\bar  v^\circ|^2\big),&\text{in the case~\eqref{eq:limeqn1}, with $\alpha=0$}\\
C(e^{C(1+\lambda^{-1})t}-1)+e^{C(1+\lambda^{-1})t}\int_{\R^2} a| v^\circ-\bar  v^\circ|^2,&\text{in the case~\eqref{eq:limeqn2}, with $\alpha=0$, $\lambda>0$};\end{cases}
\end{multline*}
where the constant $C$ depends only on an upper bound on $\alpha$, $|\beta|$, $\lambda$, $\|h\|_{W^{1,\infty}}$, $\|(\Psi, \bar  v^\circ)\|_{\Ld^\infty}$, $\|\bar\zeta^\circ\|_{\Ld^2}$, and additionally on $\|\bar\omega^\circ\|_{\Ld^2}$ and $\|(\nabla\Psi,\nabla\bar  v^\circ)\|_{\Ld^{\infty}}$ in the case $\alpha=0$.\qedhere
\end{enumerate}
\end{lem}

\begin{proof}
Item~(i) is a standard consequence of the fact that $\omega$ satisfies a transport equation~\eqref{eq:limeqn1VF}.
It thus remains to check items (ii) and (iii). We split the proof into three steps.

\medskip
\noindent\step1 Proof of~(ii).

Let $ v$ be a weak solution of the compressible equation~\eqref{eq:limeqn2} as in the statement, and let also $C>0$ denote any constant as in the statement. We prove more precisely, for all $t\in[0,T)$ and $x_0\in\R^2$,
\begin{align}\label{eq:it2apriori}
&\int ae^{-|x-x_0|}| v^t|^2+\alpha\int_0^tdu\int ae^{-|x-x_0|}| v^u|^2\omega^u+\lambda\int_0^tdu\int a^{-1}e^{-|x-x_0|}|\zeta^u|^2\\
&\hspace{5cm}\le \begin{cases}e^{C(1+\lambda^{-1})t}\int ae^{-|x-x_0|}| v^\circ|^2,&\text{if $\alpha=0$, $\lambda>0$};\\
C\alpha^{-1}\lambda^{-1}(e^{\lambda t}-1)+e^{\lambda t}\int ae^{-|x-x_0|}| v^\circ|^2,&\text{if $\alpha>0$, $\lambda>0$;}\\
C\alpha^{-1}t+\int ae^{-|x-x_0|}| v^\circ|^2,&\text{if $\alpha>0$, $\lambda=0$.}
\end{cases}\nonumber
\end{align}
Item~(ii) directly follows from this, noting that
\[\|f\|_{\Ld^p_\uloc}^p\simeq\sup_{x_0\in\R^2}\int e^{-|x-x_0|}|f(x)|^pdx\]
holds for all $1\le p<\infty$. So it suffices to prove~\eqref{eq:it2apriori}.
Let $x_0\in\R^2$ be fixed, and denote by $\chi(x):=e^{-|x-x_0|}$ the exponential cut-off function centered at $x_0$.
From equation~\eqref{eq:limeqn2} we compute the following time derivative
\begin{align*}
\partial_t\int a\chi| v^t|^2&=2\int a\chi \big(\lambda\nabla(a^{-1}\zeta^t)-\alpha(\Psi+ v^t)\omega^t+\beta(\Psi+ v^t)^\bot\omega^t\big)\cdot  v^t,
\end{align*}
and hence, by integration by parts with $|\nabla\chi|\le\chi$,
\begin{align}
\partial_t\int a\chi | v^t|^2&=-2\lambda\int a^{-1}\chi |\zeta^t|^2-2\lambda\int \nabla\chi \cdot v^t\zeta^t-2\alpha\int a\chi | v^t|^2\omega^t+2\int a\chi (-\alpha\Psi+\beta\Psi^\bot)\cdot v^t\omega^t\nonumber\\
&\le-2\lambda\int a^{-1}\chi |\zeta^t|^2+2\lambda\int \chi |\zeta^t|| v^t|-2\alpha\int a\chi | v^t|^2\omega^t+2\int a\chi (-\alpha\Psi+\beta\Psi^\bot)\cdot v^t\omega^t.\label{eq:identityderenergy}
\end{align}
First consider the case $\alpha>0$. We may then bound the terms as follows, using the inequality $2xy\le x^2+y^2$,
\begin{align*}
\partial_t\int a\chi | v^t|^2&\le-2\lambda\int a^{-1}\chi |\zeta^t|^2+2\lambda \int\chi |\zeta^t| | v^t|-2\alpha\int a\chi | v^t|^2\omega^t+2C\int a\chi | v^t|\omega^t\\
&\le-\lambda\int a^{-1}\chi|\zeta^t|^2+\lambda\int a\chi| v^t|^2-\alpha\int a\chi| v^t|^2\omega^t+C\alpha^{-1}\underbrace{\int a\chi\omega^t}_{\le C}.
\end{align*}
As $\omega^t$ is nonnegative by item~(i), the first and third right-hand side terms are nonpositive, and the Grönwall inequality yields $\int a\chi| v^t|^2\le C\alpha^{-1}\lambda^{-1}(e^{\lambda t}-1)+e^{\lambda t}\int a\chi| v^\circ|^2$ (or $\int a\chi| v^t|^2\le C\alpha^{-1}t+\int a\chi| v^\circ|^2$ if $\lambda=0$). The above estimate may then be rewritten as follows,
\begin{align*}
\alpha\int a\chi| v^t|^2\omega^t+\lambda\int a^{-1}\chi|\zeta^t|^2&\le C\alpha^{-1}+\lambda\int a\chi| v^t|^2-\partial_t\int a\chi| v^t|^2\\
&\le C\alpha^{-1}e^{\lambda t}+\lambda e^{\lambda t}\int a\chi| v^\circ|^2-\partial_t\int a\chi| v^t|^2.
\end{align*}
Integrating in time yields
\begin{align*}
\alpha\int_0^tdu\int a\chi| v^t|^2\omega^u+\lambda\int_0^tdu\int a^{-1}\chi|\zeta^u|^2\le C\alpha^{-1}\lambda^{-1}(e^{-\lambda t}-1)+e^{\lambda t}\int a\chi| v^\circ|^2-\int a\chi| v^t|^2,
\end{align*}
so that~\eqref{eq:it2apriori} is proven for $\alpha>0$.
We now turn to the case $\alpha=0$, $\lambda>0$.
In that case, using the following Delort type identity, which holds here in $\Ld^\infty_\loc([0,T);W^{-1,1}_\loc(\R^2)^2)$,
\begin{gather*}
\omega v=a^{-1}\zeta v^\bot-\frac12| v|^2\nabla^\bot h-a^{-1}(\Div (aS_{ v}))^\bot,\qquad S_{ v}:= v\otimes v-\frac12 | v|^2\Id,
\end{gather*}
the estimate~\eqref{eq:identityderenergy} becomes, by integration by parts with $|\nabla\chi|\le\chi$,
\begin{multline*}
\partial_t\int a\chi| v^t|^2\le-2\lambda\int a^{-1}\chi|\zeta^t|^2+2\lambda\int\chi|\zeta^t|| v^t|-2\alpha\int a\chi| v^t|^2\omega^t+2\int\chi(-\alpha\Psi+\beta\Psi^\bot)\cdot ( v^t)^\bot\zeta^t\\
-\int a\chi(-\alpha\Psi+\beta\Psi^\bot)\cdot \nabla^\bot h| v^t|^2+2\int a\chi(\alpha\nabla\Psi^\bot+\beta\nabla\Psi):S_{ v^t}+2\int a\chi|\alpha\Psi^\bot+\beta\Psi||S_{ v^t}|,
\end{multline*}
and hence, noting that $|S_{ v^t}|\le C| v^t|^2$, and using the inequality $2xy\le x^2+y^2$,
\begin{align*}
\partial_t\int a\chi| v^t|^2&\le-2\lambda\int a^{-1}\chi|\zeta^t|^2+2C\int \chi|\zeta^t|| v^t|-2\alpha\int a\chi| v^t|^2\omega^t+C\int a\chi| v^t|^2\\
&\le-\lambda\int a^{-1}\chi|\zeta^t|^2+C(1+\lambda^{-1})\int a\chi| v^t|^2.
\end{align*}
The Grönwall inequality yields $\int a\chi| v^t|^2\le e^{C(1+\lambda^{-1})t}\int a\chi| v^\circ|^2$. The above estimate may then be rewritten as follows,
\begin{align*}
\lambda\int a^{-1}\chi|\zeta^t|^2&\le C(1+\lambda^{-1})\int a\chi| v^t|^2-\partial_t\int a\chi| v^t|^2\\
&\le C(1+\lambda^{-1})e^{C(1+\lambda^{-1})t}\int a\chi| v^\circ|^2-\partial_t\int a\chi| v^t|^2.
\end{align*}
Integrating in time, the result~\eqref{eq:it2apriori} is proven for $\alpha=0$. (Note that this proof cannot be adapted to the incompressible case~\eqref{eq:limeqn1}, due to the lack of a sufficiently good control on the pressure $ P$ in~\eqref{eq:limeqn1} in general.)

\medskip
\noindent\step2 Proof of~(iii) for~\eqref{eq:limeqn2}.

We denote by $C$ any positive constant as in the statement of item~(iii).
From equation~\eqref{eq:limeqn2}, we compute the following time derivative,
\begin{align*}
\partial_t\int a| v^t-\bar  v^\circ|^2&=2\int a(\lambda\nabla(a^{-1}\zeta^t)-\alpha(\Psi+ v^t)\omega^t+\beta(\Psi+ v^t)^\bot\omega^t)\cdot( v^t-\bar  v^\circ),
\end{align*}
or equivalently, integrating by parts and suitably regrouping the terms,
\begin{multline}\label{eq:decompapest}
\partial_t\int a| v^t-\bar  v^\circ|^2=-2\lambda\int a^{-1}|\zeta^t|^2+2\lambda\int a^{-1}\zeta^t\bar\zeta^\circ-2\alpha\int a| v^t-\bar  v^\circ|^2\omega^t\\
+2\int a(-\alpha(\Psi+\bar  v^\circ)+\beta(\Psi+\bar  v^\circ)^\bot)\cdot( v^t-\bar  v^\circ)\omega^t.
\end{multline}
First consider the case $\alpha>0$. We may then bound the terms as follows, using the inequality $2xy\le x^2+y^2$,
\begin{align*}
\partial_t\int a| v^t-\bar  v^\circ|^2&\le-2\lambda\int a^{-1}|\zeta^t|^2+2\lambda\int a^{-1}\zeta^t\bar\zeta^\circ-2\alpha\int a| v^t-\bar  v^\circ|^2\omega^t+2C\int a| v^t-\bar  v^\circ|\omega^t\\
&\le-\lambda\int a^{-1}|\zeta^t|^2+\lambda\int a^{-1}|\bar\zeta^\circ|^2-\alpha\int a| v^t-\bar  v^\circ|^2\omega^t+C\alpha^{-1},
\end{align*}
and the result of item~(iii) in the case $\alpha>0$ follows by integration.
%Applying the Grönwall inequality as in Step~1, item~(iii) easily follows from this in the case $\alpha>0$. [[même pas de Grönwall à faire dans ce cas: juste intégration...]]
We now turn to the case $\alpha=0$, $\lambda>0$. In that case, we rather rewrite~\eqref{eq:decompapest} in the form
\begin{multline*}
\partial_t\int a| v^t-\bar  v^\circ|^2=-2\lambda\int a^{-1}|\zeta^t|^2+2\lambda\int a^{-1}\zeta^t\bar\zeta^\circ-2\alpha\int a| v^t-\bar  v^\circ|^2\omega^t\\
+2\int a(-\alpha(\Psi+\bar  v^\circ)+\beta(\Psi+\bar  v^\circ)^\bot)\cdot( v^t-\bar  v^\circ)(\omega^t-\bar\omega^\circ)+2\int a(-\alpha(\Psi+\bar  v^\circ)+\beta(\Psi+\bar  v^\circ)^\bot)\cdot( v^t-\bar  v^\circ)\bar\omega^\circ,
\end{multline*}
so that, using the following Delort type identity, which holds here in $\Ld^\infty_\loc([0,T);W^{-1,1}_\loc(\R^2)^2)$,
\begin{gather*}
(\omega-\bar\omega^\circ)( v-\bar  v^\circ)=a^{-1}(\zeta-\bar\zeta^\circ)( v-\bar  v^\circ)^\bot-\frac12| v-\bar  v^\circ|^2\nabla^\bot h-a^{-1}(\Div (aS_{v-\bar  v^\circ}))^\bot,
\end{gather*}
we find by integration by parts
\begin{multline*}
\partial_t\int a| v^t-\bar  v^\circ|^2=-2\lambda\int a^{-1}|\zeta^t|^2+2\lambda\int a^{-1}\zeta^t\bar\zeta^\circ-2\alpha\int a| v^t-\bar  v^\circ|^2\omega^t\\
+2\int (-\alpha(\Psi+\bar  v^\circ)+\beta(\Psi+\bar  v^\circ)^\bot)\cdot( v^t-\bar  v^\circ)^\bot(\zeta^t-\bar\zeta^\circ)-\int a(-\alpha(\Psi+\bar  v^\circ)+\beta(\Psi+\bar  v^\circ)^\bot)\cdot\nabla^\bot h| v^t-\bar  v^\circ|^2\\
+2\int a\nabla(\alpha(\Psi+\bar  v^\circ)^\bot+\beta(\Psi+\bar  v^\circ)): S_{ v^t-\bar  v^\circ}+2\int a(-\alpha(\Psi+\bar  v^\circ)+\beta(\Psi+\bar  v^\circ)^\bot)\cdot( v^t-\bar  v^\circ)\,\bar\omega^\circ.
\end{multline*}
We may then bound the terms as follows, using the inequality $2xy\le x^2+y^2$,
\begin{align*}
\partial_t\int a| v^t-\bar  v^\circ|^2&\le-2\lambda\int a^{-1}|\zeta^t|^2+2\lambda\int a^{-1}|\zeta^t||\bar \zeta^\circ|-2\alpha\int a| v^t-\bar  v^\circ|^2\omega^t\\
&\qquad+C\int | v^t-\bar  v^\circ|\,|\zeta^t|+C\int | v^t-\bar  v^\circ|\,|\bar \zeta^\circ|+C\int a| v^t-\bar  v^\circ|^2+C\int a| v^t-\bar  v^\circ|\bar \omega^\circ\\
&\le-\lambda\int a^{-1}|\zeta^t|^2+C\int a^{-1}|\bar \zeta^\circ|^2+C\int |\bar\omega^\circ|^2+C(1+\lambda^{-1})\int a | v^t-\bar  v^\circ|^2.
\end{align*}
Item~(iii) in the case $\alpha=0$ then easily follows from the Grönwall inequality.

\medskip
\noindent\step3 Proof of~(iii) for~\eqref{eq:limeqn1}.

We denote by $C$ any positive constant as in the statement of item~(iii).
Noting that the identity $ v-\bar  v^\circ=a^{-1}\nabla^\bot(\Div a^{-1}\nabla)^{-1}(\omega-\bar\omega^\circ)$ follows from~\eqref{eq:Helmholtz} together with the constraint $\Div(a v)=\Div(a\bar v^\circ)=0$, and recalling that by assumption $ v-\bar  v^\circ\in\Ld^2_\loc([0,T);\Ld^2(\R^2)^2)$, we deduce
$\omega-\bar\omega^\circ\in\Ld^2_\loc([0,T);\dot H^{-1}(\R^2))$ and $(\Div a^{-1}\nabla)^{-1}(\omega-\bar\omega^\circ)\in\Ld^2_\loc([0,T);\dot H^{1}(\R^2))$.
In particular, this implies by integration by parts
\begin{align}\label{eq:ippomegav}
\int a| v-\bar  v^\circ|^2=\int a^{-1}|\nabla(\Div a^{-1}\nabla)^{-1}(\omega-\bar\omega^\circ)|^2=\int (\omega-\bar\omega^\circ) (-\Div a^{-1}\nabla)^{-1}(\omega-\bar\omega^\circ).
\end{align}
From equation~\eqref{eq:limeqn1VF}, we compute the following time derivative
\begin{eqnarray*}
\lefteqn{\partial_t\int(\omega-\bar \omega^\circ)(-\Div a^{-1}\nabla)^{-1}(\omega-\bar \omega^\circ)}\\
&=&2\int \nabla(\Div a^{-1}\nabla)^{-1}(\omega-\bar \omega^\circ)\cdot(\alpha(\Psi+ v)^\bot+\beta(\Psi+ v))\omega\\
&=&-2\int a( v-\bar  v^\circ)^\bot\cdot\Big(\alpha( v-\bar  v^\circ)^\bot+\beta( v-\bar  v^\circ)+\alpha(\Psi+\bar  v^\circ)^\bot+\beta(\Psi+\bar  v^\circ)\Big)\omega\\
&=&-2\alpha\int a| v-\bar  v^\circ|^2\omega-2\int a\omega( v-\bar  v^\circ)^\bot\cdot(\alpha(\Psi+\bar  v^\circ)^\bot+\beta(\Psi+\bar  v^\circ)).
\end{eqnarray*}
Combining this with identity~\eqref{eq:ippomegav}, we are now in position to conclude exactly as in Step~2 after equation~\eqref{eq:decompapest} (but with here $\zeta,\bar\zeta^\circ=0$).
\end{proof}

The energy estimates given by Lemma~\ref{lem:aprioriest} above are not strong enough to deduce global existence, and the key is to find an additional a priori $\Ld^p$-estimate for the vorticity $\omega$ with $p>1$.
We start with the following new result, based on a careful examination of the evolution of $\Ld^p$-norms of the vorticity. The argument can unfortunately not be adapted to the mixed-flow compressible case (that is, \eqref{eq:limeqn2} with $\alpha\ge0$, $\beta\ne0$), as it would require a too strong additional control on the norm $\|\zeta^t\|_{\Ld^{p+1}}$; this is why this case is excluded from our global results in Theorem~\ref{th:main}.

\begin{lem}[$\Ld^p$-estimates for vorticity]\label{lem:Lpvort}
Let $\lambda,\alpha\ge0$, $\beta\in\R$, $T>0$, $h,\Psi\in W^{1,\infty}(\R^2)$, $\bar  v^\circ\in \Ld^{\infty}(\R^2)^2$, and $ v^\circ\in\bar  v^\circ+\Ld^2(\R^2)^2$, with $\omega^\circ:=\curl v^\circ\in\Pc(\R^2)$, $\bar\omega^\circ:=\curl\bar  v^\circ\in\Pc\cap\Ld^\infty(\R^2)$. In the case~\eqref{eq:limeqn1}, also assume $\Div(a v^\circ)=\Div(a\bar  v^\circ)=0$. Let $ v\in\Ld^\infty_\loc([0,T);\bar  v^\circ+\Ld^2\cap\Ld^\infty(\R^2)^2)$ be a weak solution of~\eqref{eq:limeqn1} or of~\eqref{eq:limeqn2} on $[0,T)\times\R^2$ with initial data $ v^\circ$, and with $\omega:=\curl v\in\Ld^\infty_\loc([0,T);\Pc\cap\Ld^{\infty}(\R^2))$. For all $1<p\le\infty$ and $t\in[0,T)$,
\begin{enumerate}[(i)]
\item in the case~\eqref{eq:limeqn1} with $\alpha>0$, $\beta\in\R$, we have
\begin{align}\label{eq:boundomegaLp1p}
\bigg(\frac{\alpha(p-1)}2\bigg)^{1/p}\|\omega\|_{\Ld^{p+1}_t\Ld^{p+1}}^{1+1/p}+\|\omega^t\|_{\Ld^p}&\le \|\omega^\circ\|_{\Ld^p}+C_p,
\end{align}
where the constant $C_p$ depends only on an upper bound on $(p-1)^{-1}$, $\alpha$, $\alpha^{-1}$, $|\beta|$, $T$, $\|(h,\Psi)\|_{W^{1,\infty}}$, $\|(\bar  v^\circ,\bar\omega^\circ)\|_{\Ld^\infty}$, and on $\| v^\circ-\bar  v^\circ\|_{\Ld^2}$;
\item in both cases~\eqref{eq:limeqn1} and~\eqref{eq:limeqn2} with $\alpha\ge0$, $\beta=0$, $\lambda\ge0$, the same estimate~\eqref{eq:boundomegaLp1p} holds, where the constant $C_p=C$ depends only on an upper bound on $\alpha$, $T$, and on $\|(\curl\Psi)_-\|_{\Ld^{\infty}}$.
\qedhere
\end{enumerate}
\end{lem}

\begin{proof}
It is sufficient to prove the result for all $1<p<\infty$. In this proof, we use the notation $\lesssim$ for $\le$ up to a constant $C>0$ as in the statement but independent of $p$.
As explained at the end of Step~1, we may focus on item~(i), the other being much simpler. Set $\bar\theta^\circ:=\Div\bar  v^\circ$, $\theta:=\Div v$.
We repeatedly use the a priori estimate of Lemma~\ref{lem:aprioriest}(i) in the following interpolated form: for all $s\le q$ and $t\in[0,T)$,
\begin{align}\label{eq:interpolom}
\|\omega^t\|_{\Ld^s}\le \|\omega^t\|_{\Ld^q}^{q'/s'}\|\omega^t\|_{\Ld^1}^{1-q'/s'}=\|\omega^t\|_{\Ld^q}^{q'/s'}.
\end{align}
We split the proof into three steps.

\medskip
\noindent\step1 Preliminary estimate for $\omega$ (in case~(i)): for all $1<p<\infty$ and all $t\in[0,T)$,
\begin{align}\label{eq:apomegLp}
\alpha(p-1)\|\omega\|_{\Ld^{p+1}_t\Ld^{p+1}}^{p+1}+\|\omega^t\|_{\Ld^p}^p\le \|\omega^\circ\|_{\Ld^p}^p+C(p-1)(t^{1/p}+\| v\|_{\Ld^p_t\Ld^\infty})\|\omega\|_{\Ld^{p+1}_t\Ld^{p+1}}^{p-1/p}.
\end{align}

Using equation~\eqref{eq:limeqn1VF} and integrating by parts we may compute
\begin{align*}
\partial_t\int(\omega^t)^p&=p\int(\omega^t)^{p-1}\Div(\omega^t(\alpha(\Psi+ v^t)^\bot+\beta(\Psi+ v^t)))\\
&=-p(p-1)\int(\omega^t)^{p-1}\nabla\omega^t\cdot(\alpha(\Psi+ v^t)^\bot+\beta(\Psi+ v^t))\\
&=-(p-1)\int\nabla(\omega^t)^p\cdot(\alpha(\Psi+ v^t)^\bot+\beta(\Psi+ v^t))\\
&=(p-1)\int(\omega^t)^p\Div(\alpha(\Psi+ v^t)^\bot+\beta(\Psi+ v^t)).
\end{align*}
In case~(i), using the constraint $\Div(a v)=0$ to compute $\Div(\alpha v^\bot+\beta v)=-\alpha\omega+\beta\Div v=-\alpha\omega-\beta\nabla h\cdot v$, we find
\begin{align*}
(p-1)^{-1}\partial_t\int(\omega^t)^p&\le -\alpha\int(\omega^t)^{p+1}+C\int(\omega^t)^p(1+| v^t|)\le -\alpha\int(\omega^t)^{p+1}+C(1+\| v^t\|_{\Ld^\infty})\int(\omega^t)^p.
\end{align*}
By interpolation~\eqref{eq:interpolom}, we obtain
\begin{align*}
\alpha\int(\omega^t)^{p+1}+(p-1)^{-1}\partial_t\int(\omega^t)^p&\le C(1+\| v^t\|_{\Ld^\infty})\|\omega^t\|_{\Ld^{p+1}}^{p-1/p},
\end{align*}
and the result~\eqref{eq:apomegLp} directly follows by integration with respect to $t$ and by the Hölder inequality.
In case~(ii) we rather have $\Div(\alpha(\Psi+ v)^\bot+\beta(\Psi+ v))=-\alpha(\curl\Psi+\omega)$, and hence
\[\alpha\int(\omega^t)^{p+1}+(p-1)^{-1}\partial_t\int(\omega^t)^p\le \alpha\|(\curl\Psi)_-\|_{\Ld^\infty}\int(\omega^t)^p\le \alpha\|(\curl\Psi)_-\|_{\Ld^\infty}\Big(\int(\omega^t)^{p+1}\Big)^{1-1/p},\]
from which the conclusion~(ii) already follows.

\medskip
\noindent\step2 Preliminary estimate for $ v$ (in case~(i)): for all $2<q\le\infty$ and $t\in[0,T)$,
\begin{align}\label{eq:boundvincompr}
\| v^t\|_{\Ld^\infty}&\lesssim1+(1-2/q)^{-1/2}\|\omega^t\|_{\Ld^q}^{q'/2}\log^{1/2}(2+\|\omega^t\|_{\Ld^q}).
\end{align}

Let $2<q\le\infty$.
Note that $ v^t-\bar  v^\circ=\nabla^\bot\triangle^{-1}(\omega^t-\omega^\circ)+\nabla\triangle^{-1}(\theta^t-\bar \theta^\circ)$.
By Lemma~\ref{lem:singint-1}(i) for $w:=\omega^t-\bar\omega^\circ$ and Lemma~\ref{lem:singint-1}(ii) for $w:=\theta^t-\bar\theta^\circ=\Div( v^t-\bar  v^\circ)$, we find
\begin{align*}
\| v^t\|_{\Ld^\infty}&\le \|\bar  v^\circ\|_{\Ld^\infty}+\|\nabla\triangle^{-1}(\omega^t-\bar\omega^\circ)\|_{\Ld^\infty}+\|\nabla\triangle^{-1}(\theta^t-\bar \theta^\circ)\|_{\Ld^\infty}\\
&\lesssim1+(1-2/q)^{-1/2}\|\omega^t-\bar\omega^\circ\|_{\Ld^2}\log^{1/2}(2+\|\omega^t-\bar\omega^\circ\|_{\Ld^1\cap\Ld^q})\\
&\qquad+\|\theta^t-\bar \theta^\circ\|_{\Ld^2}\log^{1/2}(2+\|\theta^t-\bar \theta^\circ\|_{\Ld^2\cap\Ld^\infty})+\| v^t-\bar  v^\circ\|_{\Ld^2}.
\end{align*}
Noting that $\theta^t-\bar\theta^\circ=-\nabla h\cdot ( v^t-\bar  v^\circ)$, using interpolation~\eqref{eq:interpolom} in the form $\|\omega^t\|_{\Ld^2}\lesssim\|\omega^t\|_{\Ld^q}^{q'/2}$, and using the a priori estimates of Lemma~\ref{lem:aprioriest} in the form $\| v^t-\bar  v^\circ\|_{\Ld^2}+\|\omega^t\|_{\Ld^1}\lesssim1$, we obtain
\begin{align*}
\| v^t\|_{\Ld^\infty}&\lesssim(1-2/q)^{-1/2}\|\omega^t\|_{\Ld^q}^{q'/2}\log^{1/2}(2+\|\omega^t\|_{\Ld^q})+\log^{1/2}(2+\| v^t-\bar  v^\circ\|_{\Ld^\infty}),
\end{align*}
and the result follows, absorbing in the left-hand side the last norm of $ v$.

\medskip
\noindent\step3 Conclusion.

Let $1<p<\infty$. From~\eqref{eq:boundvincompr} with $q=p+1$, we deduce in particular
\[\| v^t\|_{\Ld^\infty}\lesssim1+(1-1/p)^{-1/2}\|\omega^t\|_{\Ld^{p+1}}^{\frac12(1+1/p)}\log^{1/2}(2+\|\omega^t\|_{\Ld^{p+1}})\lesssim (1-1/p)^{-1/2}\big(1+\|\omega^t\|_{\Ld^{p+1}}^{\frac34(1+1/p)}\big),\]
and hence, integrating with respect to $t$ and combining with~\eqref{eq:apomegLp},
\begin{align*}
\alpha(p-1)\|\omega\|_{\Ld^{p+1}_t\Ld^{p+1}}^{p+1}+\|\omega^t\|_{\Ld^p}^p&\le \|\omega^\circ\|_{\Ld^p}^p+Cp\big(1+\|\omega\|_{\Ld^{p+1}_t\Ld^{p+1}}^{\frac34(1+1/p)}\big)\|\omega\|_{\Ld^{p+1}_t\Ld^{p+1}}^{p-1/p}\\
&\le \|\omega^\circ\|_{\Ld^p}^p+Cp\|\omega\|_{\Ld^{p+1}_t\Ld^{p+1}}^{p-1/p}+Cp\|\omega\|_{\Ld^{p+1}_t\Ld^{p+1}}^{p+\frac34}.
\end{align*}
We may now absorb in the left-hand side the last two terms, to the effect of
\begin{align*}
\frac{\alpha(p-1)}2\|\omega\|_{\Ld^{p+1}_t\Ld^{p+1}}^{p+1}+\|\omega^t\|_{\Ld^p}^p&\le \|\omega^\circ\|_{\Ld^p}^p+C_p^p,
\end{align*}
where the constant $C_p$ further depends on an upper bound on $(p-1)^{-1}$, and the conclusion follows.
\end{proof}

The following result partially improves and completes the results of Lemma~\ref{lem:Lpvort} above in the case~\eqref{eq:limeqn1} with either $\alpha=0$ or $h$ constant (cf.\@ item~(ii) below), and in both cases~\eqref{eq:limeqn1} and~\eqref{eq:limeqn2} with $\alpha>0$, $\beta=0$ (cf.\@ item~(iii) below).
For that purpose, inspired by the work of Lin and Zhang~\cite{Lin-Zhang-00}, we exploit by ODE arguments the very particular structure of the transport equation~\eqref{eq:limeqn1VF}. In the parabolic case $\alpha>0$, $\beta=0$, note that we establish in item~(iii) an a priori $\Ld^p$-estimate for the vorticity $\omega$ through its initial $\Ld^1$-norm only,
which is the key for global existence results with vortex-sheet initial data.
While in~\cite{Lin-Zhang-00} for the simpler model~\eqref{eq:LZh} such an a priori estimate is achieved by explicitly integrating the evolution of the vorticity along characteristics, this explicit structure is lost for the more sophisticated models~\eqref{eq:limeqn1} and~\eqref{eq:limeqn2}, and a more subtle argument is required.

\begin{lem}[$\Ld^p$-estimates for vorticity, cont'd]\label{lem:Lpest}
Let $\lambda\ge0$, $\alpha\ge0$, $\beta\in\R$, $T>0$, and $h,\Psi, v^\circ\in W^{1,\infty}(\R^2)^2$, with $\omega^\circ:=\curl v^\circ\in \Pc\cap C^0(\R^2)$. Set $\zeta^\circ:=\Div(a v^\circ)$, and in the case~\eqref{eq:limeqn1} assume that $\Div(a v^\circ)=0$. Let $ v\in W^{1,\infty}_\loc([0,T);W^{1,\infty}(\R^2)^2)$ be a weak solution of~\eqref{eq:limeqn1} or of~\eqref{eq:limeqn2} on $[0,T)\times\R^2$ with initial data $ v^\circ$.
For all $1\le p\le \infty$ and $t\in[0,T)$, the following properties hold,
\begin{enumerate}[(i)]
\item in both cases~\eqref{eq:limeqn1} and~\eqref{eq:limeqn2}, without restriction on the parameters,
\begin{multline*}
\|\omega^t\|_{\Ld^p}\le\|\omega^\circ\|_{\Ld^p}\min\bigg\{\exp\Big(\frac{p-1}p\big(Ct+C|\beta|\|\zeta\|_{\Ld^1_t\Ld^\infty}+C|\beta|\|\nabla h\|_{\Ld^\infty}\| v\|_{\Ld^1_t\Ld^\infty}\big)\Big);\\
\exp\Big(\frac{p-1}p\big(C+Ct+C|\beta|\|\zeta\|_{\Ld^1_t\Ld^\infty}+C\alpha\|\nabla h\|_{\Ld^\infty}\| v\|_{\Ld^1_t\Ld^\infty}\big)\Big)\bigg\};
\end{multline*}
\item in the case~\eqref{eq:limeqn1} with either $\beta=0$ or $\alpha=0$ or $h$ constant, and in the case~\eqref{eq:limeqn2} with $\beta=0$, we have
\[\|\omega^t\|_{\Ld^p}\le Ce^{Ct}\|\omega^\circ\|_{\Ld^p};\]
\item given $\alpha>0$, in the case~\eqref{eq:limeqn1} with either $\beta=0$ or $h$ constant, and in the case~\eqref{eq:limeqn2} with $\beta=0$, we have
\[\|\omega^t\|_{\Ld^p}\le \Big((\alpha t)^{-1}+C\alpha^{-1}e^{Ct}\Big)^{1-1/p};\]
\end{enumerate}
where the constant $C$ depends only on an upper bound on $\alpha$, $|\beta|$, and on $\|(h,\Psi)\|_{W^{1,\infty}}$.
\end{lem}

\begin{rem}
In the context of item~(iii),
if we further assume $\Psi\equiv0$ (i.e. no forcing),
then the constant~$C$ in Step~2 of the proof below may then be set to $0$, so that we simply obtain, for all $1\le p<\infty$ and all $t>0$,
\[\|\omega^t\|_{\Ld^p}\le\bigg(\int|\omega^\circ|^p(1+\alpha t\omega^\circ)^{1-p}\bigg)^{1/p}\le (\alpha t)^{-(1-1/p)},\]
without additional exponential growth.
\end{rem}

\begin{proof}
We split the proof into two steps, and we use the notation $\lesssim$ for $\le$ up to a constant $C>0$ as in the statement.

\medskip
\noindent\step1 General bounds.

In this step, we prove (i) (from which (ii) directly follows, noting that choosing $a$ constant implies $\nabla h\equiv0$).
Let us consider the flow
\[\partial_t\psi^t(x)=-\alpha(\Psi+ v^t)^\bot(\psi^t(x))-\beta(\Psi+ v^t)(\psi^t(x)),\qquad \psi^t(x)|_{t=0}=x.\]
The Lipschitz assumptions ensure that $\psi$ is well-defined in $W_\loc^{1,\infty}([0,T);W^{1,\infty}(\R^2)^2)$. As $\omega$ satisfies the transport equation~\eqref{eq:limeqn1VF} with initial data $\omega^\circ\in C^0(\R^2)$, the method of propagation along characteristics yields
\[\omega^t(x)=\omega^\circ((\psi^{t})^{-1}(x))|\det\nabla(\psi^{t})^{-1}(x)|=\omega^\circ((\psi^{t})^{-1}(x))|\det\nabla\psi^{t}((\psi^t)^{-1}(x))|^{-1},\]
and hence for all $1\le p<\infty$ we have
\begin{align}\label{eq:rewriteLpnormdetpsi}
\int|\omega^t|^p&=\int |\omega^\circ((\psi^{t})^{-1}(x))|^p|\det\nabla\psi^t((\psi^{t})^{-1}(x))|^{-p}dx=\int |\omega^\circ(x)|^p|\det\nabla\psi^t(x)|^{1-p}dx,
\end{align}
while for $ P=\infty$,
\[\|\omega^t\|_{\Ld^\infty}\le\|\omega^\circ\|_{\Ld^\infty}\|(\det\nabla\psi^{t})^{-1}\|_{\Ld^\infty}.\]
Now let us examine this determinant more closely. By the Liouville-Ostrogradski formula,
\begin{align}\label{eq:detdevLiouv}
|\det\nabla \psi^{t}(x)|^{-1}&=\exp\bigg(\int_0^t\Div\Big(\alpha(\Psi+ v^u)^\bot+\beta(\Psi+ v^u)\Big)(\psi^u(x))du\bigg).
\end{align}
A simple computation gives
\begin{align}\label{eq:rewritedivchpvect}
\Div(\alpha( v^t)^\bot+\beta v^t)&=-\alpha\curl v^t+\beta \Div v^t=-\alpha\omega^t+\beta a^{-1}\zeta^t-\beta \nabla h\cdot v^t,
\end{align}
hence by non-negativity of $\omega$,
\begin{align*}
\Div(\alpha( v^t)^\bot+\beta v^t)&\le|\beta|\|a^{-1}\|_{\Ld^\infty}\|\zeta^t\|_{\Ld^\infty}+|\beta|\|\nabla h\|_{\Ld^\infty}\| v^t\|_{\Ld^\infty}.
\end{align*}
We then deduce from~\eqref{eq:detdevLiouv},
\begin{align*}
|\det\nabla \psi^{t}(x)|^{-1}&\le\exp\big(t\alpha\|\curl\Psi\|_{\Ld^\infty}+t|\beta|\|\Div\Psi\|_{\Ld^\infty}+|\beta|\|a^{-1}\|_{\Ld^\infty}\|\zeta\|_{\Ld^1_t\Ld^\infty}+|\beta|\|\nabla h\|_{\Ld^\infty}\| v\|_{\Ld^1_t\Ld^\infty}\big),
\end{align*}
and thus, combined with~\eqref{eq:rewriteLpnormdetpsi}, for all $1\le p\le\infty$,
\begin{multline}\label{eq:estomegaLp-1}
\|\omega^t\|_{\Ld^p}\le\|\omega^\circ\|_{\Ld^p}\exp\bigg(\frac{p-1}p\big(t\alpha\|\curl\Psi\|_{\Ld^\infty}+t|\beta|\|\Div\Psi\|_{\Ld^\infty}\\
+|\beta|\|a^{-1}\|_{\Ld^\infty}\|\zeta\|_{\Ld^1_t\Ld^\infty}+|\beta|\|\nabla h\|_{\Ld^\infty}\| v\|_{\Ld^1_t\Ld^\infty}\big)\bigg).
\end{multline}
On the other hand, noting that
\[\partial_t h(\psi^t(x))=-\nabla h(\psi^t(x))\cdot(\alpha(\Psi+ v^t)^\bot+\beta(\Psi+ v^t))(\psi^t(x)),\]
we may alternatively rewrite
\begin{align*}
\Div(\alpha( v^t)^\bot+\beta  v^t)(\psi^t(x))&=\big(-\alpha\omega^t+\beta a^{-1}\zeta^t-\beta \nabla h\cdot v^t\big)(\psi^t(x))\\
&=\partial_t h(\psi^t(x))+\big(-\alpha\omega^t+\beta a^{-1}\zeta^t-\alpha\nabla^\bot h\cdot v^t+\nabla h\cdot(\alpha\Psi^\bot+\beta\Psi)\big)(\psi^t(x)).
\end{align*}
Integrating this identity with respect to $t$ and using again the same formula for $|\det\nabla\psi^t|^{-1}$, we obtain
\begin{multline}\label{eq:estomegaLp-2}
\|\omega^t\|_{\Ld^p}\le\|\omega^\circ\|_{\Ld^p}\exp\bigg(\frac{p-1}p\big(t\alpha\|\curl\Psi\|_{\Ld^\infty}+t|\beta|\|\Div\Psi\|_{\Ld^\infty}+|\beta|\|a^{-1}\|_{\Ld^\infty}\|\zeta\|_{\Ld^1_t\Ld^\infty}\\
+2\|h\|_{\Ld^\infty}+t(\alpha+|\beta|)\|\nabla h\|_{\Ld^\infty}\|\Psi\|_{\Ld^\infty}+\alpha\|\nabla h\|_{\Ld^\infty}\|v\|_{\Ld^1_t\Ld^\infty}\big)\bigg).
\end{multline}
Combining~\eqref{eq:estomegaLp-1} and~\eqref{eq:estomegaLp-2}, the conclusion~(i) follows.

\medskip
\noindent\step2 Proof of~(iii).

It suffices to prove the result for any $1<p<\infty$. Let such a $p$ be fixed.
Assuming either $\beta=0$, or $\zeta\equiv0$ and $a$ constant, we deduce from~\eqref{eq:rewriteLpnormdetpsi}, \eqref{eq:detdevLiouv}, and~\eqref{eq:rewritedivchpvect},
\begin{align}\label{eq:rewriteLpnormdetpsiLiouv}
\int|\omega^t|^p&=\int |\omega^\circ(x)|^p\exp\Big((p-1)\int_0^t\Div\big(\alpha(\Psi+ v^u)^\bot+\beta(\Psi+ v^u)\big)(\psi^u(x))du\Big)dx\nonumber\\
&\le e^{C(p-1)t}\int |\omega^\circ(x)|^p\exp\Big(-\alpha(p-1)\int_0^t \omega^u(\psi^u(x))du\Big)dx.
\end{align}
Let $x$ be momentarily fixed, and set $f_x(t):=\omega^t(\psi^t(x))$. We need to estimate the integral $\int_0^tf_x(u)du$. For that purpose, we first compute $\partial_t f_x$: again using~\eqref{eq:rewritedivchpvect} (with either $\beta=0$, or $\zeta\equiv0$ and $a$ constant), we find
\begin{align*}
\partial_t f_x(t)&=\Div\big(\omega^t(\alpha(\Psi+ v^t)^\bot+\beta(\Psi+ v^t))\big)(\psi^t(x))-\nabla\omega^t(\psi^t(x))\cdot\big(\alpha(\Psi+ v^t)^\bot+\beta(\Psi+ v^t)\big)(\psi^t(x))\\
&=\omega^t(\psi^t(x))\Div\big(\alpha(\Psi+ v^t)^\bot+\beta(\Psi+ v^t))\big)(\psi^t(x))\\
&=-\alpha(\omega^t(\psi^t(x)))^2+\big(\!-\alpha\omega^t\curl \Psi+\beta\omega^t\Div\Psi\big)(\psi^t(x)),
\end{align*}
and hence
\[\partial_tf_x\ge -\alpha f_x^2-Cf_x.\]
We may then deduce $f_x\ge g_x$ pointwise, where $g_x$ satisfies
\[\partial_t g_x= -\alpha g_x^2- C g_x,\qquad g_x(0)=f_x(0)=\omega^\circ(x).\]
A direct computation yields
\[g_x(t)=\frac{Ce^{-Ct}\omega^\circ(x)}{C+\alpha(1-e^{-Ct})\omega^\circ(x)},\]
and hence
\[\int_0^tf_x(u)du\ge\int_0^tg_x(u)du=\alpha^{-1}\log\Big(1+\alpha C^{-1}(1-e^{-Ct})\omega^\circ(x)\Big).\]
Inserting this into~\eqref{eq:rewriteLpnormdetpsiLiouv}, we obtain for all $t>0$
\begin{align*}
\int|\omega^t|^p&\le e^{C(p-1)t}\int |\omega^\circ(x)|^p\Big(1+\alpha C^{-1}(1-e^{-Ct})\omega^\circ(x)\Big)^{1-p}dx\\
&\le \bigg(\frac{C\alpha^{-1}e^{Ct}}{1-e^{-Ct}}\bigg)^{p-1}\int |\omega^\circ(x)|dx=\bigg(\frac{C\alpha^{-1}e^{Ct}}{1-e^{-Ct}}\bigg)^{p-1}.
\end{align*}
The result~(iii) then follows from the obvious inequality $e^{Ct}(1-e^{-Ct})^{-1}\le e^{Ct}+1+(Ct)^{-1}$ for all $t>0$.
\end{proof}

The previous two lemmas establish uniform bounds on the vorticity $\omega$ in various regimes. As a preliminary to the propagation of regularity, we now show that any uniform bound on $\omega$ implies similar bounds on $ v$ and on the divergence $\zeta$.
In the incompressible case of equation~\eqref{eq:limeqn1}, this already follows from Step~2 of the proof of Lemma~\ref{lem:Lpvort} above, but more analysis is needed in the compressible case~\eqref{eq:limeqn2}.

\begin{lem}[Relative $\Ld^p$-estimates]\label{lem:Lpestdiv}
Let $\lambda>0$, $\alpha\ge0$, $\beta\in\R$, $T>0$, $h,\Psi,\bar  v^\circ\in W^{1,\infty}(\R^2)^2$, and $ v^\circ\in \bar  v^\circ+\Ld^2(\R^2)^2$, with $\omega^\circ:=\curl v^\circ\in\Pc(\R^2)$, $\bar\omega^\circ:=\curl \bar  v^\circ\in\Pc\cap\Ld^\infty(\R^2)$, and with either $\Div(a v^\circ)=\Div(a\bar  v^\circ)=0$ in the case~\eqref{eq:limeqn1}, or $\zeta^\circ:=\Div(a v^\circ)$, $\bar\zeta^\circ:=\Div(a\bar  v^\circ)\in\Ld^2\cap\Ld^\infty(\R^2)$ in the case~\eqref{eq:limeqn2}. Let $ v\in\Ld_\loc^\infty([0,T);\bar  v^\circ+\Ld^2(\R^2)^2)$ be a weak solution of~\eqref{eq:limeqn1} or of~\eqref{eq:limeqn2} on $[0,T)\times\R^2$ with initial data $ v^\circ$, and with $\omega:=\curl v\in\Ld^\infty([0,T];\Ld^\infty(\R^2))$. Then we have for all $t\in[0,T)$
\[\|\zeta^t\|_{\Ld^2\cap\Ld^\infty}\le C, \qquad\|\Div( v^t-\bar  v^\circ)\|_{\Ld^2\cap\Ld^\infty}\le C,\qquad\| v^t\|_{\Ld^\infty}\le C,\]
where the constant $C$ depends only on an upper bound on $\alpha$, $|\beta|$, $\lambda$, $\lambda^{-1}$, $T$, $\|h\|_{W^{1,\infty}}$, $\|(\Psi,\bar  v^\circ)\|_{\Ld^\infty}$, $\| v^\circ-\bar  v^\circ\|_{\Ld^2}$, $\|\bar\omega^\circ\|_{\Ld^1\cap\Ld^\infty}$, $\|(\zeta^\circ,\bar\zeta^\circ)\|_{\Ld^2\cap\Ld^\infty}$, $\|\omega\|_{\Ld^\infty_T\Ld^\infty}$, and additionally on $\|(\nabla\Psi,\nabla\bar  v^\circ)\|_{\Ld^\infty}$ (resp.\@ on $\alpha^{-1}$) in the case $\alpha=0$ (resp.\@ $\alpha>0$).
\end{lem}

\begin{proof}
In this proof, we use the notation $\lesssim$ for $\le$ up to a constant $C>0$ as in the statement, and we also set $\theta:=\Div v$ and $\bar\theta^\circ:=\Div\bar  v^\circ$. In the incompressible case~\eqref{eq:limeqn1} the conclusion follows from Step~2 of the proof of Lemma~\ref{lem:Lpvort} together with the identity $\Div v=-\nabla h\cdot v$.
We may thus focus on the case of the compressible equation~\eqref{eq:limeqn2}. We split the proof into three steps.

\medskip
\noindent\step1 Preliminary estimate for $ v$: for all $t\in[0,T)$,
\begin{align}
\| v^t\|_{\Ld^\infty}&\lesssim1+\|\theta^t-\bar\theta^\circ\|_{\Ld^2}\log^{1/2}(2+\|\theta^t-\bar\theta^\circ\|_{\Ld^2\cap\Ld^\infty}).\label{eq:v^tL^infty0}
\end{align}

Note that $ v^t-\bar  v^\circ=\nabla^\bot\triangle^{-1}(\omega^t-\bar \omega^\circ)+\nabla\triangle^{-1}(\theta^t-\bar \theta^\circ)$. By Lemma~\ref{lem:singint-1}(i)--(ii),
we may then estimate
\begin{multline*}
\| v^t-\bar  v^\circ\|_{\Ld^\infty}\le \|\nabla\triangle^{-1}(\omega^t-\bar \omega^\circ)\|_{\Ld^\infty}+\|\nabla\triangle^{-1}(\theta^t-\bar \theta^\circ)\|_{\Ld^\infty}\\
\lesssim\|\omega^t-\bar \omega^\circ\|_{\Ld^2}\log^{1/2}(2+\|\omega^t-\bar \omega^\circ\|_{\Ld^1\cap\Ld^\infty})+\|\theta^t-\bar \theta^\circ\|_{\Ld^2}\log^{1/2}(2+\|\theta^t-\bar \theta^\circ\|_{\Ld^2\cap\Ld^\infty})+\| v^t-\bar  v^\circ\|_{\Ld^2},
\end{multline*}
so that~\eqref{eq:v^tL^infty0} follows from the a priori estimates of Lemma~\ref{lem:aprioriest} (in the form $\| v^t-\bar  v^\circ\|_{\Ld^2}+\|\omega^t\|_{\Ld^1}\lesssim1$) and from the boundedness assumption $\|\omega\|_{\Ld_T^\infty\Ld^\infty}\lesssim1$.

\medskip
\noindent\step2 Boundedness of $\theta$: we prove $\|\theta^t-\bar\theta^\circ\|_{\Ld^2\cap\Ld^\infty}\lesssim1$ for all $t\in[0,T)$.

We start with the $\Ld^2$-estimate. As $\zeta$ satisfies the transport-diffusion equation~\eqref{eq:limeqn2VF}, Lemma~\ref{lem:parreg+tsp}(i) with $s=0$ leads to
\begin{align*}
\|\zeta^t\|_{\Ld^2}&\lesssim \|\zeta^\circ\|_{\Ld^2}+\|a\omega(-\alpha(\Psi+ v)+\beta(\Psi+ v)^\bot)\|_{\Ld^2_t\Ld^2}\lesssim 1+\|\omega\|_{\Ld^2_t\Ld^\infty}\| v-\bar  v^\circ\|_{\Ld^\infty_t\Ld^2}+\|\omega\|_{\Ld^2_t\Ld^2}\|(\Psi,\bar  v^\circ)\|_{\Ld^\infty},
\end{align*}
and hence $\|\zeta^t\|_{\Ld^2}\lesssim1$ follows from the a priori estimates of Lemma~\ref{lem:aprioriest} (in the form $\| v^t-\bar  v^\circ\|_{\Ld^2}+\|\omega^t\|_{\Ld^1}\lesssim1$) and the boundedness assumption for $\omega$. Similarly, for $\theta^t=a^{-1}\zeta^t-\nabla h\cdot v^t$, we deduce $\|\theta^t-\bar \theta^\circ\|_{\Ld^2}\lesssim1$. We now turn to the $\Ld^\infty$-estimate.
Lemma~\ref{lem:parreg+tsp}(iii) with $ P=q=s=\infty$ gives
\begin{align}
\|\zeta^t\|_{\Ld^\infty}&\lesssim \|\zeta^\circ\|_{\Ld^\infty}+\|a\omega(-\alpha(\Psi+ v)+\beta(\Psi+ v)^\bot)\|_{\Ld^\infty_t\Ld^\infty}\lesssim 1+\|\omega\|_{\Ld^\infty_t\Ld^\infty}(1+\| v\|_{\Ld^\infty_t\Ld^\infty}),\label{eq:estzetaLinftyfin}
\end{align}
or alternatively, for $\theta^t=a^{-1}\zeta^t-\nabla h\cdot v^t$,
\begin{align*}
\|\theta^t\|_{\Ld^\infty}&\lesssim 1+\| v^t\|_{\Ld^\infty}+\|\omega\|_{\Ld^\infty_t\Ld^\infty}(1+\|v\|_{\Ld^\infty_t\Ld^\infty}).
\end{align*}
Combining this estimate with the result of Step~1 yields
\begin{multline*}
\|\theta^t\|_{\Ld^\infty}\lesssim 1+\|\theta^t-\bar \theta^\circ\|_{\Ld^2}\log^{1/2}(2+\|\theta^t-\bar \theta^\circ\|_{\Ld^2\cap\Ld^\infty})\\
+\|\omega\|_{\Ld^\infty_t\Ld^\infty}(1+\|\theta-\bar \theta^\circ\|_{\Ld^\infty_t\Ld^2}\log^{1/2}(2+\|\theta-\bar \theta^\circ\|_{\Ld^\infty_t(\Ld^2\cap\Ld^\infty)})).
\end{multline*}
Now the boundedness assumption on $\omega$ and the $\Ld^2$-estimate for $\theta$ proven above reduce this expression to
\begin{align*}
\|\theta^t\|_{\Ld^\infty}&\lesssim \log^{1/2}(2+\|\theta\|_{\Ld^\infty_t\Ld^\infty}).
\end{align*}
Taking the supremum with respect to $t$, we may then conclude $\|\theta^t\|_{\Ld^\infty}\lesssim1$ for all $t\in[0,T)$.

\medskip
\noindent\step3 Conclusion.

By the result of Step~2, the estimate~\eqref{eq:v^tL^infty0} of Step~1 takes the form $\| v^t\|_{\Ld^\infty}\lesssim1$. The estimate~\eqref{eq:estzetaLinftyfin} of Step~2 then yields $\|\zeta^t\|_{\Ld^\infty}\lesssim1$, while the $\Ld^2$-estimate for $\zeta$ is already established in Step~2.
\end{proof}

\subsection{Propagation of regularity}\label{chap:propagation}

Since local existence is established in Section~\ref{chap:local} only for smooth enough data, it is necessary for the global existence result to first prove propagation of regularity along the flow. In this section, we show that propagation of regularity is a consequence of the boundedness of the vorticity $\omega$, which has itself been proven to hold in various regimes in Lemmas~\ref{lem:Lpvort} and~\ref{lem:Lpest} above.
We start with the propagation of Sobolev $H^s$-regularity.

\begin{lem}[Sobolev regularity]\label{lem:conservSob}
Let $s>1$. Let $\lambda>0$, $\alpha\ge0$, $\beta\in\R$, $T>0$, $h,\Psi,\bar  v^\circ\in W^{s+1,\infty}(\R^2)^2$, and $ v^\circ\in\bar  v^\circ+\Ld^2(\R^2)^2$, with $\omega^\circ:=\curl v^\circ,\bar\omega^\circ:=\curl \bar  v^\circ\in \Pc\cap H^s(\R^2)$, and with either $\Div(a v^\circ)=\Div(a\bar  v^\circ)=0$ in the case~\eqref{eq:limeqn1}, or $\zeta^\circ:=\Div(a v^\circ),\bar\zeta^\circ:=\Div(a\bar  v^\circ)\in H^s(\R^2)$ in the case~\eqref{eq:limeqn2}. Let $ v\in\Ld^\infty([0,T];\bar  v^\circ+H^{s+1}(\R^2)^2)$ be a weak solution of~\eqref{eq:limeqn1} or of~\eqref{eq:limeqn2} on $[0,T)\times\R^2$ with initial data $ v^\circ$. Then for all $t\in[0,T)$ we have
\[\|\omega^t\|_{H^s}\le C,\qquad\|\zeta^t\|_{H^s}\le C,
\qquad\| v^t-\bar  v^\circ\|_{H^{s+1}}\le C,\qquad\|\nabla v^t\|_{\Ld^\infty}\le C,\]
where the constant $C$ depends only on an upper bound on $s$, $(s-1)^{-1}$, $\alpha$, $|\beta|$, $\lambda$, $\lambda^{-1}$, $T$, $\|(h,\Psi,\bar  v^\circ)\|_{W^{s+1,\infty}}$, $\| v^\circ-\bar  v^\circ\|_{\Ld^2}$, $\|(\omega^\circ,\bar\omega^\circ,\zeta^\circ,\bar\zeta^\circ)\|_{H^s}$, $\|\omega\|_{\Ld^\infty_T\Ld^\infty}$, and additionally on $\alpha^{-1}$ in the case $\alpha>0$.
\end{lem}

\begin{proof}
We set $\theta:=\Div v$, $\bar\theta^\circ:=\Div\bar  v^\circ$. In this proof, we use the notation $\lesssim$ for $\le$ up to a constant $C>0$ as in the statement.
We focus on the compressible case~\eqref{eq:limeqn2}, the other case being similar and simpler. We split the proof into four steps.

\medskip
\noindent\step1 Time derivative of $\|\omega\|_{H^s}$: for all $s\ge0$ and $t\in[0,T)$,
\begin{align*}
\partial_t\|\omega^t\|_{H^s}&\lesssim(1+\|\nabla v^t\|_{\Ld^\infty})(1+\|\omega^t\|_{H^s})+\|\theta^t-\bar\theta^\circ\|_{H^s}.
\end{align*}

Lemma~\ref{lem:katoponce} with $\rho=\omega$, $w=\alpha(\Psi+ v)^\bot+\beta(\Psi+ v)$, and $W=\alpha(\Psi+\bar  v^\circ)^\bot+\beta(\Psi+\bar  v^\circ)$ yields
\begin{align}\label{eq:pre-step1-lem-conservSob}
\partial_t\|\omega^t\|_{H^s}&\lesssim(1+\|\nabla v^t\|_{\Ld^{\infty}})\|\omega^t\|_{H^s}+\| v^t-\bar v^\circ\|_{H^{s+1}}\|\omega^t\|_{\Ld^\infty}.
\end{align}
Using Lemma~\ref{lem:reconstr}, noting that $\|(\omega^t-\bar\omega^\circ,\theta^t-\bar\theta^\circ)\|_{\dot H^{-1}}\lesssim\| v^t-\bar  v^\circ\|_{\Ld^2}$, and using Lemma~\ref{lem:aprioriest}(iii) in the form $\| v^t-\bar  v^\circ\|_{\Ld^2}\lesssim1$, we obtain
\[\| v^t-\bar v^\circ\|_{H^{s+1}}\lesssim \|\omega^t-\bar\omega^\circ\|_{\dot H^{-1}\cap H^s}+\|\theta^t-\bar\theta^\circ\|_{\dot H^{-1}\cap H^s}\lesssim1+\|\omega^t-\bar\omega^\circ\|_{H^s}+\|\theta^t-\bar\theta^\circ\|_{H^s}.\]
Injecting this into~\eqref{eq:pre-step1-lem-conservSob}, the claim follows from Lemma~\ref{lem:Lpestdiv} and the boundedness assumption $\|\omega\|_{\Ld^\infty_T\Ld^\infty}\lesssim1$.

\medskip
\noindent\step2 Lipschitz estimate for $ v$: for all $s>1$ and $t\in[0,T)$,
\begin{align}
\|\nabla v^t\|_{\Ld^\infty}&\lesssim\log(2+\|\omega^t\|_{H^s}+\|\theta^t-\bar\theta^\circ\|_{H^s}).\label{eq:v^tnablaL^infty0}
\end{align}

Since $ v^t-\bar  v^\circ=\nabla^\bot\triangle^{-1}(\omega^t-\bar\omega^\circ)+\nabla\triangle^{-1}(\theta^t-\bar\theta^\circ)$, Lemma~\ref{lem:singint-1}(iii) yields, together with the Sobolev embedding of $H^s$ into a Hölder space for all $s>1$,
\begin{align*}
\|\nabla ( v^t-\bar  v^\circ)\|_{\Ld^\infty}&\le\|\nabla^2\triangle^{-1}(\omega^t-\bar\omega^\circ)\|_{\Ld^\infty}+\|\nabla^2\triangle^{-1}(\theta^t-\bar\theta^\circ)\|_{\Ld^\infty}\\
&\lesssim\|\omega^t-\bar\omega^\circ\|_{\Ld^\infty}\log(2+{\|\omega^t-\bar\omega^\circ\|_{H^s}})+\|\omega^t-\bar\omega^\circ\|_{\Ld^1}\\
&\hspace{3cm}+\|\theta^t-\bar\theta^\circ\|_{\Ld^\infty}\log(2+\|\theta^t-\bar\theta^\circ\|_{H^s})+\|\theta^t-\bar\theta^\circ\|_{\Ld^2},
\end{align*}
and the claim~\eqref{eq:v^tnablaL^infty0} then follows from Lemma~\ref{lem:aprioriest}(i), Lemma~\ref{lem:Lpestdiv}, and the boundedness assumption on $\omega$.

\medskip
\noindent\step3 Sobolev estimate for $\theta$: for all $s\ge0$ and $t\in[0,T)$,
\begin{align}\label{eq:v^tnablaL^infty00bis}
\|\theta^t-\bar\theta^\circ\|_{H^{s}}\lesssim1+\|\omega\|_{\Ld^\infty_tH^s}.
\end{align}

As $\zeta$ satisfies the transport-diffusion equation~\eqref{eq:limeqn2VF}, Lemma~\ref{lem:parreg+tsp}(i) gives for all $s\ge0$,
\begin{align*}
\|\zeta^t\|_{H^{s}}&\lesssim \|\zeta^\circ\|_{H^{s}}+\|a\omega(-\alpha(\Psi+ v)+\beta(\Psi+ v)^\bot)\|_{\Ld^2_tH^{s}}.
\end{align*}
Using Lemma~\ref{lem:katoponce-1} to estimate the right-hand side, we find for all $s\ge0$,
\begin{align*}
\|\zeta^t\|_{H^{s}}&\lesssim1+\|a\omega(-\alpha( v-\bar  v^\circ)+\beta( v-\bar  v^\circ)^\bot)\|_{\Ld^2_tH^{s}}+\|a\omega(-\alpha(\Psi+\bar  v^\circ)+\beta(\Psi+\bar  v^\circ)^\bot)\|_{\Ld^2_tH^{s}}\\
&\lesssim1+\|\omega\|_{\Ld^\infty_t\Ld^\infty}\| v-\bar  v^\circ\|_{\Ld^2_tH^s}+\|\omega\|_{\Ld^2_tH^s}\| v-\bar  v^\circ\|_{\Ld^\infty_t\Ld^\infty}\\
&\hspace{3cm}+\|\omega\|_{\Ld^2_t\Ld^2}(1+\|\bar  v^\circ\|_{W^{s,\infty}})+\|\omega\|_{\Ld^2_tH^s}(1+\|\bar  v^\circ\|_{\Ld^\infty}),
\end{align*}
and hence, by Lemma~\ref{lem:Lpestdiv} and the boundedness assumption on $\omega$,
\begin{align}\label{eq:v^tnablaL^infty0ter}
\|\zeta^t\|_{H^s}&\lesssim1+\|\omega\|_{\Ld^\infty_tH^s}+\| v-\bar  v^\circ\|_{\Ld^\infty_tH^s}.
\end{align}
Lemma~\ref{lem:reconstr} then yields for all $s\ge0$,
\begin{align*}
\|\zeta^t\|_{H^s}&\lesssim1+\|\omega\|_{\Ld^\infty_tH^s}+\|\omega-\bar\omega^\circ\|_{\Ld^\infty_t(\dot H^{-1}\cap H^{s-1})}+\|\zeta-\bar\zeta^\circ\|_{\Ld^\infty_t(\dot H^{-1}\cap H^{s-1})}.
\end{align*}
Noting that $\|(\omega-\bar\omega^\circ,\zeta-\bar\zeta^\circ)\|_{\dot H^{-1}}\lesssim\| v-\bar  v^\circ\|_{\Ld^2}$, and using Lemma~\ref{lem:aprioriest}(iii) in the form $\| v-\bar  v^\circ\|_{\Ld^2}\lesssim1$, we deduce
\begin{align*}
\|\zeta^t\|_{H^s}&\lesssim1+\|\omega\|_{\Ld^\infty_tH^s}+\|\zeta\|_{\Ld^\infty_tH^{s-1}}.
\end{align*}
Taking the supremum in time, we find by induction $\|\zeta\|_{\Ld^\infty_tH^s}\lesssim1+\|\omega\|_{\Ld^\infty_tH^s}+\|\zeta\|_{\Ld^\infty_t\Ld^2}$ for all $s\ge0$.
Recalling that Lemma~\ref{lem:Lpestdiv} gives $\|\theta^t-\bar \theta^\circ\|_{\Ld^2}\lesssim1$, and using the identity $\theta^t=a^{-1}\zeta^t-\nabla h\cdot v^t$, the claim~\eqref{eq:v^tnablaL^infty00bis} directly follows.

\medskip
\noindent\step4 Conclusion.
\nopagebreak

Combining the results of the three previous steps yields, for all $s>1$,
\begin{align*}
\partial_t\|\omega^t\|_{H^s}&\lesssim (1+\|\omega^t\|_{H^s})\log(2+\|\omega^t\|_{H^s}+\|\theta^t-\bar \theta^\circ\|_{H^s})+\|\theta^t-\bar \theta^\circ\|_{H^s}\\
&\lesssim(1+\|\omega\|_{\Ld^\infty_tH^s}) \log(2+\|\omega\|_{\Ld^\infty_tH^s}),
\end{align*}
hence
\begin{align*}
\partial_t\|\omega\|_{\Ld^\infty_tH^s}\le\sup_{[0,t]}\partial_t\|\omega\|_{H^s}\lesssim (1+\|\omega\|_{\Ld^\infty_tH^s})\log(2+\|\omega\|_{\Ld^\infty_tH^s}),
\end{align*}
and the Grönwall inequality then gives $\|\omega\|_{\Ld^\infty_tH^s}\lesssim1$. Combining this with~\eqref{eq:v^tnablaL^infty0}, \eqref{eq:v^tnablaL^infty00bis} and~\eqref{eq:v^tnablaL^infty0ter}, and recalling the identity $ v^t-\bar  v^\circ=\nabla^\bot\triangle^{-1}(\omega^t-\bar \omega^\circ)+\nabla\triangle^{-1}(\theta^t-\bar \theta^\circ)$, the conclusion follows.
\end{proof}

We now turn to the propagation of Hölder regularity. More precisely, we consider the Besov spaces $C^s_*(\R^2):=B^s_{\infty,\infty}(\R^2)$. Recall that these spaces coincide with the usual Hölder spaces $C^{s}_b(\R^2)$ only for non-integer $s\ge0$ (for integer $s>0$, they are strictly larger and coincide with the corresponding Zygmund spaces).

\begin{lem}[Hölder-Zygmund regularity]\label{lem:conservHold}
Let $s>0$. Let $\lambda>0$, $\alpha\ge0$, $\beta\in\R$, $T>0$, and $h,\Psi, v^\circ\in C^{s+1}_*(\R^2)^2$ with $\omega^\circ:=\curl v^\circ\in \Pc(\R^2)$, and with either $\Div(a v^\circ)=0$ in the case~\eqref{eq:limeqn1}, or $\zeta^\circ:=\Div(a v^\circ)\in \Ld^2(\R^2)$ in the case~\eqref{eq:limeqn2}. Let $ v\in\Ld^\infty([0,T];C^{s+1}_*(\R^2)^2)$ be a weak solution of~\eqref{eq:limeqn1} or of~\eqref{eq:limeqn2} on $[0,T)\times\R^2$ with initial data $ v^\circ$.
Then we have for all $t\in[0,T)$,
\[\|\omega^t\|_{C_*^s}\le C,\qquad \|\zeta^t\|_{C_*^{s}}\le C,
\qquad\| v^t\|_{C_*^{s+1}}\le C,\]
where the constant $C$ depends only on an upper bound on
$s$, $s^{-1}$, $\alpha$, $|\beta|$, $\lambda$, $\lambda^{-1}$, $T$, $\|(h,\Psi, v^\circ)\|_{C^{s+1}_*}$, $\|\zeta^\circ\|_{\Ld^2}$, $\|\omega\|_{\Ld^\infty_T\Ld^\infty}$, and additionally on $\alpha^{-1}$ in the case $\alpha>0$.
\end{lem}

\begin{proof}
We set $\theta:=\Div v$. In this proof, we use the notation $\lesssim$ for $\le$ up to a constant $C>0$ as in the statement.
We may focus on the compressible equation~\eqref{eq:limeqn2}, the other case being similar and simpler.
We split the proof into four steps, and make a systematic use of the standard Besov machinery as presented in~\cite{BCD-11}.

\medskip
\noindent\step1 Time derivative of $\|\omega^t\|_{C^s_*}$: for all $s>0$ and $t\in[0,T)$,
\begin{align*}
\partial_t\|\omega^t\|_{C^s_*}&\lesssim (1+\|\omega^t\|_{C^s_*})(1+\|\nabla v^t\|_{\Ld^\infty\cap C^{s-1}_*})+\|\theta^t\|_{C^s_*}.
\end{align*}

The transport equation~\eqref{eq:limeqn1VF} has the form $\partial_t\omega^t=\Div(\omega^tw^t)$ with $w^t=\alpha(\Psi+ v^t)^\bot+\beta(\Psi+ v^t)$. Arguing as in~\cite[Chapter~3.2]{BCD-11} (that is, similarly as in the proof of Lemma~\ref{lem:katoponce}, but using the corresponding commutator estimates in Besov spaces~\cite[Lemma~2.100]{BCD-11}), we obtain for all $s>0$,
\begin{align*}
\partial_t\|\omega^t\|_{C^s_*}\lesssim \|\omega^t\|_{C^s_*}\|\nabla w^t\|_{\Ld^\infty\cap C^{s-1}_*}+\|\omega^t\Div w^t\|_{C^s_*}.
\end{align*}
Using the usual product rules~\cite[Corollary~2.86]{BCD-11} for all $s>0$,
\begin{align*}
\partial_t\|\omega^t\|_{C^s_*}&\lesssim \|\omega^t\|_{C^s_*}\|\nabla w^t\|_{\Ld^\infty\cap C^{s-1}_*}+\|\omega^t\|_{\Ld^\infty}\|\Div w^t\|_{C^s_*}+\|\omega^t\|_{C^s_*}\|\Div w^t\|_{\Ld^\infty}\\
&\lesssim \|\omega^t\|_{C^s_*}(1+\|\nabla v^t\|_{\Ld^\infty\cap C^{s-1}_*})+\|\omega^t\|_{\Ld^\infty}(1+\|\omega^t\|_{C^s_*}+\|\theta^t\|_{C^s_*}),
\end{align*}
and the result follows from the boundedness assumption $\|\omega\|_{\Ld^\infty_T\Ld^\infty}\lesssim1$.

\medskip
\noindent\step2 Lipschitz estimate for $ v$: for all $s>0$ and $t\in[0,T)$,
\[\|\nabla v^t\|_{\Ld^\infty\cap C_*^{s-1}}\lesssim\|\omega^t\|_{C_*^{s-1}}+\|\theta^t\|_{C_*^{s-1}}+\log(2+\|\omega^t\|_{C^s_*}+\|\theta^t\|_{C^s_*}).\]

Since $ v^t- v^\circ=\nabla^\bot\triangle^{-1}(\omega^t-\omega^\circ)+\nabla\triangle^{-1}(\theta^t-\theta^\circ)$, Lemma~\ref{lem:pottheoryCsHs}(ii) yields for all $s\in\R$,
\begin{align*}
\|\nabla v^t\|_{C_*^{s-1}}\lesssim 1+\|\omega^t-\omega^\circ\|_{\dot H^{-1}\cap C_*^{s-1}}+\|\theta^t-\theta^\circ\|_{\dot H^{-1}\cap C_*^{s-1}},
\end{align*}
and thus, noting that $\|(\omega-\omega^\circ,\theta-\theta^\circ)\|_{\dot H^{-1}}\lesssim\| v- v^\circ\|_{\Ld^2}$, and using Lemma~\ref{lem:aprioriest}(iii) in the form $\| v- v^\circ\|_{\Ld^2}\lesssim1$,
\begin{align*}
\|\nabla v^t\|_{C_*^{s-1}}\lesssim 1+\|\omega^t\|_{C_*^{s-1}}+\|\theta^t\|_{C_*^{s-1}}.
\end{align*}
Arguing as in Step~2 of the proof of Lemma~\ref{lem:conservSob} further yields for all $s>0$,
\begin{align*}
\|\nabla v^t\|_{\Ld^\infty}\lesssim \log(2+\|\omega^t\|_{C^s_*}+\|\theta^t-\theta^\circ\|_{C^s_*}),
\end{align*}
and the result follows.

\medskip
\noindent\step3 Estimate for $\theta$: for all $s>0$ and $t\in[0,T)$,
\begin{align*}
\|\theta^t\|_{C^{s}_*}\lesssim1+\|\omega\|_{\Ld^\infty_tC^{s-1}_*}.
\end{align*}

As $\zeta$ satisfies the transport-diffusion equation~\eqref{eq:limeqn2VF}, we obtain for all $s>0$, arguing as in~\cite[Chapter~3.4]{BCD-11},
\begin{align*}
\|\zeta^t\|_{C^s_*}
&\lesssim\|\zeta^\circ\|_{C^s_*}+\|a\omega(-\alpha(\Psi+ v)+\beta(\Psi+ v)^\bot)\|_{\Ld^\infty_tC^{s-1}_*},
\end{align*}
and thus, by the usual product rules~\cite[Corollary~2.86]{BCD-11}, the boundedness assumption on $\omega$, and Lemma~\ref{lem:Lpestdiv}, we deduce for all $s>0$,
\begin{align*}
\|\zeta^t\|_{C^s_*}&\lesssim1+\|\omega\|_{\Ld^\infty_t(\Ld^\infty\cap C^{s-1}_*)}(1+\| v\|_{\Ld^\infty_t\Ld^\infty})+\|\omega\|_{\Ld^\infty_t\Ld^\infty}(1+\| v\|_{\Ld^\infty_t(\Ld^\infty\cap C^{s-1}_*)})\\
&\lesssim1+\|\omega\|_{\Ld^\infty_tC^{s-1}_*}+\| v\|_{\Ld^\infty_t C^{s-1}_*},
\end{align*}
or alternatively, in terms of $\theta^t=a^{-1}\zeta^t-\nabla h\cdot v^t$,
\begin{align*}
\|\theta^t\|_{C^s_*}&\lesssim\|\zeta^t\|_{\Ld^\infty\cap C^s_*}+\| v^t\|_{\Ld^\infty\cap C^s_*}\lesssim 1+\|\omega\|_{\Ld^\infty_tC^{s-1}_*}+\| v\|_{\Ld^\infty_t C^s_*}.
\end{align*}
Decomposing $ v^t- v^\circ=\nabla^\bot\triangle^{-1}(\omega^t-\omega^\circ)+\nabla\triangle^{-1}(\theta^t-\theta^\circ)$, using Lemma~\ref{lem:pottheoryCsHs}(ii), and again Lemma~\ref{lem:aprioriest}(iii) in the form $\|(\omega-\omega^\circ,\theta-\theta^\circ)\|_{\dot H^{-1}}\lesssim\| v- v^\circ\|_{\Ld^2}\lesssim1$, we find
\begin{align*}
\| v^t\|_{C_*^{s}}\lesssim 1+\|\omega^t-\omega^\circ\|_{\dot H^{-1}\cap C_*^{s-1}}+\|\theta^t-\theta^\circ\|_{\dot H^{-1}\cap C_*^{s-1}}\lesssim1+\|\omega^t\|_{C_*^{s-1}}+\|\theta^t\|_{C_*^{s-1}},
\end{align*}
and hence
\begin{align*}
\|\theta\|_{\Ld^\infty_tC^s_*}&\lesssim1+\|\omega\|_{\Ld^\infty_tC^{s-1}_*}+\|\theta\|_{\Ld^\infty_tC_*^{s-1}}.
\end{align*}
If $s\le1$, then we have $\|\cdot\|_{C^{s-1}_*}\lesssim\|\cdot\|_{\Ld^\infty}$, so that the above estimate, the boundedness assumption on $\omega$, and Lemma~\ref{lem:Lpestdiv} yield $\|\theta\|_{\Ld^\infty_t C^s_*}\lesssim1$. The result for $s>1$ then follows by induction.

\medskip
\noindent\step4 Conclusion.

Combining the results of the three previous steps yields, for all $s>0$,
\begin{align*}
\partial_t\|\omega\|_{\Ld^\infty_t C^s_*}\le\sup_{[0,t]}\partial_t\|\omega\|_{C^s_*}&\lesssim(1+ \|\omega\|_{\Ld^\infty_tC^s_*})\big(\|\omega\|_{\Ld^\infty_tC_*^{s-1}}+\|\theta\|_{\Ld^\infty_tC_*^{s-1}}+\log(2+\|\omega^t\|_{C^s_*}+\|\theta^t\|_{C^s_*})\big)+\|\theta\|_{\Ld^\infty_tC^s_*}\\
&\lesssim (1+\|\omega\|_{\Ld^\infty_tC^s_*})\big(\|\omega\|_{\Ld^\infty_tC_*^{s-1}}+\log(2+\|\omega\|_{\Ld^\infty_tC^s_*})\big).
\end{align*}
If $s\le1$, then we have $\|\cdot\|_{C^{s-1}_*}\lesssim\|\cdot\|_{\Ld^\infty}$, so that the above estimate and the boundedness assumption on $\omega$ yield $\partial_t\|\omega\|_{\Ld^\infty_t C^s_*}\lesssim (1+\|\omega\|_{\Ld^\infty_tC^s_*})\log(2+\|\omega\|_{\Ld^\infty_tC^s_*})$, hence $\|\omega\|_{\Ld^\infty_tC^s_*}\lesssim1$ by the Grönwall inequality. The conclusion for $s>1$ then follows by induction.
\end{proof}

\subsection{Global existence of solutions}\label{chap:globalpr}

With Lemma~\ref{lem:conservSob} at hand, together with the a priori bounds of Lemmas~\ref{lem:Lpvort} and~\ref{lem:Lpest} on the vorticity, it is now straightforward to deduce the following global existence result from the local existence statement of Proposition~\ref{prop:locexist}.

\begin{cor}[Global existence of smooth solutions]\label{cor:globexist1}
Let $s>1$. Let $\lambda>0$, $\alpha\ge0$, $\beta\in\R$, $h,\Psi,\bar  v^\circ\in W^{s+1,\infty}(\R^2)^2$, and $ v^\circ\in\bar  v^\circ+\Ld^2(\R^2)^2$, with $\omega^\circ:=\curl v^\circ,\bar\omega^\circ:=\curl \bar  v^\circ\in \Pc\cap H^{s}(\R^2)$, and with either $\Div(a v^\circ)=\Div(a\bar  v^\circ)=0$ in the case~\eqref{eq:limeqn1}, or $\zeta^\circ:=\Div(a v^\circ),\bar\zeta^\circ:=\Div(a\bar  v^\circ)\in H^s(\R^2)$ in the case~\eqref{eq:limeqn2}. Then,
\begin{enumerate}[(i)]
\item there exists a global weak solution $ v\in\Ld^\infty_\loc(\R^+;\bar  v^\circ+H^{s+1}(\R^2)^2)$ of~\eqref{eq:limeqn1} on $\R^+\times\R^2$ with initial data $ v^\circ$, and with $\omega:=\curl v\in \Ld^\infty_\loc(\R^+;\Pc\cap H^{ s}(\R^2))$;
\item if $\beta=0$, there exists a global weak solution $ v\in\Ld^\infty_\loc(\R^+;\bar  v^\circ+H^{ s+1}(\R^2)^2)$ of~\eqref{eq:limeqn2} on $\R^+\times\R^2$ with initial data $ v^\circ$, and with $\omega:=\curl v\in \Ld^\infty_\loc(\R^+;\Pc\cap H^{ s}(\R^2))$ and $\zeta:=\Div(a v)\in \Ld^\infty_\loc(\R^+; H^{ s}(\R^2))$.
\qedhere
\end{enumerate}
\end{cor}

\begin{proof}
We may focus on item~(ii), the first item being completely similar.
In this proof we use the notation $\simeq$ and $\lesssim$ for $=$ and $\le$ up to positive constants that depend only on an upper bound on $\alpha$, $\alpha^{-1}$, $|\beta|$, $\lambda$, $\lambda^{-1}$, $s$, $(s-1)^{-1}$, $\|(h,\Psi,\bar  v^\circ)\|_{W^{s+1,\infty}}$, $\| v^\circ-\bar  v^\circ\|_{\Ld^2}$, $\|(\omega^\circ,\bar\omega^\circ,\zeta^\circ,\bar\zeta^\circ)\|_{H^{s}}$.

Given $\bar  v^\circ\in W^{s+1,\infty}(\R^2)^2$ and $ v^\circ\in\bar  v^\circ+\Ld^2(\R^2)^2$ with $\omega^\circ,\bar\omega^\circ\in \Pc\cap H^{s}(\R^2)$ and $\zeta^\circ,\bar\zeta^\circ\in H^{s}(\R^2)$, Proposition~\ref{prop:locexist} gives a time $T>0$, $T\simeq1$, such that there exists a weak solution $ v\in\Ld^\infty([0,T);\bar  v^\circ+H^s(\R^2)^2)$ of~\eqref{eq:limeqn2} on $[0,T)\times\R^2$ with initial data $ v^\circ$. For all $t\in[0,T)$, Lemma~\ref{lem:Lpest}(ii) (with $\beta=0$) then gives $\|\omega^t\|_{\Ld^\infty}\lesssim1$, which implies by Lemma~\ref{lem:conservSob},
\begin{align*}
\|\omega^t\|_{H^s}+\|\zeta^t\|_{H^s}+\| v^t-\bar  v^\circ\|_{H^{s+1}}\lesssim1,
\end{align*}
and moreover by Lemma~\ref{lem:aprioriest}(i) we have $\omega^t\in\Pc(\R^2)$ for all $t\in[0,T)$. These a priori estimates show that the solution $ v$ can be extended globally in time.
\end{proof}

We now extend this global existence result beyond the setting of smooth initial data. We start with the following result for $\Ld^2$-data, which is easily deduced by approximation.

\begin{cor}[Global existence for $\Ld^2$-data]\label{cor:globexist2}
Let $\lambda>0$, $\alpha\ge0$, $\beta\in\R$, $h,\Psi\in W^{1,\infty}(\R^2)^2$. Let $\bar  v^\circ\in W^{1,\infty}(\R^2)^2$ be some reference map with $\bar \omega^\circ:=\curl\bar  v^\circ\in \Pc\cap H^{s}(\R^2)$ for some $s>1$, and with either $\Div(a\bar  v^\circ)=0$ in the case~\eqref{eq:limeqn1}, or $\bar\zeta^\circ:=\Div(a\bar  v^\circ)\in H^s(\R^2)$ in the case~\eqref{eq:limeqn2}. Let $ v^\circ\in \bar  v^\circ+\Ld^2(\R^2)^2$, with $\omega^\circ:=\curl v^\circ\in\Pc\cap\Ld^2(\R^2)$, and with either $\Div(a v^\circ)=0$ in the case~\eqref{eq:limeqn1}, or $\zeta^\circ:=\Div(a v^\circ)\in\Ld^2(\R^2)$ in the case~\eqref{eq:limeqn2}. Then,
\begin{enumerate}[(i)]
\item there exists a global weak solution $ v\in\Ld^\infty_\loc(\R^+;\bar  v^\circ+\Ld^2(\R^2)^2)$ of~\eqref{eq:limeqn1} on $\R^+\times\R^2$ with initial data $ v^\circ$, and with $ v\in\Ld^2_\loc(\R^+;\bar  v^\circ+H^1(\R^2)^2)$ and $\omega:=\curl v\in \Ld^\infty_\loc(\R^+;\Pc\cap \Ld^2(\R^2))$;
\item if $\beta=0$, there exists a global weak solution $ v\in\Ld^\infty_\loc(\R^+;\bar  v^\circ+\Ld^2(\R^2)^2)$ of~\eqref{eq:limeqn2} on $\R^+\times\R^2$ with initial data $ v^\circ$, and with $ v\in\Ld^2_\loc(\R^+;\bar  v^\circ+H^1(\R^2)^2)$, $\omega:=\curl v\in \Ld^\infty_\loc(\R^+;\Pc\cap \Ld^2(\R^2))$ and $\zeta:=\Div(a v)\in \Ld^2_\loc(\R^+;\Ld^2(\R^2))$.
\qedhere
\end{enumerate}
\end{cor}

\begin{proof}
We may focus on the case~(ii) (with $\beta=0$), the other case being exactly similar. In this proof we use the notation $\lesssim$ for $\le$ up to a positive constant that depends only on an upper bound on $\alpha$, $\alpha^{-1}$, $\lambda$, $(s-1)^{-1}$, $\|(h,\Psi,\bar  v^\circ)\|_{W^{1,\infty}}$, $\|(\bar\omega^\circ,\bar\zeta^\circ)\|_{H^s}$, $\| v^\circ-\bar  v^\circ\|_{\Ld^2}$, and $\|(\omega^\circ,\zeta^\circ)\|_{\Ld^2}$. We use the notation $\lesssim_t$ if it further depends on an upper bound on time $t$.

Let $\rho\in C^\infty_c(\R^2)$ with $\rho\ge0$, $\int\rho=1$, and $\rho(0)=1$. Define $\rho_\e(x):=\e^{-d}\rho(x/\e)$ for all $\e>0$, and set $\omega^\circ_\e:=\rho_\e\ast \omega^\circ$, $\bar\omega^\circ_\e:=\rho_\e\ast \bar\omega^\circ$, $\zeta^\circ_\e:=\rho_\e\ast \zeta^\circ$, $\bar\zeta^\circ_\e:=\rho_\e\ast \bar\zeta^\circ$, $a_\e:=\rho_\e\ast a$ and $\Psi_\e:=\rho_\e\ast\Psi$. For all $\e>0$, we have $\omega^\circ_\e$, $\bar\omega^\circ_\e\in\Pc\cap H^\infty(\R^2)$, $\zeta_\e^\circ$, $\bar \zeta_\e^\circ\in H^\infty(\R^2)$, and $a_\e$, $a_\e^{-1}$, $\Psi_\e\in C^\infty_b(\R^2)^2$. By construction, we have $a_\e\to a$, $a_\e^{-1}\to a^{-1}$, $\Psi_\e\to\Psi$ in $W^{1,\infty}(\R^2)$, $\bar\omega^\circ_\e-\bar\omega^\circ$, $\bar\zeta^\circ_\e-\bar\zeta^\circ\to0$ in $\dot H^{-1}\cap H^s(\R^2)$, and $\omega^\circ_\e-\omega^\circ$, $\zeta^\circ_\e-\zeta^\circ\to0$ in $\dot H^{-1}\cap \Ld^2(\R^2)$. The additional convergence in $\dot H^{-1}(\R^2)$ indeed follows from the following computation with Fourier transforms,
\begin{align*}
\|\omega_\e^\circ-\omega^\circ\|_{\dot H^{-1}}^2&=\int|\xi|^{-2}|\hat\rho(\e\xi)-1|^2|\hat\omega^\circ(\xi)|^2d\xi\le\e^2\|\nabla\hat\rho\|_{\Ld^\infty}^2\|\omega^\circ\|_{\Ld^2}^2,
\end{align*}
and similarly for $\bar\omega_\e^\circ$, $\zeta_\e^\circ$, and $\bar\zeta_\e^\circ$.
Lemma~\ref{lem:reconstr} then gives a unique $ v_\e^\circ\in  v^\circ+H^1(\R^2)^2$ and a unique $\bar  v_\e^\circ\in \bar  v^\circ+H^{s+1}(\R^2)^2$ such that $\curl v_\e^\circ=\omega_\e^\circ$, $\curl \bar  v_\e^\circ=\bar\omega_\e^\circ$, $\Div(a_\e v_\e^\circ)=\zeta_\e^\circ$, $\Div(a_\e \bar  v_\e^\circ)=\bar\zeta_\e^\circ$, and we have $ v_\e^\circ- v^\circ\to0$ in $H^1(\R^2)^2$ and $\bar  v_\e^\circ-\bar  v^\circ\to0$ in $H^{s+1}(\R^2)^2$. In particular, the assumption $\bar  v^\circ\in W^{1,\infty}(\R^2)^2$ yields by the Sobolev embedding with $s>1$, for $\e>0$ small enough,
\[\|\bar  v_\e^\circ\|_{W^{1,\infty}}\lesssim \|\bar  v_\e^\circ-\bar  v^\circ\|_{H^{s+1}}+\|\bar  v^\circ\|_{W^{1,\infty}}\lesssim1,\]
and the assumption $ v^\circ-\bar  v^\circ\in\Ld^2(\R^2)^2$ implies
\[\| v_\e^\circ-\bar  v_\e^\circ\|_{\Ld^2}\le \| v_\e^\circ-  v^\circ\|_{\Ld^2}+\| v^\circ-\bar  v^\circ\|_{\Ld^2}+\|\bar  v_\e^\circ-\bar  v^\circ\|_{\Ld^2}\lesssim1.\]

Corollary~\ref{cor:globexist1} then gives a solution $ v_\e\in \Ld^\infty_\loc(\R^+;\bar  v_\e^\circ+H^\infty(\R^2)^2)$ of~\eqref{eq:limeqn2} on $\R^+\times\R^2$ with initial data $ v^\circ_\e$, and with $(a,\Psi)$ replaced by $(a_\e,\Psi_\e)$.
Lemma~\ref{lem:aprioriest}(iii) and Lemma~\ref{lem:Lpest}(ii) (with $\beta=0$) give for all $t\ge0$,
\[\| v_\e-\bar  v_\e^\circ\|_{\Ld^\infty_t\Ld^2}+\|\zeta_\e\|_{\Ld^2_t\Ld^2}+\|\omega_\e\|_{\Ld^\infty_t\Ld^2}\lesssim_t1,\]
hence by Lemma~\ref{lem:reconstr}, together with the obvious estimate $\|(\omega_\e-\bar\omega_\e^\circ,\zeta_\e-\bar\zeta_\e^\circ)\|_{\dot H^{-1}}\lesssim\| v_\e-\bar  v_\e^\circ\|_{\Ld^2}$,
\begin{align*}
\| v_\e-\bar  v_\e^\circ\|_{\Ld^2_tH^1}&\lesssim \| v_\e-\bar  v_\e^\circ\|_{\Ld^2_t\Ld^2}+\|\zeta_\e-\bar\zeta_\e^\circ\|_{\Ld^2_t \Ld^2}+\|\omega_\e-\bar\omega_\e^\circ\|_{\Ld^2_t\Ld^2}
\lesssim_t1.
\end{align*}
As $\bar  v_\e^\circ$ is bounded in $H^1_\loc(\R^2)^2$, we deduce up to an extraction $ v_\e\cvf{}  v$ in $\Ld^2_\loc(\R^+;H^1_\loc(\R^2)^2)$, and also $\omega_\e\cvf{} \omega$, $\zeta_\e\cvf{} \zeta$ in $\Ld^2_\loc(\R^+;\Ld^2(\R^2))$, for some functions $ v,\omega,\zeta$.
Comparing equation~\eqref{eq:limeqn1VF} with the above estimates,
we deduce that $(\partial_t\omega_\e)_\e$ is bounded in $\Ld^{1}_\loc(\R^+;W^{-1,1}_\loc(\R^2))$.
Since by the Rellich theorem the space $\Ld^2(U)$ is compactly embedded in $H^{-1}(U)\subset W^{-1,1}(U)$ for any bounded domain $U\subset\R^2$, the Aubin-Simon lemma ensures that we have $\omega_\e\to\omega$ strongly in $\Ld^2_\loc(\R^+;H^{-1}_\loc(\R^2))$.
This implies $\omega_\e v_\e\to \omega v$ in the distributional sense. We may then pass to the limit in the weak formulation of equation~\eqref{eq:limeqn2}, and the result follows.
\end{proof}

We turn to the case of rougher initial data.
Using the a priori estimates of Lemmas~\ref{lem:Lpvort} and~\ref{lem:Lpest}(ii), we establish global existence for $\Ld^q$-data with $q>1$. In the parabolic regime $\alpha>0$, $\beta=0$, the finer a priori estimates of Lemma~\ref{lem:Lpest}(iii) further imply global existence for vortex-sheet data $\omega^\circ\in\Pc(\R^2)$.
Arguing by approximation, the main work consists in passing to the limit in the nonlinear term $\omega v$. For that purpose, as in~\cite{Lin-Zhang-00}, we make a crucial use of some compactness result due to Lions~\cite{Lions-98} in the context of the compressible Navier-Stokes equations. The conservative regime~(iv) below is however more subtle due to a lack of strong enough a priori estimates: only very weak solutions are then expected and obtained, and compactness needs to be carefully proven by hand.

\begin{prop}[Global existence for general data]\label{prop:globexist3}
Let $\lambda>0$, $\alpha\ge0$, $\beta\in\R$, and $h,\Psi\in W^{1,\infty}(\R^2)^2$. Let $\bar  v^\circ\in W^{1,\infty}(\R^2)^2$ be some reference map with $\bar \omega^\circ:=\curl\bar  v^\circ\in \Pc\cap H^{s}(\R^2)$ for some $s>1$, and with either $\Div(a\bar  v^\circ)=0$ in the case~\eqref{eq:limeqn1}, or $\bar\zeta^\circ:=\Div(a\bar  v^\circ)\in H^s(\R^2)$ in the case~\eqref{eq:limeqn2}. Let $ v^\circ\in \bar  v^\circ+\Ld^2(\R^2)^2$ with $\omega^\circ=\curl v^\circ\in\Pc(\R^2)$, and with either $\Div(a v^\circ)=0$ in the case~\eqref{eq:limeqn1}, or $\zeta^\circ:=\Div(a v^\circ)\in \Ld^2(\R^2)$ in the case~\eqref{eq:limeqn2}. Then the following hold.
\begin{enumerate}[(i)]
\item \emph{Case~\eqref{eq:limeqn2} with $\alpha>0$, $\beta=0$:}
There exists a weak solution $ v\in\Ld^\infty_\loc(\R^+;\bar  v^\circ+\Ld^2(\R^2)^2)$ on $\R^+\times\R^2$ with initial data $ v^\circ$, and with $\omega=\curl v\in \Ld^\infty(\R^+;\Pc(\R^2))$ and $\zeta=\Div(a v)\in \Ld^2_\loc(\R^+;\Ld^2(\R^2))$.
\item \emph{Case~\eqref{eq:limeqn1} with $\alpha>0$, and either $\beta=0$ or $a$ constant:}
There exists a weak solution $ v\in\Ld^\infty_\loc(\R^+;\bar  v^\circ+\Ld^2(\R^2)^2)$ on $\R^+\times\R^2$ with initial data $ v^\circ$, and with $\omega=\curl v\in \Ld^\infty(\R^+;\Pc(\R^2))$.
\item \emph{Case~\eqref{eq:limeqn1} with $\alpha>0$:}
If $\omega^\circ\in\Ld^q(\R^2)$ for some $q>1$, there exists a weak solution $ v\in\Ld^\infty_\loc(\R^+;\bar  v^\circ+\Ld^{2}(\R^2)^2)$ on $\R^+\times\R^2$ with initial data $ v^\circ$, and with $\omega=\curl v\in \Ld^\infty_\loc(\R^+;\Pc\cap\Ld^q(\R^2))$.
\item \emph{Case~\eqref{eq:limeqn1} with $\alpha=0$:}
If $\omega^\circ\in\Ld^q(\R^2)$ for some $q>1$, there exists a very weak solution $ v\in\Ld^\infty_\loc(\R^+;\bar  v^\circ+\Ld^{2}(\R^2)^2)$ on $\R^+\times\R^2$ with initial data $ v^\circ$, and with $\omega=\curl v\in \Ld^\infty_\loc(\R^+;\Pc\cap \Ld^q(\R^2))$. This is a weak solution whenever $q\ge4/3$.
\qedhere
\end{enumerate}
\end{prop}

\begin{proof}
We split the proof into three steps, first proving item~(i), then explaining how the argument has to be adapted to prove items~(ii) and~(iii), and finally turning to item~(iv).

\medskip
\noindent\step1 Proof of~(i).

\nopagebreak
In this step, we use the notation $\lesssim$ for $\le$ up to a positive constant that depends only on an upper bound on $\alpha$, $\alpha^{-1}$, $\lambda$, $\|(h,\Psi,\bar  v^\circ)\|_{W^{1,\infty}}$, $\|(\bar\omega^\circ,\bar \zeta^\circ)\|_{H^s}$, $\| v^\circ-\bar  v^\circ\|_{\Ld^2}$, and $\|\zeta^\circ\|_{\Ld^2}$. We use the notation $\lesssim_t$ (resp.\@ $\lesssim_{t,U}$) if it further depends on an upper bound on time $t$ (resp.\@ and on the size of $U\subset\R^2$).

Let $\rho\in C^\infty_c(\R^2)$ with $\rho\ge0$, $\int\rho=1$, $\rho(0)=1$, and $\rho|_{\R^2\setminus B_1}=0$, define $\rho_\e(x):=\e^{-d}\rho(x/\e)$ for all $\e>0$, and set $\omega^\circ_\e:=\rho_\e\ast \omega^\circ$, $\bar\omega^\circ_\e:=\rho_\e\ast \bar\omega^\circ$, $\zeta^\circ_\e:=\rho_\e\ast \zeta^\circ$, $\bar\zeta^\circ_\e:=\rho_\e\ast \bar\zeta^\circ$.
For all $\e>0$, we have $\omega^\circ_\e$, $\bar\omega^\circ_\e\in\Pc\cap H^\infty(\R^2)$, $\zeta_\e^\circ$, $\bar \zeta_\e^\circ\in H^\infty(\R^2)$. As in the proof of Corollary~\ref{cor:globexist2}, we have by construction $\bar\omega^\circ_\e-\bar\omega^\circ$, $\bar\zeta^\circ_\e-\bar\zeta^\circ\to0$ in $\dot H^{-1}\cap H^s(\R^2)$, and $\zeta^\circ_\e-\zeta^\circ\to0$ in $\dot H^{-1}\cap \Ld^2(\R^2)$. The assumption $ v^\circ-\bar  v^\circ\in\Ld^2(\R^2)^2$ further yields $\omega^\circ-\bar\omega^\circ\in\dot H^{-1}(\R^2)$, which implies
$\omega_\e^\circ-\bar\omega_\e^\circ\to\omega^\circ-\bar\omega^\circ$, hence $\omega_\e^\circ-\omega^\circ\to0$, in $\dot H^{-1}(\R^2)$.
Lemma~\ref{lem:reconstr} then gives a unique $ v_\e^\circ\in  v^\circ+\Ld^2(\R^2)^2$ and a unique $\bar  v_\e^\circ\in \bar  v^\circ+H^{s+1}(\R^2)^2$ such that $\curl v_\e^\circ=\omega_\e^\circ$, $\curl \bar  v_\e^\circ=\bar\omega_\e^\circ$, $\Div(a_\e v_\e^\circ)=\zeta_\e^\circ$, $\Div(a_\e \bar  v_\e^\circ)=\bar\zeta_\e^\circ$, and we have $ v_\e^\circ- v^\circ\to0$ in $\Ld^2(\R^2)^2$ and $\bar  v_\e^\circ-\bar  v^\circ\to0$ in $H^{s+1}(\R^2)^2$. In particular, arguing as in the proof of Corollary~\ref{cor:globexist2}, the assumption $\bar  v^\circ\in W^{1,\infty}(\R^2)^2$ yields $\|\bar  v_\e^\circ\|_{W^{1,\infty}}\lesssim1$ by the Sobolev embedding with $s>1$, and the assumption $ v^\circ-\bar  v^\circ\in\Ld^2(\R^2)^2$ implies $\| v_\e^\circ-\bar  v_\e^\circ\|_{\Ld^2}\lesssim1$.

Corollary~\ref{cor:globexist2} then gives a global weak solution $ v_\e\in \Ld^\infty_\loc(\R^+;\bar  v_\e^\circ+\Ld^2(\R^2)^2)$ of~\eqref{eq:limeqn2} on $\R^+\times\R^2$ with initial data $ v^\circ_\e$, and
Lemma~\ref{lem:aprioriest}(iii) yields for all $t\ge0$,
\begin{align}\label{eq:boundaprioriSt1}
\| v_\e-\bar  v_\e^\circ\|_{\Ld^\infty_t\Ld^2}+\|\zeta_\e\|_{\Ld^2_t\Ld^2}\lesssim_t1,
\end{align}
while Lemma~\ref{lem:Lpest}(iii) (with $\beta=0$) yields after time integration for all $1\le p<2$,
\[\|\omega_\e\|_{\Ld^p_t\Ld^p}\lesssim \bigg(\int_0^t\big(u^{1-p}+e^{Cu}\big)\,du\bigg)^{1/p}\lesssim_t(2-p)^{-1/p}.\]
Using this last estimate for $ P=3/2$ and $11/6$, and combining it with Lemma~\ref{lem:aprioriest}(i) in the form $\|\omega_\e\|_{\Ld^\infty_t\Ld^1}\le1$, we deduce by interpolation
\[\|\omega_\e\|_{\Ld^{2}_t(\Ld^{4/3}\cap\Ld^{12/7})}\lesssim_t1.\]
Now we need to prove more precise estimates on $v_\e$. First recall the identity
\begin{align}\label{eq:decompveps}
 v_\e= v_{\e,1}+ v_{\e,2},\qquad  v_{\e,1}:=\nabla^\bot\triangle^{-1}\omega_\e,\qquad  v_{\e,2}:=\nabla\triangle^{-1}\Div v_\e.
\end{align}
On the one hand, as $\omega_\e$ is bounded in $\Ld^{2}_\loc(\R^+;\Ld^{4/3}\cap\Ld^{12/7}(\R^2))$, we deduce from Riesz potential theory that $ v_{\e,1}$ is bounded in $\Ld^{2}_\loc(\R^+;\Ld^{4}\cap\Ld^{12}(\R^2)^2)$, and we deduce from the Calderón-Zygmund theory that $\nabla v_{\e,1}$ is bounded in $\Ld^2_\loc(\R^+;\Ld^{4/3}(\R^2))$.
On the other hand, decomposing
\[ v_{\e,2}=\nabla\triangle^{-1}\Div( v_\e-\bar  v_\e^\circ)+\bar  v_\e^\circ-\nabla^\bot\triangle^{-1}\bar\omega_\e^\circ,\]
noting that $ v_\e-\bar  v_\e^\circ$ is bounded in $\Ld^\infty_\loc(\R^+;\Ld^2(\R^2)^2)$ (cf.~\eqref{eq:boundaprioriSt1}), that $\bar  v_\e^\circ$ is bounded in $\Ld^2_\loc(\R^2)^2$, and that $\|\nabla\triangle^{-1}\bar\omega_\e^\circ\|_{\Ld^2}\lesssim\|\bar\omega_\e^\circ\|_{\Ld^1\cap\Ld^\infty}\lesssim1$ (cf.\@ Lemma~\ref{lem:singint-1}), we deduce that $ v_{\e,2}$ is bounded in $\Ld^\infty_\loc(\R^+;\Ld^2_\loc(\R^2)^2)$. Further, decomposing
\[ v_{\e,2}=\nabla\triangle^{-1}(a^{-1}(\zeta_\e-\bar\zeta_\e^\circ))-\nabla\triangle^{-1}(\nabla h\cdot( v_\e-\bar  v_\e^\circ))+\bar  v_\e^\circ-\nabla^\bot\triangle^{-1}\bar\omega_\e^\circ,\]
we easily check that $\nabla v_{\e,2}$ is bounded
in $\Ld^2_\loc(\R^+;\Ld^2_\loc(\R^2)^2)$. We then conclude from the Sobolev embedding that $ v_{\e,2}$ is bounded in $\Ld^2_\loc(\R^+;\Ld^q_\loc(\R^2)^2)$ for all $q<\infty$.
For our purposes it is enough to choose $q=4$ and $12$. In particular, we have proven that for all bounded subset $U\subset\R^2$,
\begin{multline}\label{eq:boundsmallbigexp}
\|\omega_\e\|_{\Ld^2_t\Ld^{4/3}}+\|\zeta_\e\|_{\Ld^2_t\Ld^2}+\| v_\e\|_{\Ld^\infty_t\Ld^2(U)}\\
+\| v_{\e,1}\|_{\Ld^2_t(\Ld^4\cap\Ld^{12})}+\|\nabla v_{\e,1}\|_{\Ld^2_t\Ld^{4/3}}+\| v_{\e,2}\|_{\Ld^2_t(\Ld^4\cap\Ld^{12}(U))}+\|\nabla v_{\e,2}\|_{\Ld^2_t\Ld^2(U)}\lesssim_{t,U}1.
\end{multline}
Therefore we have up to an extraction $\omega_\e\cvf{}\omega$ in $\Ld^{2}_\loc(\R^+;\Ld^{4/3}(\R^2))$, $\zeta_\e\cvf{}\zeta$ in $\Ld^2_\loc(\R^+;\Ld^2(\R^2))$, $ v_{\e,1}\cvf{} v_1$ in $\Ld^{2}_\loc(\R^+;\Ld^{4}(\R^2)^2)$, and $ v_{\e,2}\cvf{} v_2$ in $\Ld^2_\loc(\R^+;\Ld^4_\loc(\R^2)^2)$, for some functions $\omega,\zeta, v_1, v_2$. Comparing the above estimates with equation~\eqref{eq:limeqn1VF}, we deduce that $(\partial_t\omega_\e)_\e$ is bounded in $\Ld^{1}_\loc(\R^+;W^{-1,1}_\loc(\R^2))$. Moreover, we find by interpolation for all $|\xi|<1$ and all bounded domain $U\subset\R^2$, denoting by $U^1:=U+B_1$ its $1$-fattening,
\begin{eqnarray*}
\lefteqn{\| v_\e- v_\e(\cdot+\xi)\|_{\Ld^2_t\Ld^4(U)}\le\| v_{\e,1}- v_{\e,1}(\cdot+\xi)\|_{\Ld^2_t\Ld^4(U)}+\| v_{\e,2}- v_{\e,2}(\cdot+\xi)\|_{\Ld^2_t\Ld^4(U)}}\\
&\le&\| v_{\e,1}- v_{\e,1}(\cdot+\xi)\|_{\Ld^2_t\Ld^{4/3}(U)}^{1/4}\| v_{\e,1}- v_{\e,1}(\cdot+\xi)\|_{\Ld^2_t\Ld^{12}(U)}^{3/4}\\
&&\hspace{5cm}+\| v_{\e,2}- v_{\e,2}(\cdot+\xi)\|_{\Ld^2_t\Ld^{2}(U)}^{2/5}\| v_{\e,2}- v_{\e,2}(\cdot+\xi)\|_{\Ld^2_t\Ld^{12}(U)}^{3/5}\\
&\le&2\| v_{\e,1}- v_{\e,1}(\cdot+\xi)\|_{\Ld^2_t\Ld^{4/3}(U)}^{1/4}\| v_{\e,1}\|_{\Ld^2_t\Ld^{12}(U^1)}^{3/4}+2\| v_{\e,2}- v_{\e,2}(\cdot+\xi)\|_{\Ld^2_t\Ld^{2}(U)}^{2/5}\| v_{\e,2}\|_{\Ld^2_t\Ld^{12}(U^1)}^{3/5}\\
&\le&2|\xi|^{1/4}\|\nabla v_{\e,1}\|_{\Ld^2_t\Ld^{4/3}(U^1)}^{1/4}\| v_{\e,1}\|_{\Ld^2_t\Ld^{12}(U^1)}^{3/4}+2|\xi|^{2/5}\|\nabla v_{\e,2}\|_{\Ld^2_t\Ld^{2}(U^1)}^{2/5}\| v_{\e,2}\|_{\Ld^2_t\Ld^{12}(U^1)}^{3/5},
\end{eqnarray*}
and hence by~\eqref{eq:boundsmallbigexp},
\begin{align*}
\| v_\e- v_\e(\cdot+\xi)\|_{\Ld^2_t\Ld^4(U)}\lesssim_{t,U}|\xi|^{1/4}+|\xi|^{2/5}.
\end{align*}
Let us summarize the previous observations: up to an extraction, setting $ v:= v_1+ v_2$, we have
\begin{gather*}
\text{$\omega_\e\cvf{}\omega$ in $\Ld^{2}_\loc(\R^+;\Ld^{4/3}(\R^2))$, \quad$ v_{\e}\cvf{} v$ in $\Ld^{2}_\loc(\R^+;\Ld^{4}_\loc(\R^2)^2)$,}\\
\text{$(\partial_t\omega_\e)_\e$ bounded in $\Ld^1_\loc(\R^+;W^{-1,1}_\loc(\R^2))$,}\\
\text{$\sup_{\e>0}\| v_\e- v_\e(\cdot+\xi)\|_{\Ld^2_t\Ld^4(U)}\to0$~ as $|\xi|\to0$, for all $t\ge0$ and all bounded subset $U\subset\R^2$.}
\end{gather*}
We may then apply~\cite[Lemma~5.1]{Lions-98}, which ensures that $\omega_\e v_\e\to\omega v$ holds in the distributional sense. This allows to pass to the limit in the weak formulation of equation~\eqref{eq:limeqn2}, and the result follows.

\medskip
\noindent\step2 Proof of~(ii) and~(iii).

The proof of item~(ii) is again based on Lemma~\ref{lem:Lpest}(iii), and is completely analogous to the proof of item~(i) above. Regarding item~(iii), Lemma~\ref{lem:Lpest}(iii) does no longer apply in that case, but, since we further assume $\omega^\circ\in\Ld^q(\R^2)$ for some $q>1$, Lemma~\ref{lem:Lpvort} gives the following a priori estimate: for all $t\ge0$
\begin{align}\label{eq:additionalreg}
\|\omega\|_{\Ld^{q+1}_t\Ld^{q+1}}+\|\omega\|_{\Ld^\infty_t\Ld^q}\lesssim_t1,
\end{align}
hence in particular by interpolation $\|\omega\|_{\Ld^{p}_t\Ld^{p}}\lesssim_t1$ for all $1\le p \le2$. (Here we use the notation $\lesssim_t$ for $\le$ up to a constant that depends only on an upper bound on $t$, $(q-1)^{-1}$, $\alpha$, $\alpha^{-1}$, $|\beta|$, $\|(h,\Psi)\|_{W^{1,\infty}}$, $\| v^\circ-\bar  v^\circ\|_{\Ld^2}$, and $\|\omega^\circ\|_{\Ld^q}$.) The conclusion follows from a similar argument as in Step~1.

\medskip
\noindent\step3 Proof of~(iv).

We finally turn to the incompressible equation~\eqref{eq:limeqn1} in the conservative regime $\alpha=0$. Let $q>1$ be such that $\omega^\circ\in\Ld^q(\R^2)$.
Lemma~\ref{lem:Lpvort} or~\ref{lem:Lpest}(ii) ensures that $\omega_\e$ is bounded in $\Ld^\infty_\loc(\R^+;\Ld^1\cap\Ld^q(\R^2))$, and hence, for $q>4/3$, replacing the exponents $4/3$ and $12/7$ of Step~1 by $4/3$ and $q$, the argument of Step~1 can be immediately adapted to this case, for which we thus obtain global existence of a weak solution. In the remaining case $1<q<4/3$, the product $\omega \nabla\Delta^{-1}\omega$ (hence the product $\omega v$, cf.~\eqref{eq:decompveps}) does not make sense any more for $\omega\in\Ld^q(\R^2)$. Since in the conservative regime $\alpha=0$ no additional regularity is available (in particular, \eqref{eq:additionalreg} does not hold), we do not expect the existence of a weak solution, and we need to turn to the notion of very weak solutions as defined in Definition~\ref{defin:sol}(c), where the product $\omega v$ is reinterpreted à la Delort. Let $1<q\le 4/3$. We establish the global existence of a very weak solution.
(For the critical exponent $q=4/3$, the integrability of $ v$ found below directly implies by Remark~\ref{rem:sol}(ii) that the constructed very weak solution is automatically a weak solution.) In this step, we use the notation $\lesssim$ for $\le$ up to a constant $C$ that depends only on an upper bound on $(q-1)^{-1}$, $|\beta|$, $\|(h,\Psi,\bar  v^\circ)\|_{W^{1,\infty}}$, $\| v^\circ-\bar  v^\circ\|_{\Ld^2}$, $\|\bar\omega^\circ\|_{\Ld^2}$, and $\|\omega^\circ\|_{\Ld^q}$, and we use the notation $\lesssim_t$ (resp.\@ $\lesssim_{t,U}$) if it further depends on an upper bound on time $t$ (resp.\@ on $t$ and on the size of $U\subset\R^2$).

Let $\omega_\e^\circ$, $\bar\omega_\e^\circ$, $ v_\e^\circ$, $\bar  v_\e^\circ$ be defined as in Step~1 (with of course $\zeta_\e^\circ=\bar\zeta_\e^\circ=0$), and let $ v_\e\in\Ld^\infty_\loc(\R^+;\bar  v_\e^\circ+\Ld^2(\R^2)^2)$ be a global weak solution of~\eqref{eq:limeqn1} on $\R^+\times\R^2$ with initial data $ v_\e^\circ$, as given by Corollary~\ref{cor:globexist2}.
Lemmas~\ref{lem:aprioriest}(iii) and~\ref{lem:Lpest}(ii) then give for all $t\ge0$,
\begin{align}\label{eq:apriorieststep3-0}
\|\omega_\e\|_{\Ld^\infty_t(\Ld^1\cap\Ld^q)}+\| v_\e-\bar  v_\e^\circ\|_{\Ld^\infty_t\Ld^2}\lesssim_t1.
\end{align}
As $\bar  v_\e^\circ$ is bounded in $\Ld^2_\loc(\R^2)^2$, we deduce in particular that $ v_\e$ is bounded in $\Ld^\infty_\loc(\R^+;\Ld^2_\loc(\R^2))$.
Moreover, using the Delort type identity
\[\omega_\e  v_\e=-\frac12| v_\e|^2\nabla^\bot h-a^{-1}(\Div(aS_{ v_\e}))^\bot,\]
we then deduce that $\omega_\e v_\e$ is bounded in $\Ld^\infty_\loc(\R^+;W^{-1,1}_\loc(\R^2)^2)$.
Let us now recall the following useful decomposition,
\begin{align}\label{eq:decompv12veryweak}
 v_\e= v_{\e,1}+ v_{\e,2},\qquad  v_{\e,1}:=\nabla^\bot\triangle^{-1}\omega_\e,\qquad  v_{\e,2}:=\nabla\triangle^{-1}\Div v_\e.
\end{align}
By Riesz potential theory $ v_{\e,1}$ is bounded in $\Ld^\infty_\loc(\R^+;\Ld^p(\R^2)^2)$ for all $2< p\le \frac{2q}{2-q}$, while as in Step~1 we check that $ v_{\e,2}$ is bounded in $\Ld^\infty_\loc(\R^+;H^1_\loc(\R^2)^2)$. Hence by the Sobolev embedding, for all bounded domain $U\subset\R^2$ and all $t\ge0$,
\begin{align}\label{eq:apriorieststep3}
\|( v_\e, v_{\e,1})\|_{\Ld^\infty_t\Ld^{2q/(2-q)}(U)}\lesssim_{t,U}1.
\end{align}
Up to an extraction we then have $ v_\e\cvf* v$ in $\Ld^\infty_\loc(\R^+;\Ld^2_\loc(\R^2)^2)$ and $\omega_\e\cvf*\omega$ in $\Ld^\infty_\loc(\R^+;\Ld^q(\R^2))$, for some functions $ v,\omega$, with necessarily $\omega=\curl v$ and $\Div(a v)=0$.

We now need to pass to the limit in the nonlinearity $\omega_\e v_\e$. For that purpose, for all $\eta>0$, we set $ v_{\e,\eta}:=\rho_\eta\ast v_\e$ and $\omega_{\e,\eta}:=\rho_\eta\ast\omega_\e=\curl v_{\e,\eta}$, where $\rho_\eta(x):=\eta^{-d}\rho(x/\eta)$ is the regularization kernel defined in Step~1, and we then decompose the nonlinearity as follows,
\begin{align*}
\omega_\e v_\e&=(\omega_{\e,\eta}-\omega_\e)( v_{\e,\eta}- v_\e)-\omega_{\e,\eta} v_{\e,\eta}+\omega_{\e,\eta} v_\e+\omega_\e v_{\e,\eta}.
\end{align*}
We study each right-hand side term separately, and split the proof into four further substeps.

\medskip
\noindent\substep{3.1}
We prove that $(\omega_{\e,\eta}-\omega_\e)( v_{\e,\eta}- v_\e)\to0$ holds in the distributional sense (and even strongly in $\Ld^\infty_\loc(\R^+;W^{-1,1}_\loc(\R^2)^2)$) as $\eta\downarrow0$, uniformly in $\e>0$.

For that purpose, we use the Delort type identity
\[(\omega_{\e,\eta}-\omega_\e) ( v_{\e,\eta}- v_\e)=a^{-1}( v_{\e,\eta}- v_\e)\Div(a( v_{\e,\eta}- v_\e))-\frac12| v_{\e,\eta}- v_\e|^2\nabla^\bot h-a^{-1}(\Div (aS_{ v_{\e,\eta}- v_\e}))^\bot.\]
Noting that the constraint $0=a^{-1}\Div(a v_\e)=\nabla h\cdot v_\e+\Div v_\e$ yields
\[a^{-1}\Div(a( v_{\e,\eta}- v_\e))=\nabla h\cdot  v_{\e,\eta}+\Div  v_{\e,\eta}=\nabla h\cdot (\rho_\eta\ast v_\e)+\rho_\eta\ast\Div  v_{\e}=\nabla h\cdot (\rho_\eta\ast  v_{\e})-\rho_\eta\ast(\nabla h\cdot  v_{\e}),\]
the above identity becomes
\begin{multline*}
(\omega_{\e,\eta}-\omega_\e) ( v_{\e,\eta}- v_\e)=( v_{\e,\eta}- v_\e)\big(\nabla h\cdot (\rho_\eta\ast  v_{\e})-\rho_\eta\ast(\nabla h\cdot  v_{\e})\big)\\
-\frac12| v_{\e,\eta}- v_\e|^2\nabla^\bot h-a^{-1}(\Div(a S_{ v_{\e,\eta}- v_\e}))^\bot.
\end{multline*}
First, using the boundedness of $ v_\e$ (hence of $ v_{\e,\eta}$) in $\Ld^\infty_\loc(\R^+;\Ld^2_\loc(\R^2)^2)$, we may estimate, for all bounded domain $U\subset\R^2$, denoting by $U^\eta:=U+B_\eta$ its $\eta$-fattening,
\begin{eqnarray*}
\lefteqn{\int_U\big|( v_{\e,\eta}- v_\e)\big(\nabla h\cdot (\rho_\eta\ast  v_{\e})-\rho_\eta\ast(\nabla h\cdot  v_{\e})\big)\big|}\\
&\le&\|( v_\e, v_{\e,\eta})\|_{\Ld^2(U)}\bigg(\int_U \Big(\int \rho_\eta(y)|\nabla h(x)-\nabla h(x-y)|| v_\e(x-y)|dy\Big)^2dx\bigg)^{1/2}\\
&\lesssim&\|( v_\e, v_{\e,\eta})\|_{\Ld^2(U^\eta)}^2\Big(\int\rho_\eta(y)\int_U|\nabla h(x)-\nabla h(x-y)|^2dxdy\Big)^{1/2},
\end{eqnarray*}
where the right-hand side converges to $0$ as $\eta\downarrow0$, uniformly in $\e$.
Second, using the decomposition~\eqref{eq:decompv12veryweak}, and setting $ v_{\e,\eta,1}:=\rho_\eta\ast  v_{\e,1}$, $ v_{\e,\eta,2}:=\rho_\eta\ast  v_{\e,2}$, the Hölder inequality yields for all bounded domain $U\subset\R^2$,
\begin{multline*}
\int_{U}|( v_\e- v_{\e,\eta})\otimes( v_\e- v_{\e,\eta})|\le\int_{U}| v_\e- v_{\e,\eta}|| v_{\e,1}- v_{\e,\eta,1}|+\int_{U}| v_{\e}- v_{\e,\eta}|| v_{\e,2}- v_{\e,\eta,2}|\\
\le \|( v_\e, v_{\e,\eta})\|_{\Ld^{2q/(2-q)}(U)}\| v_{\e,1}- v_{\e,\eta,1}\|_{\Ld^{2q/(3q-2)}(U)}+\|( v_\e, v_{\e,\eta})\|_{\Ld^{2}(U)}\| v_{\e,2}- v_{\e,\eta,2}\|_{\Ld^{2}(U)}.
\end{multline*}
Recalling the choice $1<q\le 4/3$, we find by interpolation
\begin{align*}
\| v_{\e,1}- v_{\e,\eta,1}\|_{\Ld^{2q/(3q-2)}(U)}&\le \| v_{\e,1}- v_{\e,\eta,1}\|_{\Ld^{2}(U)}^{\frac{4-3q}{2-q}}\| v_{\e,1}- v_{\e,\eta,1}\|_{\Ld^{q}(U)}^{2\frac{q-1}{2-q}}\\
&\le\eta^{2\frac{q-1}{2-q}} \|( v_{\e,1}, v_{\e,\eta,1})\|_{\Ld^{2}(U)}^{\frac{4-3q}{2-q}}\|\nabla v_{\e,1}\|_{\Ld^{q}}^{2\frac{q-1}{2-q}},
\end{align*}
and hence by the Calderón-Zygmund theory,
\begin{align*}
\| v_{\e,1}- v_{\e,\eta,1}\|_{\Ld^{2q/(3q-2)}(U)}&\lesssim\eta^{2\frac{q-1}{2-q}} \|( v_{\e,1}, v_{\e,\eta,1})\|_{\Ld^{2}(U)}^{\frac{4-3q}{2-q}}\|\omega_\e\|_{\Ld^{q}}^{2\frac{q-1}{2-q}},
\end{align*}
while as in Step~1 we find
\begin{align*}
\| v_{\e,2}- v_{\e,\eta,2}\|_{\Ld^2_t\Ld^{2}(U)}&\le\eta\|\nabla v_{\e,2}\|_{\Ld^2_t\Ld^{2}(U^\eta)}\lesssim_U\eta.
\end{align*}
Combining this with the a priori estimate~\eqref{eq:apriorieststep3}, we may conclude
\[\int_0^t\int_{U}|( v_\e- v_{\e,\eta})\otimes( v_\e- v_{\e,\eta})|\lesssim_{t,U}\eta^{2\frac{q-1}{2-q}}+\eta,\]
and the claim follows.

\medskip
\noindent\substep{3.2}
We set $ v_\eta:=\rho_\eta\ast v$, $\omega_\eta:=\rho_\eta\ast\omega=\curl v_\eta$, and we prove that $-\omega_{\e,\eta} v_{\e,\eta}+\omega_{\e,\eta} v_\e+\omega_\e v_{\e,\eta}\to-\omega_{\eta} v_{\eta}+\omega_\eta v+\omega v_\eta$ in the distributional sense as $\e\downarrow0$, for any fixed $\eta>0$.

As $q<2<q'$, the weak convergences $ v_\e\cvf* v$ in $\Ld^\infty_\loc(\R^+;\Ld^2_\loc(\R^2)^2)$ and $\omega_\e\cvf*\omega$ in $\Ld^\infty_\loc(\R^+;\Ld^q(\R^2))$ imply for instance $ v_{\e,\eta}\cvf{*}  v_\eta$ in $\Ld^\infty_\loc(\R^+;W^{1,q'}_\loc(\R^2)^2)$ and $\omega_{\e,\eta}\cvf{*} \omega_\eta$ in $\Ld^\infty_\loc(\R^+;H^1(\R^2))$ as $\e\downarrow0$, for any fixed $\eta>0$ (note that these are still only weak-* convergences because no regularization occurs with respect to the time variable $t$). Moreover, examining equation~\eqref{eq:limeqn1VF} together with the a priori estimates obtained at the beginning of this step, we observe that $\partial_t\omega_\e$ is bounded in $\Ld^\infty_\loc(\R^+;W^{-2,1}_\loc(\R^2))$, hence $\partial_t\omega_{\e,\eta}=\rho_\eta\ast\partial_t\omega_\e$ is also bounded in the same space. Since by the Rellich theorem the space $\Ld^q(U)$ is compactly embedded in $W^{-1,q}(U)\subset W^{-2,1}(U)$ for all bounded domain $U\subset\R^2$, the Aubin-Simon lemma ensures that we have $\omega_\e\to\omega$ strongly in $\Ld^\infty_\loc(\R^+;W^{-1,q}_\loc(\R^2))$, and similarly, since $H^1(U)$ is compactly embedded in $\Ld^2(U)\subset W^{-2,1}(U)$, we also deduce $\omega_{\e,\eta}\to\omega_\eta$ strongly in $\Ld^\infty_\loc(\R^+;\Ld^2_\loc(\R^2))$.
This proves the claim.

\medskip
\noindent\substep{3.3}
We prove that $-\omega_{\eta} v_{\eta}+\omega_\eta v+\omega v_\eta\to-\frac12| v|^2\nabla^\bot h-a^{-1}(\Div (aS_{ v}))^\bot$ holds in the distributional sense as $\eta\downarrow0$.

\nopagebreak
For that purpose, we use the following Delort type identity,
\begin{multline*}
-\omega_{\eta} v_{\eta}+\omega_\eta v+\omega v_\eta=-a^{-1}( v_\eta- v)\Div(a( v_{\eta}- v))+\frac12| v_\eta- v|^2\nabla^\bot h+a^{-1}(\Div (aS_{ v_\eta- v}))^\bot\\
+a^{-1} v\Div(a v)-\frac12| v|^2\nabla^\bot h-a^{-1}(\Div (aS_{ v}))^\bot.
\end{multline*}
Noting that the limiting constraint $0=a^{-1}\Div(a v)=\nabla h\cdot v+\Div v$ gives
\[a^{-1}\Div(a( v_{\eta}- v))=\nabla h\cdot  v_{\eta}+\Div  v_{\eta}=\nabla h\cdot( \rho_{\eta}\ast  v)+\rho_\eta\ast\Div  v=\nabla h\cdot( \rho_{\eta}\ast  v)-\rho_\eta\ast(\nabla h\cdot  v),\]
the above identity takes the form
\begin{multline*}
-\omega_{\eta} v_{\eta}+\omega_\eta v+\omega v_\eta=-a^{-1}( v_\eta- v)\big(\nabla h\cdot( \rho_{\eta}\ast  v)-\rho_\eta\ast(\nabla h\cdot  v)\big)+\frac12| v_\eta- v|^2\nabla^\bot h+a^{-1}(\Div (aS_{ v_\eta- v}))^\bot\\
-\frac12| v|^2\nabla^\bot h-a^{-1}(\Div (aS_{ v}))^\bot,
\end{multline*}
and it is thus sufficient to prove that the first three right-hand side terms tend to $0$ in the distributional sense as $\eta\downarrow0$. This is proven just as in Substep~3.1 above, with $ v_{\e,\eta}, v_\e$ replaced by $ v_\eta, v$.

\medskip
\noindent\substep{3.4} Conclusion.

Combining the three previous substeps yields $\omega_\e v_\e\to -\frac12| v|^2\nabla^\bot h-a^{-1}(\Div (aS_{ v}))^\bot$ in the distributional sense as $\e\downarrow0$. Passing to the limit in the very weak formulation of equation~\eqref{eq:limeqn1VF}, the conclusion follows.
\end{proof}

\section{Uniqueness}\label{chap:unique}

We turn to the uniqueness results stated in Theorem~\ref{th:mainunique}. Using similar energy arguments as in the proof of Lemma~\ref{lem:aprioriest}, in the spirit of~\cite[Appendix~B]{Serfaty-15}, we prove a general weak-strong uniqueness principle.
Note that in the degenerate case $\lambda=0$ an additional term needs to be added to the usual energy, in link with the fact that $\omega$ and $v$ are then on an equal footing with regard to regularity.
In the incompressible case, we further prove uniqueness in the class of bounded vorticity, based on transport arguments à la Loeper~\cite{Loeper-06} (see also~\cite{Serfaty-Vazquez-14}), but these tools are not available in the compressible case.

\begin{prop}[Uniqueness]\label{prop:limeqn-unique}
Let $\alpha,\beta\in\R$, $\lambda\ge0$, $T>0$, and $h,\Psi\in W^{1,\infty}(\R^2)^2$. Let $ v^\circ:\R^2\to\R^2$ with $\omega^\circ:=\curl v^\circ\in\Pc(\R^2)$, and in the incompressible case~\eqref{eq:limeqn1} further assume that $\Div(a v^\circ)=0$.
\begin{enumerate}[(i)]
\item \emph{Weak-strong uniqueness principle for~\eqref{eq:limeqn1} and~\eqref{eq:limeqn2} in the non-degenerate case $\lambda>0$, $\alpha\ge0$:}\\
If~\eqref{eq:limeqn1} or~\eqref{eq:limeqn2} admits a weak solution $ v\in \Ld^2_\loc([0,T); v^\circ+\Ld^2(\R^2)^2)\cap\Ld^\infty_\loc([0,T);W^{1,\infty}(\R^2)^2)$ on $[0,T)\times\R^2$ with initial data $ v^\circ$, then it is the unique weak solution of~\eqref{eq:limeqn1} or of~\eqref{eq:limeqn2} on $[0,T)\times\R^2$ in the class $\Ld^2_\loc([0,T); v^\circ+\Ld^2(\R^2)^2)$ with initial data $ v^\circ$.
\item \emph{Weak-strong uniqueness principle for~\eqref{eq:limeqn2} in the degenerate parabolic case $\lambda=\beta=0$, $\alpha\ge0$:}\\
Let $E^2_{T, v^\circ}$ denote the class of all $w\in \Ld^2_\loc([0,T); v^\circ+\Ld^2(\R^2)^2)$ with $\curl w\in\Ld^2_\loc([0,T);\Ld^2(\R^2))$. If~\eqref{eq:limeqn2} admits a weak solution $ v\in E^2_{T, v^\circ}\cap\Ld^\infty_\loc([0,T);\Ld^{\infty}(\R^2)^2)$ on $[0,T)\times\R^2$ with initial data $ v^\circ$, and with $\omega:=\curl v\in\Ld^\infty_\loc([0,T);W^{1,\infty}(\R^2))$, then it is the unique weak solution of~\eqref{eq:limeqn2} on $[0,T)\times\R^2$ in the class $E^2_{T, v^\circ}$ with initial data $ v^\circ$.
\item \emph{Uniqueness for~\eqref{eq:limeqn1} with bounded vorticity, $\alpha,\beta\in\R$:}\\
There exists at most a unique weak solution $ v$ of~\eqref{eq:limeqn1} on $[0,T)\times\R^2$ with initial data $ v^\circ$, in the class of all $w$'s such that $\curl w\in\Ld^\infty_\loc([0,T);\Ld^\infty(\R^2))$.
\end{enumerate}
Moreover, in items~(i)--(ii), the condition $\alpha\ge0$ may be dropped if we further restrict to weak solutions $ v$ such that $\curl v\in\Ld^\infty_\loc([0,T);\Ld^\infty(\R^2))$.
\end{prop}

\begin{proof}
In this proof, we use the notation $\lesssim$ for $\le$ up to a constant $C>0$ that depends only on an upper bound on $\alpha$, $|\beta|$, $\lambda$, $\lambda^{-1}$, and $\|(h,\Psi)\|_{W^{1,\infty}}$, and we add subscripts to indicate dependence on further parameters.
We split the proof into four steps, first proving item~(i) in the case~\eqref{eq:limeqn1}, then in the case~\eqref{eq:limeqn2}, and finally turning to items~(ii) and~(iii).

\medskip
\noindent\step1 Proof of~(i) in the case~\eqref{eq:limeqn1}.

Let $\alpha\ge0$, $\beta\in\R$, and let $ v_1, v_2\in\Ld^2_\loc([0,T); v^\circ+\Ld^2(\R^2)^2)$ be two weak solutions of~\eqref{eq:limeqn1} on $[0,T)\times\R^2$ with initial data $ v^\circ$,
and assume $ v_2\in\Ld^\infty_\loc([0,T);W^{1,\infty}(\R^2)^2)$.
Set $\delta v:= v_1- v_2$ and $\delta\omega:=\omega_1-\omega_2$.
As the constraint $\Div(a\delta v)=0$ yields $\delta v=a^{-1}\nabla^\bot(\Div a^{-1}\nabla)^{-1}\delta\omega$, and as by assumption $\delta v\in\Ld^2_\loc([0,T);\Ld^2(\R^2)^2)$, we deduce
$\delta\omega\in\Ld^2_\loc([0,T);\dot H^{-1}(\R^2))$ and $(\Div a^{-1}\nabla)^{-1}\delta\omega\in\Ld^2_\loc([0,T);\dot H^{1}(\R^2))$. Moreover, the definition of a weak solution ensures that $\omega_i:=\curl v_i\in\Ld^\infty([0,T);\Pc(\R^2))$ (cf.\@ Lemma~\ref{lem:aprioriest}(i)), and $| v_i|^2\omega_i\in\Ld^1_\loc([0,T);\Ld^1(\R^2))$, for $i=1,2$, so that all the integrations by parts below are directly justified.
From equation~\eqref{eq:limeqn1VF}, we compute the following time derivative
\begin{align}
\partial_t\int\delta\omega(-\Div a^{-1}\nabla)^{-1}\delta\omega&=2\int \nabla(\Div a^{-1}\nabla)^{-1}\delta\omega\cdot\Big((\alpha(\Psi+ v_1)^\bot+\beta(\Psi+ v_1))\omega_1\nonumber\\
&\hspace{5.5cm}-(\alpha(\Psi+ v_2)^\bot+\beta(\Psi+ v_2))\omega_2\Big)\nonumber\\
&=-2\int a\delta v^\bot\cdot\Big((\alpha(\delta v)^\bot+\beta\delta v)\omega_1+(\alpha(\Psi+ v_2)^\bot+\beta(\Psi+ v_2))\delta\omega\Big)\nonumber\\
&=-2\alpha\int a|\delta v|^2\omega_1-2\int a\delta\omega\delta v^\bot\cdot(\alpha(\Psi+ v_2)^\bot+\beta(\Psi+ v_2)).\label{eq:decompderivunique}
\end{align}
As $ v_2$ is Lipschitz-continuous, and as the definition of a weak solution ensures that $\omega_1 v_1\in\Ld^1_\loc([0,T);\Ld^1(\R^2)^2)$, the following Delort type identity holds in $\Ld^1_\loc([0,T);W^{-1,1}_\loc(\R^2)^2)$,
\[\delta\omega\delta v^\bot=\frac12|\delta v|^2\nabla h+a^{-1}\Div(aS_{\delta v}).\]
Combining this with~\eqref{eq:decompderivunique} and the non-negativity of $\alpha\omega_1$ yields
\begin{align*}
\partial_t\int\delta\omega(-\Div a^{-1}\nabla)^{-1}\delta\omega&\le-\int a|\delta v|^2\nabla h\cdot(\alpha(\Psi+ v_2)^\bot+\beta(\Psi+ v_2))+2\int a S_{\delta v}:\nabla(\alpha(\Psi+ v_2)^\bot+\beta(\Psi+ v_2))\\
&\le C(1+\| v_2\|_{W^{1,\infty}})\int a|\delta v|^2.
\end{align*}
The uniqueness result $\delta v=0$ then follows from the Grönwall inequality, since by integration by parts
\[\int a|\delta v|^2=\int a^{-1}|\nabla(\Div a^{-1}\nabla)^{-1}\delta\omega|^2=\int \delta\omega (-\Div a^{-1}\nabla)^{-1}\delta\omega.\]
Note that if we further assume $\omega_1\in\Ld^\infty([0,T);\Ld^\infty(\R^2))$, then the non-negativity of $\alpha$ can be dropped: it indeed suffices to estimate in that case $-2\alpha\int a|\delta v|^2\omega_1\le C\|\omega_1\|_{\Ld^\infty}\int a|\delta v|^2$, and the result then follows as above. A similar observation also holds in the context of item~(ii).

\medskip
\noindent\step2 Proof of~(i) in the case~\eqref{eq:limeqn2}.

Let $\alpha\ge0$, $\beta\in\R$, $\lambda>0$, and let $ v_1, v_2\in\Ld^2_\loc([0,T); v^\circ+\Ld^2(\R^2)^2)$ be two weak solutions of~\eqref{eq:limeqn2} on $[0,T)\times\R^2$ with initial data $ v^\circ$, and assume $ v_2\in\Ld^\infty_\loc([0,T);W^{1,\infty}(\R^2)^2)$. The definition of a weak solution ensures that $\omega_i:=\curl v_i\in\Ld^\infty([0,T);\Pc(\R^2))$ (cf.\@ Lemma~\ref{lem:aprioriest}(i)), $\zeta_i:=\Div(a v_i)\in\Ld^2_\loc([0,T);\Ld^2(\R^2))$, and $| v_i|^2\omega_i\in\Ld^1_\loc([0,T);\Ld^1(\R^2))$, for $i=1,2$, and hence the integrations by parts below are directly justified. Set $\delta v:= v_1- v_2$, $\delta\omega:=\omega_1-\omega_2$, and $\delta\zeta:=\zeta_1-\zeta_2$. From equation~\eqref{eq:limeqn2}, we compute the following time derivative
\begin{align*}
\partial_t\int a|\delta v|^2&=2\int a\delta v\cdot \Big(\lambda\nabla(a^{-1}\delta\zeta)-\alpha(\Psi+ v_1)\omega_1+\beta(\Psi+ v_1)^\bot\omega_1+\alpha(\Psi+ v_2)\omega_2-\beta(\Psi+ v_2)^\bot\omega_2\Big)\\
&=-2\lambda\int a^{-1}|\delta\zeta|^2-2\alpha\int a|\delta v|^2\omega_1+2\int a\delta\omega\delta v\cdot\big(\alpha(\Psi+ v_2)-\beta (\Psi+ v_2)^\bot\big).
\end{align*}
As $ v_2$ is Lipschitz-continuous, and as the definition of a weak solution implies $\omega_1 v_1\in\Ld^1_\loc([0,T)\times\R^2)^2$, the following Delort type identity holds in $\Ld^1_\loc([0,T);W^{-1,1}_\loc(\R^2)^2)$,
\begin{gather*}
\delta\omega\delta v=a^{-1}\delta\zeta\delta v^\bot-\frac12|\delta v|^2\nabla^\bot h-a^{-1}(\Div (aS_{\delta v}))^\bot.
\end{gather*}
The above may then be estimated as follows, after integration by parts,
\begin{align*}
\partial_t\int a|\delta v|^2&\le-2\lambda\int a^{-1}|\delta\zeta|^2-2\alpha\int a|\delta v|^2\omega_1+C(1+\| v_2\|_{\Ld^{\infty}})\int|\delta\zeta||\delta v|+C(1+\| v_2\|_{W^{1,\infty}})\int a|\delta v|^2,
\end{align*}
and thus, using the choice $\lambda>0$, the inequality $2xy\le x^2+y^2$, and the non-negativity of $\alpha\omega_1$,
\begin{align*}
\partial_t\int a|\delta v|^2&\le C(1+\lambda_\e^{-1})(1+\| v_2\|_{W^{1,\infty}}^2)\int a|\delta v|^2.
\end{align*}
The Grönwall inequality then implies uniqueness, $\delta v=0$.

\medskip
\noindent\step3 Proof of~(ii).

Let $\lambda=\beta=0$, $\alpha=1$, and let $ v_1, v_2\in\Ld^2_\loc([0,T); v^\circ+\Ld^2(\R^2)^2)$ be two weak solutions of~\eqref{eq:limeqn2} on $[0,T)\times\R^2$ with initial data $ v^\circ$, and with $\omega_i:=\curl v_i\in\Ld^2_\loc([0,T);\Ld^2(\R^2))$ for $i=1,2$, and further assume $ v_2\in\Ld^\infty_\loc([0,T);\Ld^\infty(\R^2)^2)$ and $\omega_2\in\Ld^\infty_\loc([0,T);W^{1,\infty}(\R^2))$.
The definition of a weak solution ensures that $\omega_i:=\curl v_i\in\Ld^\infty([0,T);\Pc(\R^2))$ (cf.\@ Lemma~\ref{lem:aprioriest}(i)), $\zeta_i:=\Div(a v_i)\in\Ld^2_\loc([0,T);\Ld^2(\R^2))$, and $| v_i|^2\omega_i\in\Ld^1_\loc([0,T);\Ld^1(\R^2))$, for $i=1,2$, and hence the integrations by parts below are directly justified.
Denoting $\delta v:= v_1- v_2$ and $\delta\omega:=\omega_1-\omega_2$, equation~\eqref{eq:limeqn2} yields
\begin{align}
\partial_t\delta v=-(\Psi+ v_2)\delta\omega-\omega_1\delta v,\label{eq:lambda0eq1}
\end{align}
while equation~\eqref{eq:limeqn1VF} takes the form
\begin{align}
\partial_t\delta\omega&=\Div((\Psi+ v_2)^\bot\delta\omega)+\Div(\omega_1\delta v^\bot)\nonumber\\
&=\Div((\Psi+ v_2)^\bot\delta\omega)+\nabla\omega_1\cdot\delta v^\bot-\omega_1\delta\omega\nonumber\\
&=\Div((\Psi+ v_2)^\bot\delta\omega)+\nabla\omega_2\cdot\delta v^\bot+\nabla\delta\omega\cdot\delta v^\bot-\omega_1\delta\omega.\label{eq:lambda0eq2}
\end{align}
Testing equation~\eqref{eq:lambda0eq1} against $\delta v$ yields, by non-negativity of $\omega_1$,
\begin{align*}
\partial_t\int|\delta v|^2&=-2\int|\delta v|^2 \omega_1-2\int\delta v\cdot(\Psi+ v_2)\delta\omega\le C(1+\| v_2\|_{\Ld^\infty})\int|\delta v||\delta\omega|.
\end{align*}
Testing equation~\eqref{eq:lambda0eq2} against $\delta\omega$ and integrating by parts yields, by non-negativity of $\omega_1$ and $\omega_2$,
\begin{align*}
\partial_t\int|\delta\omega|^2&=-\int \nabla|\delta\omega|^2\cdot (\Psi+ v_2)^\bot+2\int \delta\omega\nabla\omega_2\cdot\delta v^\bot+\int\nabla|\delta\omega|^2\cdot\delta v^\bot-2\int |\delta\omega|^2\omega_1\\
&=-\int |\delta\omega|^2(\curl \Psi+\omega_2)+2\int \delta\omega \nabla\omega_2\cdot\delta v^\bot+\int|\delta\omega|^2(\omega_1-\omega_2)-2\int |\delta\omega|^2\omega_1\\
&\le C\int |\delta\omega|^2+2\|\nabla\omega_2\|_{\Ld^\infty}\int |\delta v||\delta\omega|.
\end{align*}
Combining the above two estimates and using the inequality $2xy\le x^2+y^2$, we find
\begin{align*}
\partial_t\int(|\delta v|^2+|\delta\omega|^2)&\le C(1+\|( v_2,\nabla\omega_2)\|_{\Ld^\infty})\int (|\delta v|^2+|\delta\omega|^2),
\end{align*}
and the uniqueness result follows from the Grönwall inequality.

\medskip
\noindent\step4 Proof of~(iii).

Let $\alpha,\beta\in\R$, and let $ v_1, v_2$ denote two solutions of~\eqref{eq:limeqn1} on $[0,T)\times\R^2$ with initial data $ v^\circ$, and with $\omega_1,\omega_2\in\Ld^\infty_\loc([0,T);\Ld^\infty(\R^2))$.
First we prove that $ v_1^t, v_2^t$ are log-Lipschitz for all $t\in[0,T)$ (compare with the easier situation in~\cite[Lemma~4.1]{Serfaty-Vazquez-14}). For $i=1,2$, using the identity $ v_i^t=\nabla^\bot\triangle^{-1}\omega_i^t+\nabla\triangle^{-1}\Div v_i^t$ with $\Div v_i^t=-\nabla h\cdot  v_i^t$, we may decompose for all $x,y$,
\[| v_i^t(x)- v_i^t(y)|\le|\nabla\triangle^{-1}\omega_i^t(x)-\nabla\triangle^{-1}\omega_i^t(y)|+|\nabla\triangle^{-1}(\nabla h\cdot  v_i^t)(x)-\nabla\triangle^{-1}(\nabla h\cdot  v_i^t)(y)|.\]
By the embedding of the Zygmund space $C^1_*(\R^2)=B^1_{\infty,\infty}(\R^2)$ into the space of log-Lipschitz functions (see e.g.\@ \cite[Proposition~2.107]{BCD-11}), we may estimate
\[| v_i^t(x)- v_i^t(y)|\lesssim \big(\|\nabla^2\triangle^{-1}\omega_i^t\|_{C^0_*}+\|\nabla^2\triangle^{-1}(\nabla h\cdot  v_i^t)\|_{C^0_*}\big)|x-y|(1+\log_-(|x-y|)),\]
and hence, applying Lemma~\ref{lem:pottheoryCsHs}(ii) and recalling that $\Ld^\infty(\R^2)$ is embedded in $C^0_*(\R^2)=B^0_{\infty,\infty}(\R^2)$, we find for all $1\le p<\infty$,
\begin{align*}
| v_i^t(x)- v_i^t(y)|&\lesssim_p \big(\|\omega_i^t\|_{\Ld^1\cap C^0_*}+\|\nabla h\cdot v_i^t\|_{\Ld^{p}\cap C^0_*}\big)|x-y|(1+\log_-(|x-y|))\\
&\lesssim \big(\|\omega_i^t\|_{\Ld^1\cap \Ld^\infty}+\| v_i^t\|_{\Ld^{p}\cap \Ld^\infty}\big)|x-y|(1+\log_-(|x-y|)).
\end{align*}
Noting that $ v_i^t=a^{-1}\nabla^\bot(\Div a^{-1}\nabla)^{-1}\omega_i^t$, the elliptic estimates of Lemma~\ref{lem:globellreg} yield $\| v_i^t\|_{\Ld^{p_0}\cap\Ld^\infty}\lesssim\|\omega_i^t\|_{\Ld^1\cap\Ld^\infty}$ for some exponent $2<p_0\lesssim1$. For the choice $ P=p_0$, the above thus takes the following form,
\begin{align}\label{eq:logLipom}
| v_i^t(x)- v_i^t(y)|&\lesssim \|\omega_i^t\|_{\Ld^1\cap \Ld^\infty}|x-y|(1+\log_-(|x-y|))\le(1+\|\omega_i^t\|_{\Ld^\infty})|x-y|(1+\log_-(|x-y|)),
\end{align}
which proves that $ v_1^t, v_2^t$ are log-Lipschitz for all $t\in[0,T)$.

For $i=1,2$, as the vector field $\alpha(\Psi+ v_i)+\beta(\Psi+ v_i)^\bot$ is log-Lipschitz in space, the associated flow $\psi_i:[0,T)\times\R^2\to\R^2$ is well-defined globally,
\[\partial_t\psi_i(x)=-(\alpha(\Psi+ v_i)+\beta(\Psi+ v_i)^\bot)(\psi_i(x)).\]
As the transport equation~\eqref{eq:limeqn1VF} ensures that $\omega_i^t=(\psi_i^t)_*\omega^\circ$ for $i=1,2$, the $2$-Wasserstein distance between the solutions $\omega_1^t,\omega_2^t\in\Pc(\R^2)$ is bounded by
\begin{align}\label{eq:boundW2Q}
W_2(\omega_1^t,\omega_2^t)^2\le Q^t:=\int|\psi_1^t(x)-\psi_2^t(x)|^2\omega^\circ(x)dx.
\end{align}
Now the time derivative of $Q$ is estimated by
\begin{align*}
\partial_t Q^t&=-2\int(\psi_1^t(x)-\psi_2^t(x))\cdot\big((\alpha\Psi+\beta\Psi^\bot)(\psi_1^t(x))-(\alpha\Psi+\beta\Psi^\bot)(\psi_2^t(x))\big)\omega^\circ(x)dx\\
&\qquad-2\int(\psi_1^t(x)-\psi_2^t(x))\cdot\big((\alpha v_1^t+\beta( v_1^t)^\bot)(\psi_1^t(x))-(\alpha v_2^t+\beta( v_2^t)^\bot)(\psi_2^t(x))\big)\omega^\circ(x)dx\\
&\le CQ^t+C(Q^t)^{1/2}\bigg(\int| v_1^t(\psi_1^t(x))- v_2^t(\psi_2^t(x))|^2\omega^\circ(x)dx\bigg)^{1/2}\\
&\le CQ^t+C(Q^t)^{1/2}(T_1^t+T_2^t)^{1/2},
\end{align*}
where we have set
\[T_1^t:=\int|( v_1^t- v_2^t)(\psi_2^t(x))|^2\omega^\circ(x)dx,\qquad T_2^t:=\int| v_1^t(\psi_1^t(x))- v_1^t(\psi_2^t(x))|^2\omega^\circ(x)dx.\]
We first study $T_1$. Using that $ v_i=a^{-1}\nabla^\bot(\Div a^{-1}\nabla)^{-1}\omega_i$, we find
\begin{align*}
T_1^t=\int | v_1^t- v_2^t|^2\omega_2^t\le\|\omega_2^t\|_{\Ld^\infty}\int| v_1^t- v_2^t|^2&=\|\omega_2^t\|_{\Ld^\infty}\int|\nabla(\Div a^{-1}\nabla)^{-1}(\omega_1^t-\omega_2^t)|^2\\
&\lesssim\|\omega_2^t\|_{\Ld^\infty}\int|\nabla\triangle^{-1}(\omega_1^t-\omega_2^t)|^2.
\end{align*}
(Here, we use the fact that if $-\Div (a^{-1}\nabla u_1)=-\triangle u_2$ with $u_1,u_2\in H^1(\R^2)$, then $\int a^{-1} |\nabla u_1|^2=\int\nabla u_1\cdot \nabla u_2\le\frac12\int a^{-1}|\nabla u_1|^2+\frac12\int a|\nabla u_2|^2$, hence $\int a^{-1}|\nabla u_1|^2\le\int a|\nabla u_2|^2$.) Loeper's inequality~\cite[Proposition~3.1]{Loeper-06} and the bound~\eqref{eq:boundW2Q} then imply
\begin{align*}
T_1^t\le\|\omega_2^t\|_{\Ld^\infty}(\|\omega_1^t\|_{\Ld^\infty}\vee\|\omega_2^t\|_{\Ld^\infty})W_2(\omega_1^t,\omega_2^t)^2\le\|(\omega_1^t,\omega_2^t)\|_{\Ld^\infty}^2Q^t.
\end{align*}
We finally turn to $T_2$.
Using the log-Lipschitz property~\eqref{eq:logLipom} and the concavity of the function $x\mapsto x(1+\log_-x)^2$, we obtain by Jensen's inequality,
\begin{align*}
T_2^t&\lesssim \|\omega^t_1\|_{\Ld^\infty}^2\int(1+\log_-(|\psi_1^t-\psi_2^t|))^2|\psi_1^t-\psi_2^t|^2\omega^\circ\\
&\le\|\omega^t_1\|_{\Ld^\infty}^2\Big(1+\log_-\int|\psi_1^t-\psi_2^t|^2\omega^\circ\Big)^2\int|\psi_1^t-\psi_2^t|^2\omega^\circ\\
&\lesssim \|\omega^t_1\|_{\Ld^\infty}^2(1+\log_-Q^t)^2\,Q^t.
\end{align*}
We may thus conclude $\partial_tQ\lesssim (1+\|(\omega_1,\omega_2)\|_{\Ld^\infty})(1+\log_-Q)\,Q$,
and the uniqueness result follows from a Grönwall argument.
\end{proof}

\newpage
\setcounter{section}{0}
\renewcommand\thesection{\Alph{section}}

\section{Appendix: Degenerate parabolic case (jointly written with Julian Fischer)}\label{sec:degenerate}
We now turn to the study of the compressible equation~\eqref{eq:limeqn2} in the degenerate parabolic case $\lambda=\beta=0$, $\alpha=1$, that is,
\begin{align}\label{eq:limeqn2-degen}
\partial_t v=-(\Psi+ v)\curl v,\qquad\text{in $\R^+\times\R^2$},
\end{align}
with initial data $ v|_{t=0}= v^\circ$.
A local existence result is already established in Proposition~\ref{prop:locexistdeg} above, and uniqueness is obtained in Proposition~\ref{prop:limeqn-unique}(ii), but the absence of strong enough a priori estimates on the divergence $\Div v$ due to the degeneracy of the equation makes the question of global existence delicate.
In the present appendix, jointly written with Julian Fischer, we show how to exploit the particular scalar structure of the solution $ v$ to establish global existence and finer uniqueness results. More precisely, we establish the following, which in particular implies Theorem~\ref{th:deg-case+JF}.

\begin{prop}\label{prop:main-deg}
Let $\lambda=\beta=0$, $\alpha=1$, let $ v^\circ,\Psi\in \Ld^{\infty}_\loc(\R^2)^2$ with $\curl v^\circ,\curl\Psi\in\Ld^\infty_\loc(\R^2)$ and $\curl v^\circ\ge0$, and assume that $ v^\circ$ and $\Psi$ are $\log$-Lipschitz, that is, for all $x,y$,
\[| v^\circ(x)- v^\circ(y)|+|\Psi(x)-\Psi(y)|\le C|x-y|(1+\log_-(|x-y|)).\]
There exists a unique global strong solution $ v\in \Ld^{\infty}_\loc(\R^+\times\R^2)$ of~\eqref{eq:limeqn2-degen} with $\curl v\in\Ld^\infty_\loc(\R^+\times\R^2)$ and $\curl v\ge0$.
Moreover the following hold:
\begin{enumerate}[(i)]
\item if $ v^\circ,\Psi\in W^{1,\infty}(\R^2)^2$, then the solution $ v$ satisfies $\curl v\in\Ld^\infty_\loc(\R^+;\Ld^\infty(\R^2))$, and if in addition $\curl v^\circ\in\Pc(\R^2)$, then there holds $ v\in\Ld^\infty_\loc(\R^+; v^\circ+\Ld^1\cap\Ld^\infty(\R^2)^2)$ and $\curl v\in\Ld^\infty_\loc(\R^+;\Pc\cap\Ld^\infty(\R^2))$;
\item if for some $s\ge0$ we have $ v^\circ,\Psi\in W^{s\vee1,\infty}(\R^2)^2$
and $\curl v^\circ,\curl\Psi\in W^{s,\infty}(\R^2)$,
then for all $0\le u\le s$ the solution $ v$ belongs to $W^{u+1,\infty}_\loc(\R^+;W^{s-u,\infty}(\R^2)^2)$;
\item if for some $s\ge1$ we have $ v^\circ,\Psi\in W^{s,\infty}(\R^2)^2$, $\curl v^\circ\in H^s\cap W^{s,\infty}(\R^2)$, and $\curl\Psi\in W^{s,\infty}(\R^2)$, then the solution $ v$ belongs to $\Ld^\infty_\loc(\R^+; v^\circ+H^s\cap W^{s,\infty}(\R^2)^2)$.\qedhere
\end{enumerate}
\end{prop}

We start with a suitable reduction of equation~\eqref{eq:limeqn2-degen}, making its scalar structure appear.
Assume that $ v\in W^{1,\infty}_\loc(\R^+;\Ld^\infty_\loc(\R^2))$ is a strong solution of~\eqref{eq:limeqn2-degen} with $\curl v\in\Ld^\infty_\loc(\R^+\times\R^2)$.
Since the forcing vector field $\Psi$ is time-independent, equation~\eqref{eq:limeqn2-degen} for $ v$ can be rewritten as follows,
\begin{align*}%\label{eq:limeqn2-degen-re}
\partial_t(\Psi+ v)=-(\Psi+ v)\curl  v,\qquad(\Psi+ v)|_{t=0}=\Psi+ v^\circ,
\end{align*}
which implies for all $x\in\R^2$ and $t\ge0$,
\begin{align}\label{eq:decomp-v-wkappa}
(\Psi+ v^t)(x)=\kappa^t(x)(\Psi+ v^\circ)(x),\qquad \kappa^t(x):=\exp\Big(-\int_0^t\curl v^s(x)\,ds\Big),
\end{align}
together with
the following scalar equation for $\kappa$,
\[\partial_t\kappa=-\kappa\,\curl v,\qquad\kappa|_{t=0}=1.\]
Assuming $\curl\Psi\in\Ld^\infty_\loc(\R^2)$, the definition~\eqref{eq:decomp-v-wkappa} of $\kappa$ in the form $v=-\Psi+\kappa(\Psi+ v^\circ)$ and the assumption $\curl v\in\Ld^\infty_\loc(\R^+\times\R^2)$ ensure that the directional derivative $((\Psi+ v^\circ)^\bot\cdot\nabla)\kappa$ is well-defined in $\Ld^\infty_\loc(\R^+\times\R^2)$, and the above scalar equation for $\kappa$ turns into
\begin{align}\label{eq:degen-MF-2}
\partial_t\kappa=\kappa \,((\Psi+ v^\circ)^\bot\cdot\nabla)\kappa-\kappa^2\curl v^\circ+\kappa(1-\kappa)\curl\Psi,\qquad \kappa|_{t=0}=1.
\end{align}
Along the characteristic curves of the vector field $(\Psi+ v^\circ)^\bot$, this equation takes the form of a Burgers' equation with additional quadratic damping and forcing terms. Although such a Burgers' equation may in general develop discontinuities in finite time (shock waves), we show that this cannot occur for constant initial data $\kappa|_{t=0}=1$ as considered here. Recall that we focus on the case with nonnegative vorticity $\curl v^\circ\ge0$.

\begin{lem}\label{lem:kappa-well-posed}
Let $W\in \Ld^{\infty}_\loc(\R^2)^2$ be $\log$-Lipschitz (that is, $|W(x)-W(y)|\le C|x-y|(1+\log_-(|x-y|))$ for all~$x,y$), and let $f,g\in\Ld^\infty_\loc(\R^2)$ with $f\ge0$. We consider the following Cauchy problem on $\R^+\times\R^2$,
\begin{align}\label{eq:kappa-red}
\partial_t\kappa=\kappa\,(W\cdot \nabla)\kappa- \kappa^2f+\kappa(1-\kappa)g,\qquad\kappa|_{t=0}=1,
\end{align}
where $(W\cdot\nabla)\kappa$ denotes the directional derivative of $\kappa$ along the flow of $W$.
There exists a global strong solution $\kappa\in W^{1,\infty}_\loc(\R^+;\Ld^\infty_\loc(\R^2))\cap \Ld^{\infty}(\R^+\times\R^2)$ %of~\eqref{eq:kappa-red}
with $\frac1\kappa,(W\cdot \nabla)\kappa\in\Ld^\infty_\loc(\R^+\times\R^2)$. This solution is unique in the class
\[\Cc:=\big\{\kappa\in W^{1,\infty}_\loc(\R^+;\Ld_\loc^\infty(\R^2)):(W\cdot\nabla)\kappa\in\Ld^\infty_\loc(\R^+\times\R^2)\big\}.\]
Moreover the following hold:
\begin{enumerate}[(i)]
\item if $f,g\in\Ld^\infty(\R^2)$ and $W\in W^{1,\infty}(\R^2)^2$, then the solution $\kappa$ satisfies $\frac1\kappa,(W\cdot\nabla)\kappa\in\Ld^\infty_\loc(\R^+;\Ld^\infty(\R^2))$, and if in addition $f\in\Ld^1(\R^2)$, then there holds $1-\kappa\in\Ld^\infty_\loc(\R^+;\Ld^1\cap\Ld^\infty(\R^2))$;
\item if for some $s\ge0$ we have $W\in W^{s\vee1,\infty}(\R^2)^2$ and $f,g\in W^{s,\infty}(\R^2)$,
then for all $0\le u\le s$ the solution $\kappa$ belongs to $W_\loc^{u+1,\infty}(\R^+;W^{s-u,\infty}(\R^2))$;
\item if for some $s\ge1$ we have $f\in H^s\cap W^{s,\infty}(\R^2)$, $W\in W^{s,\infty}(\R^2)^2$, and $g\in W^{s,\infty}(\R^2)$, then the solution $\kappa$ satisfies $1-\kappa\in\Ld^\infty_\loc(\R^+;H^s(\R^2))$.\qedhere
\end{enumerate}
\end{lem}

\begin{proof}
Let $W\in\Ld^\infty_\loc(\R^2)^2$ be $\log$-Lipschitz, and let $f,g\in\Ld^\infty_\loc(\R^2)$ with $f\ge0$.
Then the flow $\psi:\R\times\R^2\to\R^2:(s,x)\mapsto\psi^s_x$ associated with the vector field $-W$ is well-defined globally on $\R\times\R^2$,
\[\partial_s\psi^s_x=-W(\psi^s_x),\qquad\psi^s_x|_{s=0}=x.\]
We have $\psi\in C^1(\R;C(\R^2))$, and for all $s\in\R$ the map $\psi^s:\R^2\to\R^2$ is a homeomorphism with inverse $\psi^{-s}$. More precisely, since $W$ is $\log$-Lipschitz, the map $\psi^s$ is a Hölder homeomorphism in the following sense: we have for all $s,x,y$,
\[e^{-e^{C|s|}}(1\wedge|x-y|)^{e^{C|s|}}\le1\wedge|\psi_x^s-\psi_y^s|\le e(1\wedge|x-y|)^{e^{-C|s|}}.\]
We split the proof into three steps.

\medskip
\noindent\step1 Uniqueness.

In this step, we show that for all $x\in\R^d$ and $\sigma^\circ\in\R$ there exists a unique global solution $\sigma_x(\sigma^\circ):\R^+\to\R:t\mapsto\sigma_x^t(\sigma^\circ)$ of
\begin{align}\label{eq:def-sigma-flow}
\partial_t\sigma_x(\sigma^\circ)=1-\int_{\sigma^\circ}^{\sigma_x(\sigma^\circ)}f(\psi_x^s)\exp\Big(-\int_s^{\sigma_x(\sigma^\circ)}(f+g)(\psi_x^u)\,du\Big)ds,\qquad\sigma_x(\sigma^\circ)|_{t=0}=\sigma^\circ,
\end{align}
and that the corresponding map $\sigma^t_x:\R\to\R$ is invertible on $\R$.
In addition, assuming that for some $T>0$ there exists a local strong solution $\kappa\in W^{1,\infty}_\loc([0,T);\Ld^\infty_\loc(\R^2))$ of~\eqref{eq:kappa-red} on $[0,T)\times\R^2$ with $(W\cdot\nabla)\kappa\in \Ld^\infty_\loc([0,T)\times\R^2)$, we show that such a solution $\kappa$ is necessarily given by the following explicit formula,
\begin{align}\label{eq:form-sol-unique}
\kappa^t(x)=1-\int_{(\sigma^t_x)^{-1}(0)}^{0}f(\psi^s_x)\exp\Big(-\int_s^{0}(f+g)(\psi^u_x)du\Big)ds.
\end{align}
This implies the stated uniqueness result.

Setting $\hat\kappa^t_x(s):=\kappa^t(\psi_x^s)$, and noting that $\partial_s\hat\kappa_x^t(s)=-(W\cdot\nabla\kappa^t)(\psi_x^s)$, we deduce by assumption $\hat\kappa_x\in W^{1,\infty}_\loc([0,T)\times\R)$ for almost all $x$.
Picard's existence theorem then ensures the local existence and uniqueness of the flow $\sigma_x$ on $\R$ associated with the vector field $\hat\kappa_x$: for almost all $x$, for all $\sigma^\circ$, there exists $0<T_x(\sigma^\circ)\le T$ and a unique local solution $\sigma_x(\sigma^\circ)\in C^1([0,T_x(\sigma^\circ)))$ of the Cauchy problem
\begin{align}\label{eq:sigma-flow}
\partial_t\sigma^t_x(\sigma^\circ)=\hat\kappa^t_x(\sigma^t_x(\sigma^\circ)),\qquad\sigma^t_x(\sigma^\circ)|_{t=0}=\sigma^\circ.
\end{align}
Now note that by definition the function $t\mapsto\hat\kappa_x^t(\sigma_x^t(\sigma^\circ))$ belongs to $W^{1,\infty}_\loc([0,T_x(\sigma^\circ)))$ and satisfies
\begin{align}\label{eq:hatkappa-simpl}
\partial_t\big(\hat\kappa_x^t(\sigma_x^t(\sigma^\circ))\big)=- \big(\hat\kappa_x^t(\sigma^t_x(\sigma^\circ))\big)^2f(\psi^{\sigma_x^t(\sigma^\circ)}_x)+\hat\kappa^t_x(\sigma^t_x(\sigma^\circ))\big(1-\hat\kappa^t_x(\sigma^t_x(\sigma^\circ))\big)\,g(\psi^{\sigma_x^t(\sigma^\circ)}_x),\\
\hat\kappa_x^t(\sigma_x^t(\sigma^\circ))|_{t=0}=1.\nonumber
\end{align}
For $f,g\in \Ld^\infty_\loc(\R^2)$, this equation admits a unique global solution in $W^{1,\infty}_\loc([0,T_x(\sigma^\circ)))$, which must be given by the explicit formula
\begin{align}\label{eq:form-hatkappa}
\hat\kappa^t_x(\sigma^t_x(\sigma^\circ))=1-\int_{\sigma^0}^{\sigma^t_x(\sigma^\circ)}f(\psi^s_x)\exp\Big(-\int_s^{\sigma^t_x(\sigma^\circ)}(f+g)(\psi^u_x)\,du\Big)ds.
\end{align}
On the one hand, since the positive part $0\vee\hat\kappa_x(\sigma_x(\sigma^\circ))$ belongs to~$W^{1,\infty}_\loc([0,T_x(\sigma^\circ)))$ and also satisfies equation~\eqref{eq:hatkappa-simpl}, we deduce by uniqueness that $\hat\kappa_x(\sigma_x(\sigma^\circ))$ must remain nonnegative. Moreover, formula~\eqref{eq:form-hatkappa} with $f\ge0$ ensures that $\hat\kappa_x(\sigma_x(\sigma^\circ))$ remains bounded above by $1$, so that it is actually $[0,1]$-valued on its domain.
On the other hand, due to formula~\eqref{eq:form-hatkappa}, equation~\eqref{eq:sigma-flow} takes on the following guise,
\begin{align}\label{eq:sigma-flow-simpl}
\partial_t\sigma_x(\sigma^\circ)=Z(\sigma_x(\sigma^\circ),\sigma^\circ),\qquad\sigma_x(\sigma^\circ)|_{t=0}=\sigma^\circ,
\end{align}
where we have set
\[Z(\sigma,\sigma^\circ):=\max\bigg\{0~;~1-\int_{\sigma^0}^{\sigma}f(\psi^s_x)\exp\Big(-\int_s^{\sigma}(f+g)(\psi^u_x)du\Big)ds\bigg\}.\]
As $0\le Z(\sigma,\sigma^\circ)\le1$, we deduce $\sigma^\circ\le\sigma_x^t(\sigma^\circ)\le\sigma^\circ+t$ for all $t\ge0$. Since in addition for $f,g\in\Ld^\infty_\loc(\R^2)$ we have $Z\in W^{1,\infty}_\loc(\R\times\R)$, the flow $\sigma_x(\sigma^\circ)$ must exist globally. We may therefore choose $T_x(\sigma^\circ)=T$ and the representation~\eqref{eq:form-hatkappa} holds for all $0\le t< T$.

It remains to invert~\eqref{eq:form-hatkappa} and deduce the formula~\eqref{eq:form-sol-unique} for the solution $\kappa$ itself.
For that purpose, we need to invert the (non-decreasing) map $\sigma_x^t:\R\to\R$ globally for all $t\ge0$.
Since we have shown $\hat\kappa_x^t(\sigma_x^t(\sigma^\circ))=Z(\sigma_x^t(\sigma^\circ),\sigma^\circ)\in[0,1]$ for all $t\in[0,T)$, equation~\eqref{eq:sigma-flow-simpl} leads to
\begin{multline}\label{eq:form-dersts0}
\partial_t\frac{\partial\sigma_x^t(\sigma^\circ)}{\partial\sigma^\circ}=f(\psi^{\sigma^\circ}_x)\exp\Big(-\int_{\sigma^\circ}^{\sigma^t_x(\sigma^\circ)}(f+g)(\psi^u_x)du\Big)\\
+\frac{\partial\sigma^t_x(\sigma^\circ)}{\partial\sigma^\circ}\bigg(-f(\psi^{\sigma^t_x(\sigma^\circ)}_x)+(f+g)(\psi^{\sigma^t_x(\sigma^\circ)}_x)\int_{\sigma^\circ}^{\sigma^t_x(\sigma^\circ)}f(\psi^s_x)\exp\Big(-\int_s^{\sigma^t_x(\sigma^\circ)}(f+g)(\psi^u_x)du\Big)ds\bigg).
\end{multline}
For all $x,t,\sigma^\circ$, define the compact set $K_x^t(\sigma^\circ):=\overline B+\{\psi_x^{s}:\sigma^\circ\le s\le \sigma^\circ+t\}$, where $\overline B$ is the closed unit Euclidean ball at the origin in $\R^2$.
Hence, for $f,g\in\Ld^\infty_\loc(\R^2)$ with $f\ge0$, we find for almost all $x$, for all $t\in[0,T)$,
\begin{align*}
\partial_t\frac{\partial\sigma_x^t(\sigma^\circ)}{\partial\sigma^\circ}&\ge-\frac{\partial\sigma^t_x(\sigma^\circ)}{\partial\sigma^\circ}\bigg(\|f\|_{\Ld^\infty(K_x^t(\sigma^\circ))}+\|g\|_{\Ld^\infty(K_x^t(\sigma^\circ))}\|f\|_{\Ld^\infty(K_x^t(\sigma^\circ))}\int_{\sigma^\circ}^{\sigma_x^t(\sigma^\circ)}e^{(\sigma_x^t(\sigma^\circ)-s)\|g\|_{\Ld^\infty(K_x^t(\sigma^\circ))}}ds\bigg)\\
&\ge-\frac{\partial\sigma^t_x(\sigma^\circ)}{\partial\sigma^\circ}\,\|f\|_{\Ld^\infty(K_x^t(\sigma^\circ))}\Big(1+e^{(\sigma^t_x(\sigma^\circ)-\sigma^\circ)\|g\|_{\Ld^\infty(K_x^t(\sigma^\circ))}}\Big)\\
&\ge-2\,\frac{\partial\sigma^t_x(\sigma^\circ)}{\partial\sigma^\circ}\,\|f\|_{\Ld^\infty(K_x^t(\sigma^\circ))}\,e^{t\|g\|_{\Ld^\infty(K_x^t(\sigma^\circ))}},
\end{align*}
while from~\eqref{eq:form-hatkappa} we deduce
\begin{align*}
\partial_t\frac{\partial\sigma_x^t(\sigma^\circ)}{\partial\sigma^\circ}&=f(\psi^{\sigma^\circ}_x)\exp\Big(-\int_{\sigma^\circ}^{\sigma^t_x(\sigma^\circ)}(f+g)(\psi^u_x)du\Big)\\
&\hspace{4cm}+\frac{\partial\sigma^t_x(\sigma^\circ)}{\partial\sigma^\circ}\Big( (1-\hat\kappa^t_x(\sigma^t_x(\sigma^\circ)))\,g(\psi_x^{\sigma^t_x(\sigma^\circ)})-\hat\kappa^t_x(\sigma^t_x(\sigma^\circ))\,f(\psi_x^{\sigma^t_x(\sigma^\circ)})\Big)\\
&\le e^{t\|g\|_{\Ld^\infty(K_x^t(\sigma^\circ))}}\|f\|_{\Ld^\infty(K_x^0(\sigma^\circ))}+\frac{\partial\sigma^t_x(\sigma^\circ)}{\partial\sigma^\circ}\|g\|_{\Ld^\infty(K_x^t(\sigma^\circ))}.
\end{align*}
For almost all $x$, for all $t\in[0,T)$, this implies
\begin{align*}
\exp\Big(-2t\|f\|_{\Ld^\infty(K_x^t(\sigma^\circ))}e^{t\|g\|_{\Ld^\infty(K_x^t(\sigma^\circ))}}\Big)\le\frac{\partial\sigma_x^t(\sigma^\circ)}{\partial\sigma^\circ}\le\big(1+t\|f\|_{\Ld^\infty(K_x^0(\sigma^\circ))}\big)e^{t\|g\|_{\Ld^\infty(K_x^t(\sigma^\circ))}},
\end{align*}
which shows that the map $\sigma^t_x:\R\to\R$ is a Lipschitz diffeomorphism, with also
\begin{align}\label{eq:est-inverse-sigma}
\big(1+t\|f\|_{\Ld^\infty(K_x^0(\sigma^\circ))}\big)^{-1}e^{-t\|g\|_{\Ld^\infty(K_x^t(\sigma^\circ))}}\le\frac{\partial(\sigma_x^t)^{-1}(\sigma^\circ)}{\partial\sigma^\circ}\le \exp\Big(2t\|f\|_{\Ld^\infty(K_x^t(\sigma^\circ))}e^{t\|g\|_{\Ld^\infty(K_x^t(\sigma^\circ))}}\Big).
\end{align}
The representation~\eqref{eq:form-hatkappa} applied to $\sigma^\circ=(\sigma_x^t)^{-1}(0)$ then yields the desired result~\eqref{eq:form-sol-unique}.

\medskip
\noindent\step2 Existence.

Let $\kappa,\sigma$ be given by~\eqref{eq:def-sigma-flow}--\eqref{eq:form-sol-unique}.
Noting that for all $\sigma$ there holds
\[0=\partial_t\big((\sigma_x^t)^{-1}(\sigma_x^t(\sigma))\big)=(\partial_t(\sigma_x^t)^{-1})(\sigma_x^t(\sigma))+\partial_t\sigma_x^t(\sigma)\,\frac{\partial(\sigma_x^t)^{-1}}{\partial\sigma_0}(\sigma_x^t(\sigma)),\]
equation~\eqref{eq:sigma-flow} leads to the relation
\[\partial_t(\sigma_x^t)^{-1}(0)=-\kappa^t(x)\frac{\partial(\sigma_x^t)^{-1}}{\partial\sigma_0}(0).\]
The definition~\eqref{eq:form-sol-unique} and the estimate~\eqref{eq:est-inverse-sigma} then ensure that $\kappa\in W^{1,\infty}_\loc(\R^+;\Ld^\infty_\loc(\R^2))$.
We now check that $(W\cdot\nabla)\kappa\in \Ld^\infty_\loc(\R^+\times\R^2)$.
For almost all $x$ and for all $t,\sigma^\circ$, rewriting equation~\eqref{eq:def-sigma-flow} in the form
\begin{align*}
\partial_t\sigma_{\psi_x^r}(\sigma^\circ)=1-\int_{r+\sigma^\circ}^{r+\sigma_{\psi_x^r}(\sigma^\circ)}f(\psi_x^{s})\exp\Big(-\int_{s}^{r+\sigma_{\psi_x^r}(\sigma^\circ)}(f+g)(\psi_x^{u})\,du\Big)ds,
\end{align*}
we easily find that the map $r\mapsto \sigma_{\psi_x^r}^t(\sigma^\circ)$ belongs to $W^{1,\infty}_\loc(\R)$. Using the relation
\[\partial_r(\sigma_{\psi_x^r}^t)^{-1}(0)=-\big(\partial_r\sigma_{\psi_x^r}^t\big)\big((\sigma_{\psi_x^r}^t)^{-1}(0)\big)\,\frac{\partial(\sigma_{\psi_x^r}^t)^{-1}}{\partial\sigma^\circ}(0),\]
it follows that the map $r\mapsto (\sigma_{\psi_x^r}^t)^{-1}(0)$ also belongs to $W^{1,\infty}_\loc(\R)$.
For almost all $x$ and for all $t$, writing $(W\cdot\nabla)\kappa^t(x)=-\partial_r\kappa^t(\psi_x^r)|_{r=0}$, and using the definition~\eqref{eq:form-sol-unique} in the form
\[\kappa^t(\psi_x^r)=1-\int_{r+(\sigma_{\psi_x^r}^t)^{-1}(0)}^rf(\psi_x^{s})\exp\Big(-\int_{s}^r(f+g)(\psi_x^{u})du\Big)ds,\]
we then easily deduce that $(W\cdot\nabla)\kappa\in \Ld^\infty_\loc(\R^+\times\R^2)$. We now check that $\kappa$ is a strong solution of the Cauchy problem~\eqref{eq:kappa-red}.
By construction, the map $t\mapsto\kappa^t(\psi_x^{\sigma_x^t(\sigma^\circ)})$ is given by~\eqref{eq:form-hatkappa} and thus satisfies
\[\partial_t\big(\kappa^t(\psi_x^{\sigma_x^t(\sigma^\circ)})\big)=-\big(\kappa^t(\psi_x^{\sigma_x^t(\sigma^\circ)})\big)^2f(\psi_x^{\sigma_x^t(\sigma^\circ)})+\kappa^t(\psi_x^{\sigma_x^t(\sigma^\circ)})\big(1-\kappa^t(\psi_x^{\sigma_x^t(\sigma^\circ)})\big)\,g(\psi_x^{\sigma_x^t(\sigma^\circ)}),\]
or alternatively,
\[\big(\partial_t\kappa^t-\kappa^t\,(W\cdot\nabla)\kappa^t\big)(\psi_x^{\sigma_x^t(\sigma^\circ)})=\big(-(\kappa^t)^2f+\kappa^t(1-\kappa^t)g\big)(\psi_x^{\sigma_x^t(\sigma^\circ)}).\]
As this holds for almost all $x$ and for all $\sigma^\circ$, we indeed deduce that $\kappa$ is a strong solution of~\eqref{eq:kappa-red}.
It remains to check that $\frac1\kappa\in \Ld^\infty_\loc(\R^+\times\R^2)$. For that purpose, we note that equation~\eqref{eq:kappa-red} implies
\begin{align*}
\big|\partial_t\big(|\kappa^t(x)|^{-1}\big)\big|\le|\kappa^t(x)|^{-1}\big(|(W\cdot\nabla)\kappa^t(x)|+(1+|\kappa^t(x)|)|g(x)|\big)+|f(x)|,
\end{align*}
which easily implies by a Grönwall argument that $\frac1\kappa\in \Ld^\infty_\loc(\R^+\times\R^2)$.

\medskip
\noindent\step3 Regularity and integrability.

The additional regularity statement~(ii) in $W^{s,\infty}(\R^2)$ is a straightforward consequence of formulas~\eqref{eq:form-sol-unique}--\eqref{eq:def-sigma-flow}, together with the identity~\eqref{eq:form-dersts0} and the estimate~\eqref{eq:est-inverse-sigma}. Also note that for $f,g\in\Ld^\infty(\R^2)$ and $W\in W^{1,\infty}(\R^2)$ the argument in Step~2 ensures that $\frac1\kappa,(W\cdot\nabla)\kappa\in\Ld^\infty_\loc(\R^+;\Ld^\infty(\R^2))$.

We now turn to the additional integrability~(i) for $1-\kappa$.
Assume that $f\in\Ld^1\cap \Ld^{\infty}(\R^2)$, $W\in W^{1,\infty}(\R^2)$, and $g\in \Ld^{\infty}(\R^2)$.
For all $R\ge1$, denote by $\chi_R(x):=e^{-|x|/R}$ the exponential cut-off function at scale $R$. We compute
\begin{align*}
\partial_t\int_{\R^2}\chi_R|1-\kappa^t|\le\int_{\R^2}\chi_R\kappa^tW\cdot\nabla|1-\kappa^t|+\int_{\R^2}\chi_R(\kappa^t)^2f+\int_{\R^2}\chi_R|\kappa^tg||1-\kappa^t|,
\end{align*}
and hence, after integration by parts, using the property $|\nabla\chi_R|\le\chi_R$ of the exponential cut-off function, for all $R\ge1$,
\begin{align*}
\partial_t\int_{\R^2}\chi_R|1-\kappa^t|&\le \|\kappa^t\|_{\Ld^\infty}^2\|f\|_{\Ld^1}+(\|\chi_R^{-1}\Div(\kappa^t\chi_R W)\|_{\Ld^\infty}+\|\kappa^t g\|_{\Ld^\infty})\int_{\R^2}\chi_R|1-\kappa^t|\\
&\le \|\kappa^t\|_{\Ld^\infty}^2\|f\|_{\Ld^1}+(\|(W\cdot\nabla)\kappa^t\|_{\Ld^\infty}+\|\kappa^t\|_{\Ld^{\infty}}\|W\|_{W^{1,\infty}}+\|\kappa^t\|_{\Ld^{\infty}}\|g\|_{\Ld^\infty})\int_{\R^2}\chi_R|1-\kappa^t|.
\end{align*}
Applying the Grönwall inequality, and letting $R\uparrow\infty$, we deduce $1-\kappa\in\Ld^\infty_\loc(\R^+;\Ld^1(\R^2))$.

We finally turn to the $H^s$-regularity.
Let $s\ge1$ be fixed. Assume that $f\in H^s\cap W^{s,\infty}(\R^2)$, $W\in W^{s,\infty}(\R^2)^2$, $g\in W^{s,\infty}(\R^2)$.
For all $R\ge1$, denote by $\tilde\chi_R(x):=\exp(-(1+|x|^2)^{1/2}/R)$ a smooth exponential cut-off function at scale $R$.
We compute
\begin{multline}\label{eq:first-compute-dt-1kappa}
\partial_t\|\tilde\chi_R(1-\kappa^t)\|_{H^s}^2=-2\int_{\R^2}\langle\nabla\rangle^s\big(\tilde\chi_R(1-\kappa^t)\big)\langle\nabla\rangle^s\big(\kappa^t\tilde\chi_RW\cdot\nabla\kappa^t\big)\\
-2\int_{\R^2}\langle\nabla\rangle^s\big(\tilde\chi_R(1-\kappa^t)\big)\langle\nabla\rangle^s\big(-\tilde\chi_R(\kappa^t)^2f+\tilde\chi_R\kappa^t(1-\kappa^t)g\big).
\end{multline}
Decomposing
\begin{multline*}
-2\langle\nabla\rangle^s\big(\kappa^t\tilde\chi_RW\cdot\nabla\kappa^t\big)=2[\langle\nabla\rangle^s,\kappa^tW\cdot]\nabla(\tilde\chi_R(1-\kappa^t))+2\kappa^tW\cdot\nabla\langle\nabla\rangle^s(\tilde\chi_R(1-\kappa^t))\\
-2\langle\nabla\rangle^s\big((1-\kappa^t)\kappa^tW\cdot\nabla\tilde\chi_R\big),
\end{multline*}
we find, after integration by parts in the second right-hand side term,
\begin{multline*}
\partial_t\|\tilde\chi_R(1-\kappa^t)\|_{H^s}^2=2\int_{\R^2}\langle\nabla\rangle^s\big(\tilde\chi_R(1-\kappa^t)\big)[\langle\nabla\rangle^s,\kappa^tW\cdot]\nabla(\tilde\chi_R(1-\kappa^t))-\int_{\R^2}|\langle\nabla\rangle^s(\tilde\chi_R(1-\kappa^t))|^2\Div(\kappa^tW)\\
-2\int_{\R^2}\langle\nabla\rangle^s\big(\tilde\chi_R(1-\kappa^t)\big)\langle\nabla\rangle^s\big((1-\kappa^t)\kappa^tW\cdot\nabla\tilde\chi_R-\tilde\chi_R(\kappa^t)^2f+\tilde\chi_R\kappa^t(1-\kappa^t)g\big),
\end{multline*}
and hence,
\begin{multline*}
\partial_t\|\tilde\chi_R(1-\kappa^t)\|_{H^s}\lesssim\|[\langle\nabla\rangle^s,\kappa^tW\cdot]\nabla(\tilde\chi_R(1-\kappa^t))\|_{\Ld^2}+\|\kappa^t\|_{W^{s,\infty}}^2\|\tilde\chi_Rf\|_{H^s}\\
+\big(\|\Div(\kappa^tW)\|_{\Ld^\infty}+\|\tilde\chi_R^{-1}\kappa^tW\cdot\nabla\tilde\chi_R\|_{W^{s,\infty}}+\|\kappa^tg\|_{W^{s,\infty}}\big)\|\tilde\chi_R(1-\kappa^t)\|_{H^s}.
\end{multline*}
Applying the Kato-Ponce commutator estimate~\cite[Lemma~X1]{Kato-Ponce-88} in the form~\eqref{eq:kato-ponce-commutator-statement} with $s\ge1$ in order to estimate the first right-hand side term, we find
\begin{multline*}
\partial_t\|\tilde\chi_R(1-\kappa^t)\|_{H^s}\lesssim\big(\|\kappa^tW\|_{W^{s,\infty}}+\|\tilde\chi_R^{-1}\kappa^tW\cdot\nabla\tilde\chi_R\|_{W^{s,\infty}}+\|\kappa^tg\|_{W^{s,\infty}}\big)\|\tilde\chi_R(1-\kappa^t)\|_{H^s}+\|\kappa^t\|_{W^{s,\infty}}^2\|\tilde\chi_Rf\|_{H^s},
\end{multline*}
and thus, for all $R\ge1$, using the properties of the smooth exponential cut-off function $\tilde\chi_R$,
\begin{align*}
\partial_t\|\tilde\chi_R(1-\kappa^t)\|_{H^s}\lesssim\|\kappa^t\|_{W^{s,\infty}}\|(W,g)\|_{W^{s,\infty}}\|\tilde\chi_R(1-\kappa^t)\|_{H^s}+\|\kappa^t\|_{W^{s,\infty}}^2\|f\|_{H^s},
\end{align*}
Applying the Grönwall inequality, using the regularity result for the solution $\kappa$ in $W^{s,\infty}(\R^2)$, and letting $R\uparrow\infty$, this implies $1-\kappa\in\Ld^\infty_\loc(\R^+;H^s(\R^2))$.
\end{proof}

We may now conclude with the proof of Proposition~\ref{prop:main-deg}.

\begin{proof}[Proof of Proposition~\ref{prop:main-deg}]
Let $ v^\circ,\Psi\in\Ld^\infty_\loc(\R^2)^2$ be $\log$-Lipschitz vector fields with  $\curl v^\circ,\curl\Psi\in\Ld^\infty_\loc(\R^2)$ and $\curl v^\circ\ge0$.
We start with the existence part.
By Lemma~\ref{lem:kappa-well-posed} with $W:=(\Psi+ v^\circ)^\bot$, $f:=\curl v^\circ$, and $g:=\curl\Psi$, there exists a global strong solution $\kappa\in W^{1,\infty}_\loc(\R^+;\Ld^\infty_\loc(\R^2))$ of~\eqref{eq:degen-MF-2} with $\frac1\kappa,((\Psi+ v^\circ)^\bot\cdot\nabla)\kappa\in\Ld^\infty_\loc(\R^+\times\R^2)$. Then the function $ v:=-\Psi+\kappa(\Psi+ v^\circ)\in W^{1,\infty}_\loc(\R^+;\Ld^\infty_\loc(\R^2))$ is by construction a global strong solution of~\eqref{eq:limeqn2-degen} with initial data $ v^\circ$ and with $\curl v\in\Ld^\infty_\loc(\R^+\times\R^2)$. The additional regularity statements follow from the corresponding statements for $\kappa$ in Lemma~\ref{lem:kappa-well-posed} together with the representation $ v- v^\circ=-(1-\kappa)( v^\circ+\Psi)$.

We now turn to the uniqueness part. 
Assume that $ v_1, v_2\in \Ld^{\infty}_\loc([0,T)\times\R^2)$ are strong solutions of~\eqref{eq:limeqn2-degen} on $[0,T)\times\R^2$ with $\curl v_1,\curl v_2\in\Ld^\infty_\loc([0,T)\times\R^2)$ and $\curl v_1,\curl v_2\ge0$.
From~\eqref{eq:decomp-v-wkappa}, it follows that for $i=1,2$ we have $ v_i=-\Psi+\kappa_i(\Psi+ v^\circ)$ where $\kappa_i$ is given by
\[\kappa_i^t(x):=\exp\Big(-\int_0^t\curl v_i^s(x)\,ds\Big).\]
As $ v_i$ is a strong solution of~\eqref{eq:limeqn2-degen} on $[0,T)\times\R^2$, we deduce that $\kappa_i$ is a strong solution of equation~\eqref{eq:degen-MF-2} on $[0,T)\times\R^2$, and the boundedness assumption on $\curl v_i$ implies that $\kappa_i\in W^{1,\infty}_\loc([0,T);\Ld^\infty_\loc(\R^2))$ satisfies $\frac1{\kappa_i},((\Psi+v^\circ)^\bot\cdot\nabla)\kappa_i\in\Ld^\infty_\loc([0,T)\times\R^2)$.
The conclusion $\kappa_1=\kappa_2$ then follows from the uniqueness statement in Lemma~\ref{lem:kappa-well-posed}.
\end{proof}

\section{Appendix: Proof of the preliminary results}\label{app:proofs-prel}
In this appendix, we prove the various preliminary results stated in Section~\ref{chap:prelim}.
We start with the a priori estimate for transport equations, stated in Lemma~\ref{lem:katoponce}.

\begin{proof}[Proof of Lemma~\ref{lem:katoponce}]
We split the proof into two steps: we first prove~\eqref{eq:katoponcecom} as a corollary of the celebrated Kato-Ponce commutator estimate, and then we check estimate~\eqref{eq:tsph-1}, which is but a straightforward observation.

\medskip
\noindent\step1 Proof of~\eqref{eq:katoponcecom}.

Let $s\ge0$. The time derivative of the $H^s$-norm of the solution $\rho$ can be computed as follows, using the notation $\langle \nabla\rangle:=(1+|\nabla|^2)^{1/2}$,
\begin{align*}
\partial_t\|\rho^t\|_{H^s}^2=2\int (\langle\nabla\rangle^s\rho^t) (\langle\nabla\rangle^s\Div(\rho^t w^t))&=2\int (\langle\nabla\rangle^s\rho^t) [\langle\nabla\rangle^s\Div,w^t]\rho^t+2\int (\langle\nabla\rangle^s\rho^t) (w^t\cdot \nabla\langle\nabla\rangle^s\rho^t)\\
&=2\int (\langle\nabla\rangle^s\rho^t) [\langle\nabla\rangle^s\Div,w^t]\rho^t-\int |\langle\nabla\rangle^s\rho^t|^2\Div w^t\\
&\le 2 \|\rho^t\|_{H^s} \|[\langle\nabla\rangle^s\Div,w^t]\rho^t\|_{\Ld^2}+\|(\Div w^t)_-\|_{\Ld^\infty}\|\rho^t\|_{H^s}^2,
\end{align*}
and hence,
\begin{align}\label{eq:lem-katoponce-preapply}
\partial_t\|\rho^t\|_{H^s}&\le \|[\langle\nabla\rangle^s\Div,w^t-W]\rho^t\|_{\Ld^2}+\|[\langle\nabla\rangle^s\Div,W]\rho^t\|_{\Ld^2}+\frac12\|(\Div w^t)_-\|_{\Ld^\infty}\|\rho^t\|_{H^s}.
\end{align}
Now we recall the following forms of the Kato-Ponce commutator estimate~\cite[Lemma~X1]{Kato-Ponce-88} (see e.g.~\cite{Li-16}):
given $p\in(1,\infty)$, and $\frac1{p_i}+\frac1{q_i}=\frac1p$ with $p_i,q_i\in(1,\infty]$ for $i=1,2$, we have for all $f,g\in C^\infty_c(\R^d)$,
\begin{align*}
\|[\langle\nabla\rangle^s\nabla,f] g\|_{\Ld^p}\lesssim_{s,p,p_1,p_2} \|f\|_{W^{s+1,q_1}}\|g\|_{\Ld^{p_1}}+\|\nabla f\|_{\Ld^{p_2}}\|g\|_{W^{s,q_2}},
\end{align*}
and also
\begin{align}\label{eq:kato-ponce-commutator-statement}
\|[\langle\nabla\rangle^s,f]\nabla g\|_{\Ld^p}\lesssim_{s,p,p_1,p_2} \|f\|_{W^{s,q_1}}\|g\|_{W^{1,p_1}}+\mathds1_{s\ge1}\|\nabla f\|_{\Ld^{p_2}}\|g\|_{W^{s,q_2}}.
\end{align}
Together with the Kato-Ponce inequality of Lemma~\ref{lem:katoponce-1}, these estimates yield on the one hand
\begin{align*}
\|[\langle\nabla\rangle^s\Div,W]\rho^t\|_{\Ld^2}&\lesssim_s\|W\|_{W^{s+1,\infty}}\|\rho^t\|_{\Ld^2}+\|\nabla W\|_{\Ld^{\infty}}\|\rho^t\|_{H^s},
\end{align*}
and
\begin{align*}
\|[\langle\nabla\rangle^s\Div,w^t-W]\rho^t\|_{\Ld^2}\lesssim_s\|\rho^t\|_{\Ld^\infty}\|w^t-W\|_{H^{s+1}}+\|\nabla(w^t-W)\|_{\Ld^{\infty}}\|\rho^t\|_{H^s},
\end{align*}
and on the other hand,
\begin{align*}
\|[\langle\nabla\rangle^s\Div,w^t-W]\rho^t\|_{\Ld^2}&\le\|\rho^t\Div(w^t-W)\|_{H^s}+\|[\langle\nabla\rangle^s,(w^t-W)\cdot\,]\,\nabla\rho^t\|_{\Ld^2}\\
&\lesssim_s\|\nabla(w^t-W)\|_{\Ld^{\infty}}\|\rho^t\|_{H^s}+\|\rho^t\|_{\Ld^\infty}\|\Div(w^t-W)\|_{H^s}+\|\rho^t\|_{W^{1,\infty}}\|w^t-W\|_{H^{s}}.
\end{align*}
Injecting these estimates into~\eqref{eq:lem-katoponce-preapply}, the result~\eqref{eq:katoponcecom} follows.

\medskip
\noindent\step2 Proof of~\eqref{eq:tsph-1}.
\nopagebreak

Let $\e>0$. We denote by $\hat u$ the Fourier transform of a function $u$ on $\R^d$. Set $G:=\rho w$, so that the equation for $\rho$ takes the form $\partial_t\rho=\Div G$. Rewriting this equation in Fourier space and testing it against $(\e+|\xi|)^{-2}(\hat\rho^t-\hat\rho^\circ)(\xi)$, we find
\begin{align*}
\partial_t\int (\e+|\xi|)^{-2}|\hat \rho^t(\xi)-\hat\rho^\circ(\xi)|^2d\xi&=2i\int (\e+|\xi|)^{-2}\xi\cdot \hat G^t(\xi)(\overline{\hat \rho^t(\xi)-\hat\rho^\circ(\xi)})d\xi\\
&\le2\int (\e+|\xi|)^{-1}|\hat \rho^t(\xi)-\hat\rho^\circ(\xi)||\hat G^t(\xi)|d\xi,
\end{align*}
and hence, by the Cauchy-Schwarz inequality,
\begin{align*}
\partial_t\bigg(\int (\e+|\xi|)^{-2}|\hat \rho^t(\xi)-\hat\rho^\circ(\xi)|^2d\xi\bigg)^{1/2}&\le\bigg(\int|\hat G^t(\xi)|^2d\xi\bigg)^{1/2}.
\end{align*}
Integrating in time and letting $\e\downarrow0$, we obtain
\[\| \rho^t-\rho^\circ\|_{\dot H^{-1}}\le \|G\|_{\Ld^1_t\Ld^2}\le \|\rho\|_{\Ld^\infty_t\Ld^2}\|w\|_{\Ld^1_t\Ld^\infty},\]
that is,~\eqref{eq:tsph-1}.
\end{proof}

We turn to the proof of the a priori estimates for transport-diffusion equations, stated in Lemma~\ref{lem:parreg+tsp}.

\begin{proof}[Proof of Lemma~\ref{lem:parreg+tsp}]
We split the proof into three steps, proving items (i), (ii), and~(iii) separately.

\medskip
\noindent\step1 Proof of~(i).

\nopagebreak
Denote $G:=g-w\nabla h$, so that $w$ satisfies $\partial_tw-\triangle w=\Div G$. Set $\langle\xi\rangle:=(1+|\xi|^2)^{1/2}$, and let $\hat u$ denote the Fourier transform of a function $u$ on $\R^d$. Let $s\ge0$ be fixed, and assume that $\nabla h,w,g$ are as in the statement of~(i) (which implies $G\in\Ld^2_\loc([0,T);H^s(\R^d))$ as shown below). In this step, we use the notation $\lesssim$ for $\le$ up to a constant $C$ as in the statement.
For all $\e>0$, rewriting the equation for $w$ in Fourier space and testing it against $(\e+|\xi|)^{-2}\langle\xi\rangle^{2s}\partial_t \hat w(\xi)$, we obtain
\begin{align*}
\int(\e+|\xi|)^{-2}\langle\xi\rangle^{2s}|\partial_t\hat w^t(\xi)|^2d\xi+\frac12\int\frac{|\xi|^2}{(\e+|\xi|)^2}\langle\xi\rangle^{2s}\partial_t|\hat w^t(\xi)|^2d\xi=i\int(\e+|\xi|)^{-2}\langle\xi\rangle^{2s} \xi\cdot\hat G^t(\xi)\overline{\partial_t\hat w^t(\xi)}d\xi,
\end{align*}
and hence, integrating over $[0,t]$, and using the inequality $2xy\le x^2+y^2$,
\begin{align*}
&\int_0^t\int(\e+|\xi|)^{-2}\langle\xi\rangle^{2s}|\partial_u\hat w^u(\xi)|^2d\xi du+\frac12\int\frac{|\xi|^2}{(\e+|\xi|)^2}\langle\xi\rangle^{2s}|\hat w^t(\xi)|^2d\xi\\
=~&\frac12\int\frac{|\xi|^2}{(\e+|\xi|)^2}\langle\xi\rangle^{2s}|\hat w^\circ(\xi)|^2d\xi+i\int_0^t\int(\e+ |\xi|)^{-2}\langle\xi\rangle^{2s}\xi\cdot\hat G^u(\xi)\overline{\partial_u\hat w^u(\xi)}d\xi du\\
\le~& \frac12\int\langle\xi\rangle^{2s}|\hat w^\circ(\xi)|^2d\xi+\frac12\int_0^t\int \langle\xi\rangle^{2s}|\hat G^u(\xi)|^2d\xi du+\frac12\int_0^t\int (\e+|\xi|)^{-2}\langle\xi\rangle^{2s}|\partial_u\hat w^u(\xi)|^2d\xi du.
\end{align*}
Absorbing in the left-hand side the last right-hand side term, and letting $\e\downarrow0$, it follows that
\begin{align*}
\int\langle\xi\rangle^{2s}|\hat w^t(\xi)|^2d\xi\le\int\langle\xi\rangle^{2s}|\hat w^\circ(\xi)|^2d\xi+\int_0^t\int \langle\xi\rangle^{2s}|\hat G^u(\xi)|^2d\xi du,
\end{align*}
or equivalently
\[\|w^t\|_{H^s}\le\|w^\circ\|_{H^s}+\|G\|_{\Ld^2_tH^s}.\]
Lemma~\ref{lem:katoponce-1} yields
\begin{align*}
\|G\|_{\Ld^2_tH^s}\le\|g\|_{\Ld^2_tH^s}+\|w\nabla h\|_{\Ld^2_tH^s}&\lesssim\|g\|_{\Ld^2_tH^s}+\|\nabla h\|_{W^{s,\infty}}\|w\|_{\Ld^2_t\Ld^2}+\|\nabla h\|_{\Ld^\infty}\|w\|_{\Ld^2_tH^s}\\
&\lesssim\|g\|_{\Ld^2_tH^s}+\|w\|_{\Ld^2_tH^s},
\end{align*}
so that we obtain
\[\|w^t\|_{H^s}^2\lesssim \|w^\circ\|_{H^s}^2+\|g\|_{\Ld^2_tH^s}^2+\int_0^t\|w^u\|_{H^s}^2du,\]
and item~(i) now follows from the Grönwall inequality.

\medskip
\noindent\step2 Proof of~(ii).

Set again $G:=g-w\nabla h$, and let $\nabla h,w^\circ,w,g$ be as in the statement of~(ii). For all $\e>0$, rewriting the equation for $w$ in Fourier space and then integrating it against $(\e+|\xi|)^{-2}(\hat w^t-\hat w^\circ)(\xi)$, we may estimate
\begin{align*}
&\partial_t\int(\e+|\xi|)^{-2}|(\hat w^t-\hat w^\circ)(\xi)|^2d\xi=2\int(\e+|\xi|)^{-2}\overline{(\hat w^t-\hat w^\circ)(\xi)}\partial_t\hat w^t(\xi)d\xi\\
\le~&-2\int\frac{|\xi|^2}{(\e+|\xi|)^2}|(\hat w^t-\hat w^\circ)(\xi)|^2+2\int\frac{|\xi|^2}{(\e+|\xi|)^2}|(\hat w^t-\hat w^\circ)(\xi)||\hat w^\circ(\xi)|+2\int(\e+ |\xi|)^{-1}|(\hat w^t-\hat w^\circ)(\xi)||\hat G^t(\xi)|d\xi\\
\le~&\int\frac{|\xi|^2}{(\e+|\xi|)^2}|\hat w^\circ(\xi)|^2+\int(\e+ |\xi|)^{-2}|(\hat w^t-\hat w^\circ)(\xi)|^2d\xi+\int(1+ |\xi|^2)^{-1}|\hat G^t(\xi)|^2d\xi,
\end{align*}
that is
\begin{align*}
\partial_t\int(\e+|\xi|)^{-2}|(\hat w^t-\hat w^\circ)(\xi)|^2d\xi&\le\int(\e+ |\xi|)^{-2}|(\hat w^t-\hat w^\circ)(\xi)|^2d\xi+\|w^\circ\|_{\Ld^2}^2+\|G^t\|_{H^{-1}}^2,
\end{align*}
and hence by the Grönwall inequality,
\begin{align*}
\int(\e+|\xi|)^{-2}|(\hat w^t-\hat w^\circ)(\xi)|^2d\xi&\le e^{t}\big(t\|w^\circ\|_{\Ld^2}^2+\|G\|_{\Ld^2_tH^{-1}}^2\big).
\end{align*}
Letting $\e\downarrow0$, it follows that $w^t-w^\circ\in \dot H^{-1}(\R^2)$ with
\begin{align*}
\|w^t-w^\circ\|_{\dot H^{-1}}\le e^{Ct}(\|w^\circ\|_{\Ld^2}+\|G\|_{\Ld^2_tH^{-1}})\le e^{Ct}(\|w^\circ\|_{\Ld^2}+\|g\|_{\Ld^2_tH^{-1}}+\|\nabla h\|_{\Ld^\infty}\|w\|_{\Ld^2_t\Ld^2}).
\end{align*}
Combining this with~(i) for $s=0$, item~(ii) follows.

\medskip
\noindent\step3 Proof of~(iii).
\nopagebreak

Let $1\le p,q\le\infty$, and assume that $w\in\Ld^p([0,T);\Ld^q(\R^d))$, $\nabla h\in\Ld^\infty(\R^d)$, and $g\in\Ld^p([0,T);\Ld^q(\R^d))$. In this step, we use the notation $\lesssim$ for $\le$ up to a constant $C$ as in the statement.
Denoting by $\Gamma^t(x):=Ct^{-d/2}e^{-|x|^2/(2t)}$ the heat kernel, Duhamel's representation formula yields
\[w^t(x)=\Gamma^t\ast w^\circ(x)+\phi_g^t(x)-\int_0^t\int\nabla \Gamma^{u}(y)\cdot \nabla h(x-y)\,w^{t-u}(x-y)dydu,\]
where we have set
\[\phi_g^t(x):=\int_0^t\int\nabla\Gamma^u(y)\cdot g^{t-u}(x-y)dydu.\]
We find by the triangle inequality
\[\|w^t\|_{\Ld^q}\le \|w^\circ\|_{\Ld^q}\int|\Gamma^t(y)|dy+\|\phi_g^t\|_{\Ld^q}+\|\nabla h\|_{\Ld^\infty}\int_0^t \|w^{t-u}\|_{\Ld^q}\int|\nabla \Gamma^{u}(y)|dydu,\]
hence by a direct computation
\[\|w^t\|_{\Ld^q}\lesssim \|w^\circ\|_{\Ld^q}+\|\phi_g^t\|_{\Ld^q}+\int_0^t \|w^{t-u}\|_{\Ld^q}u^{-1/2}du.\]
Integrating with respect to $t$, and using the triangle and the Hölder inequalities, we find
\begin{align*}
\|w\|_{\Ld^p_t\Ld^q}&\lesssim t^{1/p}\|w^\circ\|_{\Ld^q}+\|\phi_g\|_{\Ld^p_t\Ld^q}+\bigg(\int_0^t\Big(\int_0^t \mathds1_{u<v}\|w^{v-u}\|_{\Ld^q}u^{-1/2}du\Big)^pdv\bigg)^{1/p}\\
&\lesssim t^{1/p}\|w^\circ\|_{\Ld^q}+\|\phi_g\|_{\Ld^p_t\Ld^q}+\int_0^t\|w\|_{\Ld^p_u\Ld^q}(t-u)^{-1/2}du\\
&\lesssim t^{1/p}\|w^\circ\|_{\Ld^q}+\|\phi_g\|_{\Ld^p_t\Ld^q}+(1-r'/2)^{-1/r'}t^{\frac12-\frac1r}\bigg(\int_0^t\|w\|_{\Ld^p_u\Ld^q}^rdu\bigg)^{1/r},
\end{align*}
for all $r>2$.
Noting that $(1-r'/2)^{-1/r'}\lesssim1+(r-2)^{-1/2}$, and optimizing in $r$, the Grönwall inequality then leads to
\begin{align}\label{eq:resparregzetapre}
\|w\|_{\Ld^p_t\Ld^q}&\lesssim (t^{1/p}\|w^\circ\|_{\Ld^q}+\|\phi_g\|_{\Ld^p_t\Ld^q})\exp\Big(\inf_{2<r<\infty}\frac{C^r}r(1+(r-2)^{-r/2})\,t^{r/2}\Big).
\end{align}
Now it remains to estimate the norm of $\phi_g$.
A similar computation as above yields $\|\phi_g\|_{\Ld^p_t\Ld^q}\lesssim t^{1/2}\|g\|_{\Ld^p_t\Ld^q}$,
but a more careful estimate is needed.
For $1\le s\le q$, we may estimate by the Hölder inequality
\begin{align*}
|\phi_g^t(x)|&\le\int_0^t\bigg(\int|\nabla \Gamma^{u}|^{s'/2}\bigg)^{1/s'}\bigg(\int|\nabla \Gamma^{u}(x-y)|^{s/2} |g^{t-u}(y)|^sdy\bigg)^{1/s}du,
\end{align*}
and hence, by the triangle inequality,
\begin{align*}
\|\phi_g^t\|_{\Ld^q}&\le\int_0^t\bigg(\int|\nabla \Gamma^{u}|^{s'/2}\bigg)^{1/s'}\bigg(\int|\nabla \Gamma^{u}|^{q/2}\bigg)^{1/q}\bigg(\int |g^{t-u}|^s\bigg)^{1/s}du.
\end{align*}
Assuming that $\kappa:=\frac d2\big(\frac1d+\frac 1q-\frac1s\big)>0$ (note that $\kappa\le1/2$ follows from the choice $s\le q$), a direct computation then yields
\begin{align*}
\|\phi_g^t\|_{\Ld^q}&\lesssim \int_0^tu^{\kappa-1}\|g^{t-u}\|_{\Ld^s}du.
\end{align*}
Integrating with respect to $t$, we find by the triangle inequality
\begin{align*}
\|\phi_g\|_{\Ld^p_t\Ld^q}&\lesssim \int_0^tu^{\kappa-1}\bigg(\int_0^{t-u}\|g^{v}\|_{\Ld^s}^pdv\bigg)^{1/p}du~\lesssim \kappa^{-1}t^\kappa\|g^{v}\|_{\Ld^p_t\Ld^s},
\end{align*}
and the result~(iii) follows from this together with~\eqref{eq:resparregzetapre}.
\end{proof}

We turn to the proof of the potential estimates in $\Ld^\infty(\R^d)$, stated in Lemma~\ref{lem:singint-1}.

\begin{proof}[Proof of Lemma~\ref{lem:singint-1}]
Recall that $-\triangle^{-1} w=g_d\ast  w$, where we define $g_d(x):=c_d|x|^{2-d}$ if $d>2$ and $g_2(x):=-c_2\log|x|$ if $d=2$. The stated results are based on suitable decompositions of this Green's integral.
We split the proof into three steps, separately proving items~(i), (ii) and~(iii).

\medskip
\noindent\step1 Proof of~(i).
\nopagebreak

Let $0<\gamma\le\Gamma<\infty$. The obvious estimate $|\nabla\triangle^{-1} w(x)|\lesssim \int|x-y|^{1-d}| w(y)|dy$ may be decomposed as
\begin{align*}
|\nabla\triangle^{-1} w(x)|&\lesssim\int_{|x-y|<\gamma}|x-y|^{1-d}| w(y)|dy+\int_{\gamma<|x-y|<\Gamma}|x-y|^{1-d}| w(y)|dy+\int_{|x-y|>\Gamma}|x-y|^{1-d}| w(y)|dy.
\end{align*}
Let $1\le p<d<q\le\infty$. We use the H\"older inequality with exponents $(q/(q-1),q)$ for the first term, $(d/(d-1),d)$ for the second, and $(p/(p-1),p)$ for the third, which yields after straightforward computations
\begin{align*}
|\nabla\triangle^{-1} w(x)|&\lesssim (q'(1-d/q))^{-1/q'}\gamma^{1-d/q}\| w\|_{\Ld^q}+(\log (\Gamma/\gamma))^{(d-1)/d}\| w\|_{\Ld^{d}}+(p'(d/p-1))^{-1/p'}\Gamma^{1-d/p}\| w\|_{\Ld^p}.
\end{align*}
Item~(i) now easily follows, choosing $\gamma^{1-d/q}=\| w\|_{\Ld^{d}}/\| w\|_{\Ld^q}$ and $\Gamma^{d/p-1}=\| w\|_{\Ld^p}/\| w\|_{\Ld^{d}}$, noting that $\gamma\le\Gamma$ follows from interpolation of $\Ld^{d}$ between $\Ld^p$ and $\Ld^\infty$, and observing that
\[(q'(1-d/q))^{-1/q'}\lesssim (1-d/q)^{-1+1/d},\qquad (p'(d/p-1))^{-1/p'}\lesssim (1-p/d)^{-1+1/d}.\]

\medskip
\noindent\step2 Proof of~(ii).
\nopagebreak

Let $0<\gamma\le1\le\Gamma<\infty$, and let $\chi_\Gamma$ denote a cut-off function with $\chi_\Gamma=0$ on $B_\Gamma$, $\chi_\Gamma=1$ outside $B_{\Gamma+1}$, and $|\nabla\chi_\Gamma|\le2$. We may then decompose
\begin{align*}
-\nabla\triangle^{-1} w(x)=~&\int_{|x-y|<\gamma}\nabla g_d(x-y) w(y)dy+\int_{\gamma\le|x-y|\le\Gamma}\nabla g_d(x-y) w(y)dy\\
&+\int_{\Gamma\le|x-y|\le\Gamma+1}\nabla g_d(x-y)(1-\chi_\Gamma(x-y)) w(y)dy+\int_{|x-y|\ge\Gamma}\nabla g_d(x-y)\chi_\Gamma(x-y) w(y)dy.
\end{align*}
Using $ w=\Div\xi$ and integrating by parts, the last term becomes
\begin{align*}
\int\nabla g_d(x-y)\chi_\Gamma(x-y) w(y)dy&=-\int\nabla g_d(x-y) \otimes\nabla\chi_\Gamma(x-y) \cdot\xi(y)dy-\int\chi_\Gamma(x-y) \nabla^2g_d(x-y)\cdot\xi(y)dy.
\end{align*}
Choosing $\Gamma=1$, we may then estimate
\begin{align*}
|\nabla\triangle^{-1} w(x)|\lesssim~&\int_{|x-y|<\gamma}|x-y|^{1-d} |w(y)|dy+\int_{\gamma\le|x-y|\le2}|x-y|^{1-d} |w(y)|dy+\int_{|x-y|\ge1}|x-y|^{-d}|\xi(y)|dy.
\end{align*}
Using the Hölder inequality just as in Step~1 for the first two terms, with $d<q\le\infty$, and using the Hölder inequality with exponents $(p/(p-1),p)$ for the last term, we obtain, for any $1\le p<\infty$,
\begin{align*}
|\nabla\triangle^{-1} w(x)|\lesssim~&(q'(1-d/q))^{-1/q'}\gamma^{1-d/q}\| w\|_{\Ld^q}+(\log(2/\gamma))^{(d-1)/d}\| w\|_{\Ld^{d}}+(d(p'-1))^{-1/p'}\|\xi\|_{\Ld^p},
\end{align*}
so that item~(ii) follows from the choice $\gamma^{1-d/q}=1\wedge(\| w\|_{\Ld^{d}}/\| w\|_{\Ld^q})$, observing that $(d(p'-1))^{-1/p'}\le p$.

\medskip
\noindent\step3 Proof of~(iii).

Given $0<\gamma\le1$, using the integration by parts
\begin{align*}
\int_{|x-y|<\gamma}\nabla^2g_d(x-y)dy=\int_{|x-y|=\gamma}n\otimes \nabla g_d(x-y)dy,
\end{align*}
we may decompose
\begin{align*}
|\nabla^2\triangle^{-1} w(x)|&\lesssim\bigg|\int_{|x-y|<\gamma}\frac{(x-y)^{\otimes2}}{|x-y|^{d+2}} w(y)dy\bigg|+\bigg|\int_{\gamma\le|x-y|<1}\frac{(x-y)^{\otimes2}}{|x-y|^{d+2}} w(y)dy\bigg|+\bigg|\int_{|x-y|\ge1}\frac{(x-y)^{\otimes2}}{|x-y|^{d+2}} w(y)dy\bigg|\\
&\lesssim\bigg|\int_{|x-y|<\gamma}\frac{(x-y)^{\otimes2}}{|x-y|^{d+2}}( w(x)- w(y))dy\bigg|+| w(x)|\bigg|\int_{|x-y|=\gamma}\frac{x-y}{|x-y|^{d}}dy\bigg|\\
&\hspace{3cm}+\bigg|\int_{\gamma\le|x-y|<1}\frac{(x-y)^{\otimes2}}{|x-y|^{d+2}} w(y)dy\bigg|+\bigg|\int_{|x-y|\ge1}\frac{(x-y)^{\otimes2}}{|x-y|^{d+2}} w(y)dy\bigg|.
\end{align*}
Let $0<s\le1$ and $1\le p<\infty$. Using the inequality $| w(x)- w(y)|\le|x-y|^s| w|_{C^s}$, and then applying the H\"older inequality with exponents $(1,\infty)$ for the first three terms, and $(p/(p-1),p)$ for the last one, we obtain after straightforward computations
\begin{align*}
|\nabla^2\triangle^{-1} w(x)|&\lesssim s^{-1}\gamma^s| w|_{C^s}+\| w\|_{\Ld^\infty}+|\log\gamma|\| w\|_{\Ld^\infty}+(d(p'-1))^{-1/p'}\| w\|_{\Ld^p}.
\end{align*}
Item~(iii) then follows for the choice $\gamma^s=\| w\|_{\Ld^\infty}/\| w\|_{C^s}\le1$.
\end{proof}

We turn to the proof of the potential estimates in Sobolev and Hölder-Zygmund spaces, stated in Lemma~\ref{lem:pottheoryCsHs}.

\begin{proof}[Proof of Lemma~\ref{lem:pottheoryCsHs}]
As item~(i) is obvious via Fourier transform, we focus on item~(ii).
Let $s\in\R$, let $\chi\in C^\infty_c(\R^d)$ be a fixed even function with $\chi=1$ in a neighborhood of the origin, and let $\chi(\nabla)$ denote the corresponding pseudo-differential operator. Applying~\cite[Proposition~2.78]{BCD-11} to the operator $(1-\chi(\nabla))\nabla\triangle^{-1}$, we find
\begin{align*}
\|\nabla\triangle^{-1}w\|_{C^s_*}\le\|(1-\chi(\nabla))\nabla\triangle^{-1}w\|_{C^s_*}+\|\chi(\nabla)\nabla\triangle^{-1}w\|_{C^s_*}\lesssim_s\|w\|_{C^{s-1}_*}+\|\chi(\nabla)\nabla\triangle^{-1}w\|_{C^s_*}.
\end{align*}
Let $k$ denote the smallest nonnegative integer $\ge s$. Noting that $\|v\|_{C^s_*}\lesssim\sum_{j=0}^k\|\nabla^jv\|_{\Ld^\infty}$ holds for all $v$, we deduce
\begin{align*}
\|\nabla\triangle^{-1}w\|_{C^s_*}\lesssim\|w\|_{C^{s-1}_*}+\sum_{j=0}^k\|\nabla^j\chi(\nabla)\nabla\triangle^{-1}w\|_{\Ld^\infty},
\end{align*}
and similarly
\begin{align*}
\|\nabla^2\triangle^{-1}w\|_{C^s_*}\lesssim\|w\|_{C^{s}_*}+\sum_{j=0}^k\|\nabla^j\chi(\nabla)\nabla^2\triangle^{-1}w\|_{\Ld^\infty}.
\end{align*}
Writing $\nabla^j\chi(\nabla)\nabla\triangle^{-1}w=\nabla^j\hat\chi\ast\nabla\triangle^{-1}w$, we find
\begin{align*}
\|\nabla^j\chi(\nabla)\nabla\triangle^{-1}w\|_{\Ld^\infty}\le\|\nabla^j\hat\chi\|_{\Ld^2}\|\nabla\triangle^{-1}w\|_{\Ld^2}=\|\nabla^j\hat\chi\|_{\Ld^2}\|w\|_{\dot H^{-1}},
\end{align*}
and the first two estimates in item~(ii) follow. Rather writing $\nabla^j\chi(\nabla)\nabla\triangle^{-1}w=\nabla\triangle^{-1}(\nabla^j\hat\chi\ast w)$, and using the estimate $|\nabla\triangle^{-1}v(x)|\lesssim\int|x-y|^{1-d}|v(y)|dy$ as in the proof of Lemma~\ref{lem:singint-1}, we find for all $1\le p<d$,
\begin{multline*}
\|\nabla^j\chi(\nabla)\nabla\triangle^{-1}w\|_{\Ld^\infty}\lesssim\sup_x\int_{|x-y|\le1}|x-y|^{1-d}|\nabla^j\hat\chi\ast w(y)|dy+\sup_x\int_{|x-y|>1}|x-y|^{1-d}|\nabla^j\hat\chi\ast w(y)|dy\\
\lesssim_p\|\nabla^j\hat\chi\ast w\|_{\Ld^p\cap\Ld^\infty}\le\|\nabla^j\hat\chi\|_{\Ld^1}\|w\|_{\Ld^p\cap\Ld^\infty},
\end{multline*}
and the third estimate in item~(ii) follows. The last estimate in item~(ii) is now easily obtained, arguing similarly as in the proof of Lemma~\ref{lem:singint-1}(iii).
\end{proof}

We turn to the proof of the 2D global elliptic regularity results stated in Lemma~\ref{lem:globellreg}.

\begin{proof}[Proof of Lemma~\ref{lem:globellreg}]
We split the proof into three steps, first proving~(i) as a consequence of Meyers' perturbative argument, then turning to the Sobolev regularity~(ii), and finally to the Schauder type estimate~(iii). The additional $\Ld^\infty$-estimate for $ v$ directly follows from item~(i) and the Sobolev embedding, while the corresponding estimate for $\nabla u$ follows from items~(i) and~(iii) by interpolation: for $2<p\le p_0$ and $s\in(0,1)$, we indeed find
\[\|\nabla u\|_{\Ld^\infty}\lesssim \|\nabla u\|_{\Ld^p}+|\nabla u|_{C^s}\le C_p\|f\|_{\Ld^{2p/(p+2)}}+C_s\|f\|_{\Ld^{2/(1-s)}}\le C_{p,s}\|f\|_{\Ld^1\cap\Ld^\infty}.\]
In the proof below, we use the notation $\lesssim$ for $\le$ up to a constant $C>0$ that depends only on an upper bound on $\Lambda$, and we add subscripts to indicate dependence on further parameters.

\medskip
\noindent\step1 Proof of~(i).
\nopagebreak

We start with the norm of $v$. By Meyers' perturbative argument~\cite{Meyers-63}, there exists some $1<r_0<2$ (depending only on $\Lambda$) such that $\|\nabla v\|_{\Ld^r}\lesssim \|g\|_{\Ld^r}$ holds for all $r_0\le r\le r_0'$, $\frac1{r_0}+\frac{1}{r_0'}=1$.
On the other hand, decomposing the equation for $v$ as
\[-\triangle v=\Div (g+(b-1)\nabla v),\]
we deduce from Riesz potential theory that for all $1<r<2$
\[\|v\|_{\Ld^{2r/(2-r)}}\lesssim_r\|g+(b-1)\nabla v\|_{\Ld^{r}}\lesssim\|g\|_{\Ld^{r}}+\|\nabla v\|_{\Ld^{r}},\]
and hence $\|v\|_{\Ld^{2r/(2-r)}}\lesssim_r\|g\|_{\Ld^r}$ for all $r_0\le r<2$, that is, $\|v\|_{\Ld^q}\lesssim_q\|g\|_{\Ld^{2q/(q+2)}}$ for all $\frac{2r_0}{2-r_0}\le q<\infty$.

We now turn to the norm of $\nabla u$. The proof follows from a suitable adaptation of Meyers' perturbative argument~\cite{Meyers-63}, again combined with Riesz potential theory. For the reader's convenience a complete proof is included.
First recall that the Calderón-Zygmund theory yields $\|\nabla^2\triangle w\|_{\Ld^p}\le K_p\|w\|_{\Ld^p}$ for all $1<p<\infty$ and all $w\in C^\infty_c(\R^2)$, where the constants $K_p$'s moreover satisfy $\limsup_{p\to2}K_p\le K_2$, while a simple energy estimate allows to choose $K_2=1$. Now rewriting the equation for $u$ as
\[-\triangle u=\frac2{\Lambda+1}f+\Div\bigg(\frac{2}{\Lambda+1}\Big(b-\frac{\Lambda+1}2\Big)\nabla u\bigg),\]
we deduce from Riesz potential theory and from the Calderón-Zygmund theory (applied to the first and to the second right-hand side term, respectively), for all $2<p<\infty$,
\begin{align*}
\|\nabla u\|_{\Ld^p}&\le \frac2{\Lambda+1}\|\nabla\triangle^{-1}f\|_{\Ld^p}+\bigg\|\nabla\triangle^{-1}\Div\bigg(\frac{2}{\Lambda+1}\Big(b-\frac{\Lambda+1}2\Big)\nabla u\bigg)\bigg\|_{\Ld^p}\\
&\le \frac{2C_p}{\Lambda+1}\|f\|_{\Ld^{2p/(p+2)}}+\frac{2K_p}{\Lambda+1}\Big\|\Big(b-\frac{\Lambda+1}2\Big)\nabla u\Big\|_{\Ld^p}\\
&\le \frac{2C_p}{\Lambda+1}\|f\|_{\Ld^{2p/(p+2)}}+\frac{K_p(\Lambda-1)}{\Lambda+1}\|\nabla u\|_{\Ld^p},
\end{align*}
where the last inequality follows from $\Id\le b\le\Lambda\Id$. Since we have $\frac{\Lambda-1}{\Lambda+1}<1$ and $\limsup_{p\to2}K_p\le K_2=1$, we may choose $ P_0>2$ close enough to $2$ such that $\frac{K_p(\Lambda-1)}{\Lambda+1}<1$ holds for all $2\le p\le p_0$. This allows to absorb the last right-hand side term, and to conclude $\|\nabla u\|_{\Ld^p}\lesssim_p \|f\|_{\Ld^{2p/(p+2)}}$ for all $2<p\le p_0$.

\medskip
\noindent\step2 Proof of~(ii).
\nopagebreak

We focus on the result for $u$, as the argument for $v$ is very similar. A simple energy estimate yields
\[\int|\nabla u|^2\le\int \nabla u\cdot b\nabla u=\int fu\le\|f\|_{\dot H^{-1}}\|\nabla u\|_{\Ld^2},\]
hence $\|\nabla u\|_{\Ld^2}\le\|f\|_{\dot H^{-1}}$, that is, (ii) with $s=0$. The result~(ii) for any integer $s\ge0$ is then deduced by induction, successively differentiating the equation.
It remains to consider the case of fractional values $s\ge0$. We only display the argument for $0<s<1$, while the other cases are similarly obtained after differentiation of the equation. Let $0<s<1$ be fixed. We use the following finite difference characterization of the fractional Sobolev space $H^s(\R^2)$: a function $w\in\Ld^2(\R^2)$ belongs to $H^s(\R^2)$, if and only if it satisfies $\|w-w(\cdot+h)\|_{\Ld^2}\le K|h|^s$ for all $h\in\R^2$, for some $K>0$, and we then have $\|w\|_{\dot H^s}\le K$. This characterization is easily checked, using e.g.\@ the identity $\|w-w(\cdot+h)\|_{\Ld^2}^2\simeq\int|1-e^{i\xi\cdot h}|^2|\hat w(\xi)|^2d\xi$, where $\hat w$ denotes the Fourier transform of $w$, and noting that $|1-e^{ia}|\le 2\wedge |a|$ holds for all $a\in\R$.
Now applying finite difference to the equation for $u$, we find for all $h\in\R^2$,
\[-\Div( b(\cdot+h)(\nabla u-\nabla u(\cdot+h)))=\Div( (b-b(\cdot+h))\nabla u)+f-f(\cdot+h),\]
and hence, testing against $u-u(\cdot+h)$,
\begin{align*}
\int |\nabla u-\nabla u(\cdot+h)|^2&\le-\int (\nabla u-\nabla u(\cdot+h))\cdot (b-b(\cdot+h))\nabla u+\int (u-u(\cdot+h))(f-f(\cdot+h))\\
&\le |h|^s|b|_{C^s}\|\nabla u\|_{\Ld^2}\|\nabla u-\nabla u(\cdot+h)\|_{\Ld^2}+\|f-f(\cdot+h)\|_{\dot H^{-1}}\|\nabla u-\nabla u(\cdot+h)\|_{\Ld^2},
\end{align*}
where we compute by means of Fourier transforms
\begin{align*}
\|f-f(\cdot+h)\|_{\dot H^{-1}}^2&\simeq\int |\xi|^{-2}|1-e^{i\xi\cdot h}|^2|\hat f(\xi)|^2d\xi\lesssim \int|\xi|^{-2}|\xi\cdot h|^{2s}|\hat f(\xi)|^2d\xi\lesssim |h|^{2s}\|f\|_{\dot H^{-1}\cap H^{s-1}}^2.
\end{align*}
Further combining this with the $\Ld^2$-estimate for $\nabla u$ proven at the beginning of this step, we conclude
\begin{align*}
\|\nabla u-\nabla u(\cdot+h)\|_{\Ld^2}&\lesssim |h|^s(|b|_{C^{s}}\|\nabla u\|_{\Ld^2}+\|f\|_{\dot H^{-1}\cap H^{s-1}})\lesssim |h|^s(1+|b|_{C^{s}})\|f\|_{\dot H^{-1}\cap H^{s-1}},
\end{align*}
and the result follows from the above stated characterization of $H^s(\R^2)$.

\medskip
\noindent\step3 Proof of~(iii).

We focus on the result for $u$, while that for $v$ is easily obtained as an adaptation of~\cite[Theorem~3.8]{Han-Lin-97}.
Let $x_0\in\R^2$ be fixed. The equation for $u$ may be rewritten as
\[-\Div(b(x_0)\nabla u)=f+\Div((b-b(x_0))\nabla u).\]
For all $r>0$, let $w_r\in u+H^1_0(B(x_0,r))$ be the unique solution of $-\Div(b(x_0)\nabla w_r)=0$ in $B(x_0,r)$. The difference $v_r:=u-w_r\in H^1_0(B(x_0,r))$ then satisfies in $B(x_0,r)$
\[-\Div(b(x_0)\nabla v_r)=f+\Div((b-b(x_0))\nabla u).\]
Testing this equation against $v_r$ itself, we obtain
\begin{align*}
\int|\nabla v_r|^2&\le \bigg|\int_{B(x_0,r)}fv_r\bigg|+\int_{B(x_0,r)}|b-b(x_0)||\nabla u||\nabla v_r|\le \bigg|\int_{B(x_0,r)}fv_r\bigg|+r^s|b|_{C^s}\|\nabla u\|_{\Ld^2(B(x_0,r))}\|\nabla v_r\|_{\Ld^2}.
\end{align*}
We estimate the first term as follows
\begin{align*}
\bigg|\int_{B(x_0,r)}fv_r\bigg|=\bigg|\int_{B(x_0,r)}\nabla v_r\cdot\nabla\triangle^{-1}(\mathds1_{B(x_0,r)}f)\bigg|&\le\|\nabla v_r\|_{\Ld^{p'}(B(x_0,r))}\|\nabla\triangle^{-1}(\mathds1_{B(x_0,r)}f)\|_{\Ld^p},
\end{align*}
and hence by Riesz potential theory, for all $2<p<\infty$,
\begin{align*}
\bigg|\int_{B(x_0,r)}fv_r\bigg|&\lesssim_p\|\nabla v_r\|_{\Ld^{p'}(B(x_0,r))}\|f\|_{\Ld^{2p/(p+2)}(B(x_0,r))}.
\end{align*}
The Hölder inequality then yields, choosing $q:=\frac2{1-s}>2$,
\begin{align*}
\bigg|\int_{B(x_0,r)}fv_r\bigg|&\lesssim_pr^{\frac2{p'}-1}\|\nabla v_r\|_{\Ld^{2}}~r^{1+\frac2p-\frac2q}\|f\|_{\Ld^{q}}=r^{2(1-\frac1q)}\|\nabla v_r\|_{\Ld^{2}}\|f\|_{\Ld^{q}}=r^{1+s}\|\nabla v_r\|_{\Ld^{2}}\|f\|_{\Ld^{2/(1-s)}}.
\end{align*}
Combining the above estimates, we deduce
\begin{align*}
\int|\nabla v_r|^2\lesssim r^{2(1+s)}\|f\|_{\Ld^{2/(1-s)}}^2+r^{2s}|b|_{C^s}^2\|\nabla u\|_{\Ld^2(B(x_0,r))}^2.
\end{align*}
We are now in position to conclude exactly as in the classical proof of the Schauder estimates (see e.g.~\cite[Theorem~3.13]{Han-Lin-97}).
\end{proof}

We turn to the proof of Lemma~\ref{lem:reconstr}, concerning the reconstruction of $v$ from the knowledge of $\curl v$ and $\Div(av)$.

\begin{proof}[Proof of Lemma~\ref{lem:reconstr}]
We split the proof into two steps.

\medskip
\noindent\step1 Uniqueness.

We prove that at most one function $\delta v\in\Ld^2(\R^2)^2$ can be associated with a given couple $(\delta\omega,\delta\zeta)$. For that purpose, we assume that $\delta v\in \Ld^2(\R^2)^2$ satisfies $\curl\delta v=0$ and $\Div(a\delta v)=0$, and we deduce $\delta v=0$. By the Hodge decomposition in $\Ld^2(\R^2)^2$, there exist functions $\phi,\psi\in H^1_\loc(\R^2)$ such that $a\delta v=\nabla\phi+\nabla^\bot\psi$ with $\nabla\phi,\nabla\psi\in\Ld^2(\R^2)^2$. Now note that $\triangle \phi=\Div(a\delta v)=0$ and $\Div (a^{-1}\nabla\psi)+\curl(a^{-1}\nabla \phi)=\curl\delta v=0$, which implies $\nabla\phi=0$ and $\nabla\psi=0$, hence $\delta v=0$.

\medskip
\noindent\step2 Existence.
\nopagebreak

Given $\delta\omega,\delta\zeta\in\dot H^{-1}(\R^2)$,
we observe that $\nabla(\Div a^{-1}\nabla)^{-1}\delta\omega$ and $\nabla(\Div a\nabla)^{-1}\delta\zeta$ are well-defined in $\Ld^2(\R^2)^2$. The vector field
\begin{align*}
\delta v:=a^{-1}\nabla^\bot(\Div a^{-1}\nabla)^{-1}\delta\omega+\nabla(\Div a\nabla)^{-1}\delta\zeta
\end{align*}
is thus well-defined in $\Ld^2(\R^2)^2$, and trivially satisfies $\curl \delta v=\delta\omega$, $\Div(a\delta v)=\delta\zeta$. The additional estimate follows from Lemmas~\ref{lem:katoponce-1} and~\ref{lem:globellreg}(ii).
\end{proof}

Finally, we turn to the control on the pressure stated in Lemma~\ref{lem:pressure}.

\begin{proof}[Proof of Lemma~\ref{lem:pressure}]
In this proof, we use the notation $\lesssim$ for $\le$ up to a constant $C$ depending only on an upper bound on $\|(h,\Psi,\bar  v^\circ)\|_{\Ld^\infty}$.
Let $2<p_0,q_0\lesssim1$ and $r_0=p_0$ be as in Lemma~\ref{lem:globellreg}(i) (with $b$ replaced by $a$ or $a^{-1}$), and note that $q_0$ can be chosen large enough such that $\frac1{p_0}+\frac1{q_0}\le\frac12$. Assume that $\omega\in \Ld^\infty_\loc([0,T);\Pc\cap\Ld^{q_0}(\R^2))$ holds for this choice of the exponent $q_0$. By Lemma~\ref{lem:globellreg}(i), the function
\[P:=(-\Div a\nabla)^{-1}\Div(a\omega(-\alpha(\Psi+v)+\beta(\Psi+v)^\bot))\]
is well-defined in $\Ld^\infty_\loc([0,T);\Ld^{q_0}(\R^2))$ and satisfies for all $t\ge0$,
\begin{align*}
\|P^t\|_{\Ld^{q_0}}&\lesssim \|a\omega^t(-\alpha(\Psi+ v^t)+\beta(\Psi+ v^t)^\bot)\|_{\Ld^{2q_0/(2+q_0)}}\\
&\lesssim\|\Psi+\bar  v^\circ\|_{\Ld^\infty}\|\omega^t\|_{\Ld^{2q_0/(2+q_0)}}+\| v^t-\bar  v^\circ\|_{\Ld^{2}}\|\omega^t\|_{\Ld^{q_0}}\\
&\lesssim (1+\| v^t-\bar  v^\circ\|_{\Ld^2})\|\omega^t\|_{\Ld^1\cap\Ld^{q_0}}.
\end{align*}
Now note that the following Helmholtz-Leray type identity follows from the proof of Lemma~\ref{lem:reconstr}: for any vector field $F\in C^\infty_c(\R^2)^2$,
\begin{align}\label{eq:Helmholtz}
F=a^{-1}\nabla^\bot(\Div a^{-1}\nabla)^{-1}\curl F+\nabla(\Div a\nabla)^{-1}\Div (aF).
\end{align}
This implies in particular, for the choice $F=\omega\big(-\alpha(\Psi+ v)+\beta(\Psi+ v)^\bot\big)$,
\begin{eqnarray}
\lefteqn{a^{-1}\nabla^\bot(\Div a^{-1}\nabla)^{-1}\Div\big(\omega(\alpha(\Psi+ v)^\bot+\beta(\Psi+ v))\big)}\nonumber\\
&=&a^{-1}\nabla^\bot(\Div a^{-1}\nabla)^{-1}\curl\big(\omega(-\alpha(\Psi+ v)+\beta(\Psi+ v)^\bot)\big)\nonumber\\
&=&\omega\big(-\alpha(\Psi+ v)+\beta(\Psi+ v)^\bot\big)+\nabla P.\label{eq:linkpressure}
\end{eqnarray}
For $\phi\in C^\infty_c([0,T)\times\R^2)^2$, it follows from Lemma~\ref{lem:globellreg}(i) that $(\Div a^{-1}\nabla)^{-1}\curl (a^{-1}\phi)\in C^\infty_c([0,T);\Ld^{q_0}(\R^2))$ and that $\nabla(\Div a^{-1}\nabla)^{-1}\curl (a^{-1}\phi)\in C^\infty_c([0,T);\Ld^2\cap\Ld^{p_0}(\R^2))$. With the choice $\frac1{p_0}+\frac1{q_0}\le\frac12$, the $\Ld^{q_0}$-regularity of $\omega$ then allows to test the weak formulation of~\eqref{eq:limeqn1VF} (which defines weak solutions of~\eqref{eq:limeqn1}, cf.\@ Definition~\ref{defin:sol}(b)) against $(\Div a^{-1}\nabla)^{-1}\curl (a^{-1}\phi)$, to the effect of
\begin{multline*}
\int\omega^\circ(\Div a^{-1}\nabla)^{-1}\curl (a^{-1}\phi(0,\cdot))+\iint\omega(\Div a^{-1}\nabla)^{-1}\curl (a^{-1}\partial_t\phi)\\
=\iint \omega(\alpha(\Psi+ v)^\bot+\beta(\Psi+ v))\cdot\nabla(\Div a^{-1}\nabla)^{-1}\curl (a^{-1}\phi).
\end{multline*}
As by~\eqref{eq:Helmholtz} the constraint $\Div(a v)=0$ implies $ v=a^{-1}\nabla^\bot(\Div a^{-1}\nabla)^{-1}\omega$ and $ v^\circ=a^{-1}\nabla^\bot(\Div a^{-1}\nabla)^{-1}\omega^\circ$, and as by definition $\omega\in\Ld^\infty_\loc([0,T);\Ld^1\cap\Ld^{2}(\R^2))$, Lemma~\ref{lem:globellreg}(i) implies $ v\in\Ld^\infty_\loc([0,T);\Ld^{p_0}(\R^2)^2)$. We may then integrate by parts in the weak formulation above, which yields
\begin{align*}
\int \phi(0,\cdot)\cdot  v^\circ+\iint \partial_t\phi\cdot  v=-\iint a^{-1}\phi\cdot\nabla^\bot(\Div a^{-1}\nabla)^{-1}\Div(\omega(\alpha(\Psi+ v)^\bot+\beta(\Psi+ v))),
\end{align*}
and the result now directly follows from the decomposition~\eqref{eq:linkpressure}.
\end{proof}

\subsection*{Acknowledgements}
The work of the author is supported by F.R.S.-FNRS (Belgian National Fund for Scientific Research) through a Research Fellowship.
The author would like to thank his PhD advisor Sylvia Serfaty as well as two anonymous referees for valuable comments and suggestions on this work.

\bigskip
\bibliographystyle{plain}
\bibliography{biblio}

\bigskip
{\small (Mitia Duerinckx) {\sc Université Libre de Bruxelles (ULB), Brussels, Belgium, \& Laboratoire Jacques-Louis-Lions, Université Pierre et Marie Curie (UPMC), Paris, France}

{\it E-mail address:} mduerinc@ulb.ac.be
}

\medskip
{\small (Julian Fischer) {\sc Institute of Science and Technology Austria (IST Austria), Klosterneuburg, Austria}

{\it E-mail address:} julian.fischer@ist.ac.at
}

\end{document}